\documentclass[reqno]{amsart}
\usepackage{latexsym,amsfonts,amssymb,epsfig}
\usepackage{hyperref} 
\usepackage{color}

\DeclareMathOperator{\ch}{char}

\DeclareMathOperator{\ad}{ad}
\DeclareMathOperator{\Der}{Der}

\DeclareMathOperator{\Lie}{Lie}
\DeclareMathOperator{\Alg}{Alg}

\DeclareMathOperator{\End}{End}

\DeclareMathOperator{\cwt}{\mathsf{wt}}    
\DeclareMathOperator{\wt}{wt}    
\DeclareMathOperator{\Gr}{Gr}     
\DeclareMathOperator{\gen}{gen}     
\DeclareMathOperator{\var}{var}     

\DeclareMathOperator{\depl}{depl} 
\renewcommand {\limsup}{\operatorname* {\overline{lim}}}
\renewcommand {\liminf}{\operatorname* {\underline{lim}}}

\DeclareMathOperator{\GKdim}{GKdim}
\DeclareMathOperator{\LGKdim}{\underline{GKdim}}
\DeclareMathOperator{\Dim}{Dim}
\DeclareMathOperator{\LDim}{\underline{Dim}}
\DeclareMathOperator{\Sdim}{Sdim}
\DeclareMathOperator{\LSdim}{\underline{Sdim}}

\DeclareMathOperator{\Ldim}{Ldim}
\DeclareMathOperator{\LLdim}{\underline{Ldim}}

\newcommand\dd{\partial}

\renewcommand{\a}{\alpha}
\renewcommand{\b}{\beta}
\newcommand{\g}{\gamma}

\newcommand{\Z}{\mathbb Z}            
\newcommand{\R}{\mathbb R}            
\newcommand{\N}{\mathbb N}            
\newcommand{\C}{\mathbb C}            
\newcommand\NO{\mathbb N_0}           

\newcommand{\LL}{\mathbf L}        
\newcommand{\QQ}{\mathbf Q}        
\newcommand{\RR}{\mathbf R}        
\newcommand{\TT}{\mathbf T}        

\renewcommand{\AA}{\mathbf A}      
\newcommand{\uu}{\mathbf u}         
\newcommand{\WW}{\mathbf W}         
\newcommand{\Sym}{\mathrm {Sym}}      
\newcommand{\NN}{\mathcal N}      
\newcommand{\AAA}{\mathcal A}      

\newtheorem{Theorem}{Theorem}[section]
\newtheorem{Corollary}[Theorem]{Corollary}
\newtheorem{Lemma}[Theorem]{Lemma}
\theoremstyle{remark}
\newtheorem{Remark}{Remark}
\theoremstyle{Example}
\newtheorem{Example}{Example}
\theoremstyle{Conjecture}

\begin{document}
\title{Nil restricted Lie algebras of oscillating intermediate growth}
\author{Victor Petrogradsky}
\address{Department of Mathematics, University of Brasilia, 70910-900 Brasilia DF, Brazil}
\email{petrogradsky@rambler.ru}
\thanks{The author was partially supported by grants CNPq~309542/2016-2, DPI-UnB~04/2019  
\\ \vspace{-0.3cm}
}


\begin{abstract}
The Grigorchuk and Gupta-Sidki groups are natural examples of self-similar finitely generated periodic groups.
The author constructed their analogue in case of restricted Lie algebras of characteristic 2~\cite{Pe06},
Shestakov and Zelmanov extended this construction to an arbitrary positive characteristic~\cite{ShZe08}.
It is known that the famous construction of Golod yields finitely generated associative nil-algebras of exponential growth.
Recent extensions of that approach allowed to construct finitely generated associative nil-algebras
of polynomial and intermediate growth~\cite{LenSmo07,BellYoung11,LenSmoYoung12,Smo14}.
Another motivation of the paper is a construction of  groups of oscillating growth by Kassabov and Pak~\cite{KasPak13}.
\par
For any prime $p$ we construct a family of 3-generated restricted Lie algebras of intermediate oscillating growth.
We call them {\it Phoenix algebras} because,  for infinitely many periods of time, the algebra is "almost dying"
by having a {\it quasi-linear} growth, namely the lower Gelfand-Kirillov dimension is one, more precisely,
the growth is  of type $n \big(\ln^{(q)} \!n\big )^{\kappa}$.
On the other hand, for infinitely many $n$ the growth has a rather fast intermediate behaviour of type
$\exp( n/ (\ln n)^{\lambda})$, for such periods the algebra is "resuscitating".
Moreover, the growth function is oscillating between these two types of behaviour.
These restricted Lie algebras have a nil $p$-mapping.
\end{abstract}
\maketitle

\section{Introduction}
Different versions of Burnside Problem ask what one can say about finitely generated periodic groups under additional assumptions.
For associative algebras, Kurosh type problems ask similar questions
about properties of finitely generated nil (more, generally algebraic) algebras.
In case of finitely generated Lie algebras, the periodicity is replaced by the condition that the adjoint mapping is nil.
In particular, for Lie $p$-algebras one assumes that the $p$-mapping is nil.
One of recent important directions in these areas is to study the growth of finitely generated periodic groups and nil algebras.
The conditions that describe growth functions of all finitely generated associative algebras and monoids are
recently given in~\cite{BellZel19}.
The present paper is devoted to the construction of finitely generated nil restricted Lie algebras
with extreme growth properties, described in our main Theorem~\ref{T_main}.

\subsection{Burnside Problem and Golod-Shafarevich algebras and groups}
The General Burnside Problem asks whether a finitely generated periodic group is finite.
The first negative answer was given by Golod and Shafarevich,
they proved that there exist finitely generated infinite $p$-groups for each prime $p$~\cite{Golod64}.
As an important instrument, they first construct finitely generated
infinite dimensional associative nil-algebras~\cite{Golod64}.
Using this construction, there are also examples of infinite dimensional 3-generated Lie algebras $L$
such that $(\ad x)^{n(x,y)}(y)=0$, for all $x,y\in L$, the field being arbitrary~\cite{Golod69}.
Similarly, one easily obtains infinite dimensional finitely generated restricted Lie algebras $L$ with  a nil $p$-mapping.
This gives a negative answer to the question of Jacobson whether
a finitely generated restricted Lie algebra $L$ is finite dimensional provided that
each element $x\in L$ is algebraic, i.e. satisfies some $p$-polynomial $f_{p,x}(x)=0$
(\cite[Ch.~5, ex.~17]{JacLie}). 

It is known that the construction of Golod yields associative nil-algebras of exponential growth.
Using specially chosen relations, Lenagan and Smoktunowicz constructed associative nil-algebras of polynomial growth~\cite{LenSmo07},
there are more constructions including associative nil-algebras of intermediate growth~\cite{BellYoung11,LenSmoYoung12,Smo14}.
On further developments concerning Golod-Shafarevich algebras and groups see~\cite{Voden09,Ershov12}.
Embeddings of associative algebras of countable dimension of
locally subexponential growth into finitely generated algebras of subexponential growth are obtained in~\cite{AlAlJainZel18}.

A close by spirit but different construction was motivated by respective group-theoretic results.
A restricted Lie algebra $G$ is called {\it large} if there is a subalgebra  $H\subset G$ of finite codimension
such that $H$ admits a surjective homomorphism on a nonabelian free restricted Lie algebra.
Let $K$ be a perfect at most countable field of positive characteristic.
Then there exist infinite-dimensional finitely generated nil restricted Lie algebras over $K$ that
are residually finite dimensional and direct limits of large restricted Lie algebras~\cite{BaOl07}.

\subsection{Grigorchuk and Gupta-Sidki groups and their growth}
The construction of Golod is rather undirect, Grigorchuk gave a direct and elegant construction of
an infinite 2-group generated by three elements of order 2~\cite{Grigorchuk80}.
Originally, this group was defined as a group of transformations of the interval $[0,1]$ from which
rational points of the form $\{k/2^n\mid  0\le k\le 2^n,\ n\ge 0\}$ are removed.
For each prime $p\ge 3$, Gupta and Sidki gave a direct construction of an infinite $p$-group
on two generators, each of order $p$~\cite{GuptaSidki83}.
This group was constructed as a subgroup of an automorphism group of an infinite regular tree of degree $p$.

The Grigorchuk and Gupta-Sidki groups are counterexamples to the General Burnside Problem.
Moreover, they gave answers to important problems in group theory.
So, the Grigorchuk group and its further generalizations
are first examples of groups of intermediate growth~\cite{Grigorchuk84}, thus answering
in negative to a conjecture of Milnor that groups of intermediate growth do not exist.
The construction of Gupta-Sidki also yields groups of subexponential growth~\cite{FabGup85}.
The Grigorchuk and Gupta-Sidki groups are {\it self-similar}.
Now self-similar, and so called {\it branch groups}, form a well-established area in group theory~\cite{Grigorchuk00horizons,Nekr05}.
There are also constructions of self-similar associative  algebras~\cite{Bartholdi06,Sidki09,PeSh13ass}.

Recently, the intermediate growth of the Grigorchuk group $G$ was determined
by Erschler and Zheng~\cite{ErshlerZheng20}:
$$\gamma_G(n)= \exp\big(n^{\a_0+o(1)}\big),\qquad n\to\infty;\qquad  \a_0=\log_{\lambda_0} 2\approx 0,7674; $$
where $\lambda_0$ is the positive root of the polynomial $x^3-x^2-2x-4=0$ 
(the lower bound was proved by L.~Bartholdi earlier~\cite{Bartholdi98}).
The Grigorchuk group has the lowest known intermediate growth.
Moreover , Grigorchuk conjectured that there are no groups with growth between the polynomial growth and $\exp(\sqrt n)$,
see more details in review~\cite{Gri14}.

\subsection{Groups of oscillating growth}
The present research is also motivated by the following result.
Our goal is to establish its analogue in case of restricted Lie algebras.
\begin{Theorem}[Kassabov, Pak~\cite{KasPak13}]\label{TKassabovPak}
Let $f_1(n), f_2(n), g_1(n), g_2(n):\N\to\N$ be monotone increasing subexponential functions:
$$ \gamma_G(n)\le g_2(n)< g_1(n)< f_2(n)< f_1(n), $$
where $\gamma_G(n)$ is the growth function of the Grigorchuk group $G$,
and they also satisfy some more technical assumptions.
Then there exists a finitely generated group $\Gamma$ which growth function $\gamma_\Gamma(n)$
takes values in the intervals $[g_2(n),g_1(n)]$ and $[f_2(n),f_1(n)]$, $n\in\N$,
infinitely many times.
\end{Theorem}
The upper bounds $f_1,f_2$ can be very close to the exponential function~\cite{KasPak13}.
Remark that nothing is said about the periodicity of $\Gamma$.

The result above on the growth of the Grigorchuk group has a stronger version.
Namely, for any $\a,\b$ such that $\a_0\le \a\le\b\le 1$, where $\a_0$ as above,
there exists a sequence $\omega\in \{0,1,2\}^\infty$ and the respective Grigorchuk group $G_\omega$ (suggested in~\cite{Grigorchuk84})
has an oscillating intermediate growth as follows~\cite{ErshlerZheng20}:
$$
\liminf_{n\to\infty} \frac {\ln\ln \gamma_{G_\omega}(n)}{\ln n}=\a,\qquad
\limsup_{n\to\infty} \frac {\ln\ln \gamma_{G_\omega}(n)}{\ln n}=\b.
$$
We introduce clover restricted Lie algebras studied in details in~\cite{Pe20clover}, they are analogues of groups $G_\omega$
in terms of their properties and because of their importance for our construction of Lie algebras of oscillating growth.

\subsection{Self-similar nil restricted Lie algebras, Fibonacci Lie algebra}
Unlike associative algebras, for restricted Lie algebras,
natural analogues of the Grigorchuk and Gupta-Sidki groups are known.
Namely, over a field of characteristic 2,
the author constructed an infinite dimensional restricted Lie algebra $\LL$
generated by two elements, called a {\it Fibonacci restricted Lie algebra}~\cite{Pe06}.
Let $\ch K=p=2$ and $R=K[t_i| i\ge 0 ]/(t_i^p| i\ge 0)$ a truncated polynomial ring.
Put $\dd_i=\frac {\dd}{\partial t_i}$, $i\ge 0$.
Define the following two derivations of $R$:
\begin{align*}
v_1 & =\dd_1+t_0(\dd_2+t_1(\dd_3+t_2(\dd_4+t_3(\dd_5+t_4(\dd_6+\cdots )))));\\
v_2 & =\qquad\quad\;\,
\dd_2+t_1(\dd_3+t_2(\dd_4+t_3(\dd_5+t_4(\dd_6+\cdots )))).
\end{align*}
These two derivations generate
a restricted Lie algebra $\LL=\Lie_p(v_1,v_2)\subset\Der R$ and an associative algebra $\AA=\Alg(v_1,v_2)\subset \End R$.
The Fibonacci restricted Lie algebra $\LL$ has a slow polynomial growth
with Gelfand-Kirillov dimension $\GKdim \LL=\log_{(\sqrt 5+1)/2} 2\approx 1.44$~\cite{Pe06}.
Further properties of the Fibonacci restricted Lie algebra 
 are studied in~\cite{PeSh09,PeSh13fib}.

Probably, the most interesting property of $\LL$ is that it has a nil $p$-mapping~\cite{Pe06},
which is an analog of the periodicity of the Grigorchuk and Gupta-Sidki groups.
We do not know whether the associative hull $\AA$ is a nil-algebra.
We have a weaker statement.
The algebras $\LL$, $\AA$, and the augmentation ideal
of the restricted enveloping algebra $\uu=\omega u(\LL)$ are direct sums of two locally nilpotent subalgebras~\cite{PeSh09}.
In case of arbitrary prime characteristic,
Shestakov and Zelmanov suggested an example of a finitely generated restricted Lie algebra with a nil $p$-mapping~\cite{ShZe08}.
An example of a $p$-generated  nil restricted Lie algebra $L$, characteristic $p$ being arbitrary, was studied in~\cite{PeShZe10}.
These infinite dimensional restricted Lie algebras yield
different decompositions into a direct sum of two locally nilpotent subalgebras~\cite{PeShZe10}.

Observe that only the original example has a clear monomial basis~\cite{Pe06,PeSh09}.
In other examples, elements of a Lie algebra are  linear combinations of monomials,
to work with such linear combinations is sometimes an essential technical difficulty, see e.g.~\cite{ShZe08,PeShZe10}.
A family of nil restricted Lie algebras of slow growth having good monomial bases
is constructed in~\cite{Pe17},
these algebras are close relatives of a two-generated Lie superalgebra of~\cite{Pe16}.

\subsection{Lie algebras in characteristic zero}
Since the Grigorchuk group is of finite width,
a right analogue of it should be a Lie algebra of finite width having $\ad$-nil elements, in the next result
the components are of bounded dimension and consist of $\ad$-nil elements.
Informally speaking, there are no "natural analogues" of the Grigorchuk and Gupta-Sidki groups
in the world of Lie algebras of characteristic zero,
strictly in terms of the following result.
\begin{Theorem}[{Martinez and Zelmanov~\cite{MaZe99}}]
\label{TMarZel}
Let $L=\oplus_{\a\in\Gamma}L_\alpha$ be a finitely generated Lie algebra over a field $K$,
$\ch K=0$, graded by an abelian group $\Gamma$. Assume that
\begin{enumerate}
\item
there exists $d>0$ such that $\dim_K L_\alpha \le d $, $\alpha\in\Gamma$,
\item
every homogeneous element $a\in L_\a$, $\a\in\Gamma$, is ad-nilpotent.
\end{enumerate}
Then $L$ is finite dimensional.
\end{Theorem}

\subsection{Fractal nil graded Lie superalgebras}\label{SSsuper}
In the class of {\it Lie  superalgebras} of an {\it arbitrary characteristic},
the author suggested analogues of the Grigorchuk and Gupta-Sidki groups~\cite{Pe16}.
Namely, two Lie superalgebras $\RR$, $\QQ$ are constructed (Example~\ref{E1} and  Example~\ref{Example_Q}).
These two examples have clear monomial bases.
They are of slow polynomial growth:
$\GKdim \RR=\log_34\approx 1.26$ and $\GKdim \QQ=\log_38\approx 1.89$.
Also, $\ad a$ is nilpotent,
$a$ being an even or odd element with respect to the $\Z_2$-gradings of these Lie superalgebras.
We get an analogue of the periodicity of the Grigorchuk and Gupta-Sidki groups.
The Lie superalgebra $\RR$ is $\Z^2$-graded, while $\QQ$ has a natural fine $\Z^3$-grading with at most one-dimensional components.
In particular, $\QQ$ is a nil finely graded Lie superalgebra, which shows
that an extension of Theorem~\ref{TMarZel}\ 
for the class of Lie {\it super}algebras of characteristic zero is not valid.
Also, $\QQ$ has a $\Z^2$-grading which yields a continuum of decompositions
into sums of two locally nilpotent subalgebras $\QQ=\QQ_+\oplus\QQ_-$.
Both Lie superalgebras are {\it self-similar},
they also contain infinitely many copies of itself, we call them {\it fractal} due to the last property.

We construct a more "handy" 2-generated fractal Lie superalgebra $\mathbf{R}$
(the same notation as above but this is a different algebra) over an arbitrary field~\cite{PeOtto}.
This Lie superalgebra $\RR$ is $\Z^2$-graded by multidegree in the generators and the
$\Z^2$-components are at most one-dimensional.
As an analogue of periodicity,
we establish that homogeneous elements of the $\Z_2$-grading $\mathbf{R}=\mathbf{R}_{\bar 0}\oplus\mathbf{R}_{\bar 1}$ are $\ad$-nilpotent.
In case of $\mathbb{N}$-graded algebras, a close analogue to being simple is being just infinite.
Unlike previous examples of Lie superalgebras~\cite{Pe16}, we are able to prove that $\mathbf{R}$ is just infinite.
This example is close to the smallest possible one, because $\mathbf{R}$ has a linear growth
with the growth function $\gamma_\mathbf{R}(m)\approx 3m$, as $m\to\infty$.
Moreover, its degree $\mathbb{N}$-grading is of finite width 4 ($\ch K\ne 2$).
In case $\ch K=2$, we obtain a Lie algebra of width 2 that is not thin.

We also construct a  just infinite fractal 3-generated Lie superalgebra ${\mathbf Q}$ over an arbitrary field,
which gives rise to an associative hull $\mathbf A$, a Poisson superalgebra ${\mathbf P}$, and
two Jordan superalgebras ${\mathbf J}$ and ${\mathbf K}$, the latter can be also considered as
analogues of the Grigorchuk and Gupta-Sidki groups in respective classes of algebras~\cite{PeSh18FracPJ}.

\section{Basic notions: restricted Lie algebras, Growth}\label{Sdef}

As a rule, $K$ is an arbitrary field of positive characteristic $p$,
$\langle S\rangle_K$ denotes a linear span of a subset $S$ in a $K$-vector space.
Let $L$ be a Lie algebra, then $U(L)$ denotes the universal enveloping algebra.
Long commutators are {\it right-normed}: $[x,y,z]:=[x,[y,z]]$.
We use a standard notation $\ad x(y)=[x,y]$, where $x,y\in L$.
Also, we use notation $[x^k,y]:=(\ad x)^k (y)$, where $k\ge 1$, $x,y\in L$; in case $k=p^l$, we have also $[x^{p^l},y]=[x^{[p^l]},y]$,
in terms of the $p$-mapping (see below).

\subsection{Restricted Lie algebras}
Let $L$ be a Lie algebra over a field $K$ of characteristic $p>0$.
Then $L$ is called a
\textit{restricted Lie algebra} (or \textit{Lie $p$-algebra}),
if it is additionally supplied with a unary operation
 $x\mapsto x^{[p]}$, $x\in L$, that satisfies the following
 axioms~\cite{JacLie,Strade1,StrFar,BMPZ}:
\begin{itemize}
\item $(\lambda x)^{[p]}=\lambda^px^{[p]}$, for $\lambda\in K$, $x\in L$;
\item $\ad(x^{[p]})=(\ad x)^p$, $x\in L$;
\item $(x+y)^{[p]}=x^{[p]}+y^{[p]}+\sum_{i=1}^{p-1}s_i(x,y)$, for all $x,y\in L$,
where $i s_i(x,y)$~is the coefficient of $t^{i-1}$ in the polynomial
$\operatorname{ad}(tx+y)^{p-1}(x)\in L[t]$.
\end{itemize}
This notion is motivated by the following construction.
Let $A$ be an associative algebra over a field ~$K$.
The vector space $A$ is supplied with a new product $[x,y]=xy-yx$, $x,y\in A$,
one obtains a Lie algebra denoted by $A^{(-)}$.
In case $\operatorname{char}K=p>0$,
the mapping  $x\mapsto x^p$, $x\in A^{(-)}$, satisfies three axioms above.

Suppose that $L$~ is a restricted Lie algebra.
Let $J$~be an ideal of the universal enveloping algebra~$U(L)$ generated by $\{x^{[p]}-x^p\mid x\in L\}$.
Then $u(L)=U(L)/J$ is called a \textit{restricted enveloping algebra}.
In this algebra, the formal operation $x^{[p]}$ coincides with the $p$th power~$x^p$ for any $x\in L$.
One has an analogue of Poincare-Birkhoff-Witt's theorem
yielding a basis of the restricted enveloping algebra~\cite[p.~213]{JacLie}.
We shall use the following version of the formula above:
\begin{equation}\label{power_P}
(x+y)^{[p]}=x^{[p]}+y^{[p]}+(\ad x)^{p-1}(y)+
\sum_{i=1}^{p-2}s_i(x,y),\qquad x,y\in L,
\end{equation}
where $s_i(x,y)$ consists of commutators containing $i$ letters $x$ and $p-i$ letters $y$.

\subsection{Growth}
Let $A$  be an associative (or Lie) algebra  generated by a finite set $X$.
Denote  by $A^{(X,n)}$ the subspace of $A$ spanned by all  monomials  in $X$ of length not  exceeding  $n$, $n\ge 0$.
If $A$ is a restricted Lie algebra, we define
$A^{(X,n)}=\langle\, [x_{i_1},\dots,x_{i_s}]^{p^k}\mid x_{i_j}\in X,\, sp^k\le n\rangle_K$~\cite{Pape01}.
One obtains a {\em growth function}:
$$
\gamma_A(n)=\gamma_A(X,n):=\dim_KA^{(X,n)},\quad n\ge 0.
$$
Clearly, the growth function depends on the choice of the generating set $X$.
Let $f,g:\N\to\R^+$ be increasing functions.
Write $f(n)\preccurlyeq g(n)$ if and only if there exist positive constants $N,C$ such that $f(n)\le g(Cn)$ for all $n\ge N$.
Introduce equivalence $f(n)\sim g(n)$ if and only if  $f(n)\preccurlyeq g(n)$ and $g(n)\preccurlyeq f(n)$.
Different generating sets of an algebra yield equivalent growth functions~\cite{KraLen}.

It is well known that the exponential growth is the highest possible growth for finitely generated Lie and associative algebras.
A growth function $\gamma_A(n)$ is compared with polynomial functions $n^k$, $k\in\R^+$,
by computing the {\em upper and lower Gelfand-Kirillov dimensions}~\cite{KraLen}:
\begin{align*}
\GKdim A&:=\limsup_{n\to\infty} \frac{\ln\gamma_A(n)}{\ln n}=\inf\{\a>0\mid \gamma_A(n)\preccurlyeq n^\a\} ;\\
\LGKdim A&:=\liminf_{n\to\infty}\,  \frac{\ln\gamma_A(n)}{\ln n}=\sup\{\a>0\mid \gamma_A(n)\succcurlyeq n^\a\}.
\end{align*}
Denote $\ln^{(q)}(x):=\underbrace{\ln(\cdots\ln}_{q\text{ times}}(x)\cdots)$ and
$\exp^{(q)}(x):=\underbrace{\exp(\cdots\exp}_{q\text{ times}}(x)\cdots)$ for all $q\in\N$.
In this paper, we study algebras of {\it quasi-linear} growth, growth functions of these algebras $A$ behave as
$m\exp \big((\ln m)^{\beta}\big)$, $\beta\in(0,1)$, and even slower, like
$m (\ln^{(q)} m)^{\beta}$, where $q\in \N$, $\beta\in \R^+$.
Clearly, $\GKdim A=\LGKdim A=1$.
In order to specify parameters $q,\beta$ define numbers:
\begin{align*}
\Ldim^0 A=& \inf\{\beta\in(0,1) \mid \gamma_A(n) \preccurlyeq m \exp \big((\ln m)^{\beta}\big)\};\\
\LLdim^0 A=& \sup\{\beta\in(0,1) \mid \gamma_A(n) \succcurlyeq m \exp \big((\ln m)^{\beta}\big)\};\\
\Ldim^q A=& \inf\{\beta\in\R^+ \mid \gamma_A(n) \preccurlyeq m (\ln^{(q)} m)^{\beta}\},\qquad q\in\N;\\
\LLdim^q A=& \sup\{\beta\in\R^+ \mid \gamma_A(n) \succcurlyeq m (\ln^{(q)} m)^{\beta}\},\qquad q\in\N.
\end{align*}
One checks that these numbers are invariants not depending on a generating set. Remark that notations are different from~\cite{Pe17}.

Assume that generators $X=\{x_1,\dots,x_k\}$ are assigned positive weights $\wt(x_i)=\lambda_i$, $i=1,\dots,k$.
Define a {\it weight growth function}:
$$
\tilde \gamma_A(n)=\dim_K\langle x_{i_1}\cdots x_{i_m}\mid \wt(x_{i_1})+\cdots+\wt(x_{i_m})\le n,\
          x_{i_j}\in X\rangle_K,\quad n\ge 0.
$$
Set $C_1=\min\{\lambda_i\mid i=1,\dots,k \}$, $C_2=\max\{\lambda_i\mid i=1,\dots,k \}$,
then $\tilde\gamma_A(C_1 n) \le \gamma_A(n)\le \tilde\gamma_A(C_2 n)$ for $n\ge 1$.
Thus, we obtain an equivalent growth function $\tilde \gamma_A(n)\sim\gamma_A(n)$.
Therefore, we can use the weight growth function $\tilde\gamma_A(n)$ in order to
compute the Gelfand-Kirillov dimensions and $\Ldim^\lambda A$, $\LLdim^\lambda A$ as well.

Suppose that $L$ is a Lie algebra and $X\subset L$.
By $\Lie(X)$ denote the subalgebra of $L$ generated by $X$.
In case $L$ is a restricted Lie algebra $\Lie_p(X)$ denotes the restricted subalgebra of $L$ generated by $X$.
Similarly, assume that $X$ is a subset in  an associative algebra $A$.
Write $\Alg(X)\subset A$ to denote an associative subalgebra (without unit) generated by~$X$.

\subsection{Scale for intermediate growth of (Lie) algebras}\label{SSinterm}
The growth of an integer sequence $\{a_n| n\ge 0\}$ is {\em subexponential} if one of the following
equivalent conditions holds~\cite{Ufn}:
\begin{enumerate}
   \item
    for any $r>1$ there exists $N$ such that $|a_n|\le r^n$, $n\ge N$;
   \item
    $\limsup_{n\to\infty} \sqrt[n]{|a_n|}=1$;
   \item
    For a series $f(z)=\sum _{n=0}^\infty a_n z^n$, $z\in\C$, we have $R_{convergence}=1$.
\end{enumerate}
M.~Smith proved that a subexponential growth of a Lie algebra implies a subexponential growth of its
universal enveloping algebra~\cite{Smith76}.
Borho and Kraft suggested {\em superdimensions} in order to
measure subexponential growths like $\exp(n^\beta)$, $0< \beta< 1$~\cite{BorKra}:
$$ \Sdim A= \limsup_{n\to\infty}  \frac{\ln\ln \gamma_A(n)}{\ln n},\qquad
\LSdim A= \liminf_{n\to\infty}  \frac{\ln\ln \gamma_A(n)}{\ln n}.$$
Kirillov and Kontsevich showed that a Lie algebra generated by two vector fields
in general position on the line has an intermediate growth of such kind~\cite{KirKon83}.
Lie algebras can have growth faster than any function $\exp(n^\beta)$, $\beta< 1$, but still subexponential.
For example, Lichtman proved that finitely generated solvable Lie algebras have subexponential growth~\cite{Licht84}.

In order to study such growths the author introduced the following scale of functions.
Let $\ln^{(1)}n=\ln n$, $\ln^{(q+1)}n=\ln(\ln^{(q)}n)$, $q\in\N$.
For small values, $\ln^{(q)}n$ may be not defined, negative, or less than 1, in such cases we redefine it to be 1.
Consider a series of functions $\Phi^q_\a (n)$, $q=1,2,3,\dots$ with a real parameter $\a\in\R^+$:
\begin{equation*}
\begin{split}
 \Phi^1_\a(n)&=\a,\\
 \Phi^2_\a(n)&=n^\a,\\
 \Phi^3_\a(n)&=\exp(n^{\a/(\a+1)}),\\
 \Phi^q_\a(n)&=\exp\bigg(\frac n {(\ln^{(q-3)}n)^{1/\a}} \bigg);
        \qquad q=4,5,6,\dots.
\end{split}
\end{equation*}
Now, upper and lower {\em $q$--dimensions}, where $q\in\N$, of a finitely generated algebra $A$ are defined as~\cite{Pe96}:
\begin{align*}
\Dim^q A  &= \inf \{ \a\in\R^+ \mid \exists N:\,\gamma_A(n)\le \Phi^q_\a(n),\ n\ge N\},\\
\LDim^q A &= \sup \{ \a\in\R^+ \mid \exists N:\,\gamma_A(n)\ge \Phi^q_\a(n),\ n\ge N\}.
\end{align*}

\begin{Lemma}[{\cite{Pe96}}]
$q$-dimensions of an algebra $A$ have the following properties:
\begin{enumerate}
 \item
   Functions $\Phi^q_\a(n)$ are of subexponential growth;
 \item
   $\Dim^q A=\a$ means that $\gamma_A(n)$ behaves like $\Phi^q_\a (n)$;
       \item
         $\Dim^1A=\dim_K A$ (vector space dimension);
       \item
         $\Dim^2A=\GKdim A$ (Gelfand-Kirillov dimension);
       \item
         $\Dim^3A$ coincides with $\Sdim A$ up to normalization (i.e. superdimension);
 \item
   Let $\Dim^q A=\a$, $0<\a<\infty$, $q\ge 2$; then
   $\Dim^{q-1} A=\infty$ and $\Dim^{q+1} A=0$.
\end{enumerate}
\end{Lemma}

Importance of $q$-dimensions is demonstrated by the following results.

\begin{Theorem}[{\cite{Pe96}}]\label{TgrowthU}
Let $L$ be a finitely  generated Lie algebra belonging to a level $q$, namely:
$$\Dim^q L=\a>0,\quad  q\in\N, \quad \text {(plus one technical condition for $q=2$)}.$$
Then its universal enveloping algebra $U(L)$ has growth of the next level $q+1$ with the same parameter $\a$:
$$\Dim^{q+1}U(L)=\a.$$
\end{Theorem}

Recall that {\em solvable} Lie algebras of length $q$ are defined
by the identity $S_q(X_1,\dots,X_{2^q})\equiv 0$;
where the last polynomial is defined recursively $S_1(X_1,X_2)=[X_1,X_2]$, and
$$S_{q+1}(X_1,\dots,X_{2^{q+1}})=[S_q(X_1,\dots,X_{2^q}),S_q(X_{2^q+1},\dots,X_{2^{q+1}})],\quad q\ge 1.$$

Our $q$-dimensions allow to specify the subexponential growth of solvable Lie algebras.
\begin{Theorem}[{\cite{Pe96}}]\label{TAq}
Let  $L=F(\AAA^q,k)$ be the free solvable Lie algebra of length $q$, $q\in\N$,
with $k$ generators, $k\geq2$. Then
$$\Dim^{q}L=\LDim^{q}L=k.$$
\end{Theorem}

In particular, these two theorems imply that $\Dim^{q+1}U(F(\AAA^q,k))=k$.
More generally, a Lie algebra $L$ is called {\em polynilpotent} with a tuple $(s_q,\dots,s_2,s_1)$ iff there exists  a chain of ideals~\cite{Ba}
$$0=L_{q+1}\subset L_q\subset\dots\subset L_2\subset L_1=L,$$
where  $L_i/L_{i+1}$ is nilpotent of class $s_i$, $i=1,\dots,q$.
By $\NN_{s_q}\!\cdots \NN_{s_2}\NN_{s_1}$ we denote the class of all such algebras, the tuple being fixed.
If $s_q=\cdots=s_1=1$ then we obtain the variety ${\AAA}^q$ of solvable Lie algebras of length~$q$.
On the other hand, any polynilpotent Lie algebra is solvable.

{\em Free polynilpotent} Lie algebras yield interesting examples of solvable Lie algebras.
The growth in case $\NN_{s_q}\!\cdots \NN_{s_2}\NN_{s_1}$ is similar to that for $\AAA^q$ (with other constants).
In particular, we have the following generalization of Theorem~\ref{TAq}.

\begin{Theorem}[{\cite{Pe96}}]
Let  $L=F(\NN_{s_q}\!\cdots \NN_{s_2}\NN_{s_1},k)$ be the free polynilpotent Lie algebra of rank $k\geq2$. Then
$$\Dim^{q}L=\LDim^{q}L= s_2\dim F(\NN_{s_1},k).$$
\end{Theorem}
A more precise asymptotic for free finitely generated solvable (more generally, polynilpotent)
Lie algebras was specified by the author in~\cite{Pe99int}.
A further generalization for solvable Lie {\it super}algebras see in~\cite{KlePe05}.

Let us mention one more result on intermediate growth of Lie algebras.
Let $W_ n$ be the Witt algebra and $\var(W_n)$ the variety defined by all identical relations of this algebra.
Formulae from~\cite{Molev86} imply the following result in terms of $q$-dimensions.
\begin{Theorem}[{\cite{Molev86}}]
Let $L$ be the free algebra of rank $k$ of the variety $\var(W_n)$, where $k\ge n+1$.
Then $\Dim^3 L = \LDim^ 3 L = n$.
\end{Theorem}

\section{Main result: Phoenix Lie algebras, their construction and properties}

\subsection{Drosophila Lie algebras}
Fix an integer $k\ge 3$.
Consider a set of letters $\Theta_0=\{1,\ldots,k\}$. We construct words in the alphabet $\Theta_0$ recursively:
$$
\Theta_{n+1} \subset \{ab\mid a,b\in \Theta_n,\ a\ne b\},\quad n\ge 0;\qquad
\Theta := \mathop{\cup}\limits_{n=0}^\infty \Theta_n.
$$
The inclusion sign above means that at each step we select words, while the remaining words are discarded.
We always additionally assume that $|\Theta_n|\ge 3$ for $n\ge 0$.
We call $\Theta_n$ {\it (fruit) flies} or {\it drosophilas} of {\it generation} $n$, $n\ge 0$, and $\Theta$ a {\it specie} of flies,
the selection process is interpreted as an action of an {\it experimenter}.
Suppose that at each step above the inclusion is substituted by the equality (i.e. we take all words),
then we get a specie of {\it wild flies} denoted as $\bar\Theta$.

Fix a prime $p>0$ and a specie $\Theta$.
Let $R=R(\Theta):=K[t_a| a\in \Theta]/(t_a^p| a\in \Theta)$ be the truncated polynomial ring, and $\{\partial_a| a\in\Theta\}$
respective partial derivations.
To obtain a better result (namely, to achieve a quasi-liner growth as a lower bound)
we shall need a formal divided power series ring $R(\Theta,\bar S)$ which also depends on a tuple $\bar S$, see definitions in the next section.
Define recursively derivations of $R(\Theta)$ indexed by the specie of flies $\Theta$:
\begin{equation}\label{pivot00}
v_a=\dd_a+t_a^{p-1}\!\!\!\! \sum_{\substack{ b\in \Theta_n\\[1pt] ab\in \Theta_{n+1}}}\!\!\!\!  t_b^{p-1} v_{ab},\qquad a\in \Theta_n, \ n\ge 0.
\end{equation}
We refer to $\{v_a|  a\in \Theta\}\subset \Der R$ as {\it virtual pivot elements}.
We define a {\em drosophila algebra} as a restricted Lie algebra generated
by the pivot elements of the zero generation $\LL(\Theta):=\Lie_p(v_1,\ldots,v_k)\subset \Der R$.

\subsection{Phoenix Lie algebras}
Let us formulate the main result of the paper.
Namely, we construct a family of nil 3-generated restricted Lie algebras of intermediate oscillating growth.
We call them {\it Phoenix algebras} because their growth is "quasi-linear" for infinitely many periods of time
(the algebra is "in hibernation", or "almost dying"), interchanged with periods of a rather fast intermediate growth.

\begin{Theorem}\label{T_main}
Let $\ch K=p> 0$, denote $\lambda:=\log_2(p{-}\frac 12)$.
Fix parameters  $q\in\N$ and $\kappa\in\R^+$. 
There exist a specie $\Theta$, a tuple $\bar S$, and
3-generated restricted Lie algebra $\LL=\LL(\Theta,\bar S)$ with the following properties.
\begin{enumerate}
\item For any $\epsilon>0$, $\delta>0$ there exist infinitely many integers $n$  satisfying (moments of "fast" growth):
$$
\exp\Big( \frac{n} {(\ln n)^{\lambda+\epsilon}} \Big)
\le \gamma_{\LL}(n)\le
\exp\Big(\delta\frac{n} {(\ln n)^{\lambda}} \Big).
$$
\item For any $\epsilon>0$
there exist infinitely many integers $n$ satisfying (moments of "quasi-linear" growth):
$$ n \big(\ln^{(q)} \!n\big )^{\kappa-\epsilon} \le\gamma_\LL(n) \le n \big(\ln^{(q)} \!n\big )^{\kappa+\epsilon}.$$
\item $\LL$ has an intermediate growth.
Moreover, its growth function is in a "wide sector" formed by two functions of types i) and ii).
More precisely, for any $\epsilon>0$, $\delta>0$ there exists $n_0$ such that:
$$
n \big(\ln^{(q)} \!n\big )^{\kappa-\epsilon}
\le \gamma_{\LL}(n)\le
\exp\Big(\delta\frac{n} {(\ln n)^{\lambda}} \Big), \qquad n\ge n_0.
$$
\item The $p$-mapping of $\LL$ is nil.
\end{enumerate}
\end{Theorem}

We can reformulate the oscillating nature of the growth function above in terms of $q$-dimensions introduced be the author
(see subsection~\ref{SSinterm}).
\begin{Corollary}
Let $\LL$ be as above, then its growth function is oscillating through levels 2,3,4:
\begin{enumerate}
\item $\Dim^4\LL=1/\lambda=\log_{p-1/2} (2)$;
\item $\LDim^2\LL=1$ (i.e. $\LGKdim \LL=1$);
\item more precisely, the lower bound is described as: $\LLdim^q \LL=\kappa$.
\end{enumerate}
\end{Corollary}

\subsection{Remarks}
\begin{Remark}
This result is motivated by examples of groups of oscillating growth by Kassabov and Pak~\cite{KasPak13}.
But the nature of our construction is completely different.
Unlike that examples our upper function is significantly less than the exponent (see also comments below).
The virtue of our approach is that we obtain nil-algebras, unlike that examples of groups that do not treat periodicity.

Unlike Theorem~\ref{TKassabovPak}, we do not proceed in terms of gaps between pairs of functions $g_2(n),g_1(n)$ and $f_2(n),f_1(n)$,
the latter being in some range, in order not to add more technical difficulties.
Since by interchanging wild and clover segments we can reach the lower and the upper functions of Theorem,
for arbitrary two inner increasing functions $g_1(n)<f_2(n)$ which are between two functions of i) and ii) of Theorem,
one can construct Lie algebras $\LL$ with oscillating growth satisfying
$\gamma_\LL(n_i)< g_1(n_i)$ and $\gamma_\LL(m_j)> f_2(m_j)$  for some integers $n_i$, $m_j$, $i\in\N$.
A serious technical work is needed to specify the gaps with the outer functions $g_2(n),f_1(n)$ to guarantee that the growth function visits
infinitely many times the respective segments $[g_2(n),g_1(n)]$ and $[f_2(n),f_1(n)]$, $n\in\N$.
\end{Remark}

\begin{Remark}
Let us explain appearance of the small constant $\delta$ in Theorem.
Below, we construct hybrid species of wild and triplex flies (actually we take a particular case of clover species).
Wild species have exponential growth (Lemma~\ref{Lasymp_wild}). Hybrids lose exponential growth, but
can reach any(!) subexponential growth as an upper bound (Lemma~\ref{Lflieshybrids}).
Lie algebras corresponding to wild species have the growth between two functions $\exp(C_{*}n/(\ln n)^\lambda)$
with positive constants $C_3$, $C_4$ (Theorem~\ref{Twild_growth}).
Now, we observe a similar phenomenon, the growth of hybrid Lie algebras drops,
we lose the "gate" with nonzero constants $C_3,C_4$, but the constant $\lambda$ remains in terms
of the wording of item i) with $\epsilon,\delta>0$.
\end{Remark}

\begin{Remark}
Without divided powers we would have only the lower polynomial bound $\LGKdim \LL=\log_{2p-1}p^3\in (1,3)$.
(This is the Gelfand-Kirillov dimension of the clover restricted Lie algebra with the trivial
constant tuple $S_i=R_i=1$, $i\ge 0$, see~\cite[Theorem 6.3]{Pe20clover}).
We need divided powers in order to get the lower Gelfand-Kirillov dimension to be one.
We used this approach to construct 2-generated nil restricted Lie algebras of quasi-linear growth in~\cite{Pe17}.
But now  we need its 3-generated version, see Theorem~\ref{Tparam} and computations in~\cite{Pe20clover}.
Different parameters yield 3-generated clover Lie algebras with even slower growth in Theorem~\ref{Tparam2},
which actually supplies the lower bounds in main Theorem~\ref{T_main}.
(one can also rewrite the lower bounds using functions of Theorem~\ref{Tparam}, which are bigger and less interesting).
\end{Remark}
\begin{Remark}
There is one more motivation to use divided powers.
By Bergman's theorem, the Gelfand-Kirillov dimension of an associative algebra cannot belong to the interval $(1,2)$~\cite{KraLen}.
Similarly, Martinez and Zelmanov proved that there are no finitely generated Jordan algebras with Gelfand-Kirillov
dimension strictly between 1 and 2~\cite{MaZe96}.
The author showed that a similar gap does not exist for Lie algebras, the Gelfand-Kirillov
dimension of a finitely generated Lie algebra can be an arbitrary number $\{0\}\cup [1,+\infty)$~\cite{Pe97}.
The same fact is also established for Jordan superalgebras~\cite{PeSh18Jslow}.
Now we have a stronger result, the gap $(1,2)$ can be filled with {\it nil} Lie $p$-algebras.
Namely, using so called constant tuples, we get self-similar nil restricted Lie algebras and their
Gelfand-Kirillov dimensions are dense on $[1,3]$~\cite{Pe20clover}.
\end{Remark}

\subsection{Further generalisations}
Below we just outline some more results that can be proved.

\begin{Corollary}
Let $\LL$ be as above and $U$ either universal or restricted enveloping algebra of $\LL$.
Then $\LL$ can be constructed such that
the growth of $U$  is oscillating through levels 3,4,5:
\begin{enumerate}
\item $\Dim^5 U=1/\lambda=\log_{p-1/2} (2)$;
\item $\LDim^3 U=1$.
\end{enumerate}
\end{Corollary}
\begin{proof}
Follows from the proof of Theorem~\ref{TgrowthU} in~\cite{Pe96},
because the estimates on the growth of the universal enveloping algebra $\gamma_{U(L)}(n)$
are made vie values of the growth function of the Lie algebra  $\gamma_L(m)$,
at $m=m(n)<n$, where $m(n)$ is a described function.
Actually, that estimates are valid for the restricted enveloping algebra as well.
Thus, we can choose the oscillating intervals in the construction of $\LL$ long enough
to guarantee such estimates.
\end{proof}
So, the growth function of $U$ above oscillates between the Ramunjan function and a function of level 5:
$$
\exp(\sqrt n),\qquad\quad
\exp\bigg( \frac{n} {(\ln\ln n)^{\lambda}} \bigg).
$$
This oscillation is closer to the oscillation of groups (and their group rings) of Theorem~\ref{TKassabovPak}.
So, analogies between groups and Lie algebras are sometimes better observed for properties of
their enveloping algebras and group rings.

\begin{Remark}
Using our construction, consider the associative algebra $\AA=\Alg(\Theta_0)=\Alg(v_1,v_2,v_3)\subset \End R$.
To evaluate its growth we need some more computations repeating and extending that of~\cite{Pe17}.
It seems that, as above, the growth function of $\AA$ is oscillating between two functions,
the first being of Gelfand-Kirillov dimension 2:
$$
n^2 \big(\ln^{(q)} \!n\big )^{\kappa},\qquad\quad
\exp\bigg( \frac{n} {(\ln n)^{\lambda}} \bigg).
$$
\end{Remark}

\begin{Remark}
In case of an arbitrary characteristic, we can construct {\it Phoenix Lie  super}algebras.
Now, let $R=\Lambda(t_a| a\in\Theta)$ be the Grassmann algebra, and $\{\partial_a| a\in\Theta\}$
respective partial superderivatives, see~\cite{Kac77,Pe16,PeOtto}.
We replace powers of variables in~\eqref{pivot00} by respective Grassmann variables
and consider the Lie superalgebra $\LL=\LL(\Theta)$ generated by the pivot elements of the zero generation.
Using computations of this paper, we can construct a specie of flies $\Theta$ and obtain an example of
a 3-generated (restricted) Lie superalgebra $\LL$ over an arbitrary field with the following properties:
\begin{itemize}
\item $\Dim^4\LL=\a_2$, where $\a_2=\log_{3/2} 2\approx 1,71$ (a fast subexponential upper bound);
\item $\LDim^2\LL=\a_1$, where  $\a_1=\log_3 8\approx 1,89$ (a polynomial lower bound); 
\item so, $\LL$ has an intermediate growth between levels 2 and 4, namely oscillating between functions:
$$
n^{\a_1},\qquad\quad
\exp\bigg( \frac{n} {(\ln n)^{1/\a_2}} \bigg).
$$
\item $\LL$ is nil in the following sense: $\ad a$ is nil for any $\Z_2$-homogeneous element $a\in\LL$
       (here we need to repeat computations of~\cite{Pe16,PeOtto,PeSh18FracPJ}).
\end{itemize}
\end{Remark}

\subsection{Path to Phoenix Lie algebras}
Let us briefly describe the structure of the paper.
\begin{enumerate}
\item
We define and study species of flies in details. Those without selection are {\it wild},
and those having three flies in each generation are {\it triplex}.
We are interested in {\it hybrid} species, which alternate between wild and triplex behaviour.
We define recursively virtual pivot elements and actual pivot elements
as derivations of divided power algebras, describe relations, introduce weight functions,
and prove $\Z^3$-gradings (Sections 4--6).
\item
We define a {\it drosophila} Lie algebra $\LL(\Theta,\bar S)$, related to a specie of flies $\Theta$
and a tuple $\bar S$  (in general case we are using divided power algebras).
We study general properties  of drosophila Lie algebras.
We prove that any restricted Lie algebra
$\LL(\Theta,\bar S)$ has a nil $p$-mapping, the tuple $\bar S$ being uniform (Theorem~\ref{Tnillity}).
This is a one of central results of the paper.
Our approach is a further development of ideas of~\cite{Pe06,ShZe08,Bartholdi15,Pe17}.
\item
Particular cases of triplex (restricted) Lie (super)algebras were studies in pervious papers.
Now, we study a special case of triplex algebras, called {\it clover} algebras.
A clear monomial basis for the clover Lie algebras
is described in a separate paper~\cite[Theorem 4.7]{Pe20clover}. 
In case of constant tuples we get self-similar nil restricted Lie algebras and their
Gelfand-Kirillov dimensions are dense on $[1,3]$~\cite{Pe20clover}.
Using specially chosen tuple $\bar S$, we construct clover Lie algebras of "quasi-linear" growth
(Theorem~\ref{Tparam}, Theorem~\ref{Tparam2}, proved in~\cite{Pe20clover}).
The clover restricted Lie algebras are some analogues of the Grigorchuk groups $G_\omega$~\cite{Grigorchuk84}.
\item
We study drosophila Lie algebras for wild species, tuple $\bar S$ being trivial, i.e. we are using truncated polynomials.
We determine an intermediate growth of these algebras (Theorem~\ref{Twild_growth}).
The difficulty is that there are no good bases for that algebras.
\item
Finally, we consider species that are hybrids of wild and clover ones.
In Section~\ref{SPhoenix}  we finish the proof of our main result (Theorem~\ref{T_main}).
We can construct hybrid Lie algebras so that they contain 3-generated clover Lie subalgebras studied in details in~\cite{Pe20clover},
we use that computations to simplify some estimates.
\end{enumerate}
\section{Fruit flies and Lie algebras of derivations of divided power algebras}

\subsection{Fruit flies}
Fix an integer $k\ge 3$  (we also consider a "degenerate case" $k=2$ below).
Consider a set of letters $\Theta_0=\{1,\ldots,k\}$.
We shall construct words over the alphabet $\Theta_0$, the words will be called
{\it drosophilas}, we also call them shortly {\it flies}.
We construct them recursively.
We refer to $\Theta_0$ as the fruit flies of {\it generation zero}.
Let the flies  $\Theta_n$ of generation $n$, where $n\ge 0$, are constructed.
Now, each pair of distinct flies $a,b\in \Theta_n$ produces two flies
$ab,ba\in \Theta_{n+1}$ of the next generation $n+1$.
In a general setting, we assume that
an {\it experimenter} executes a {\it selection} of the new generation $\Theta_{n+1}$ among all pairs above, while
the non-selected flies of generation $n+1$ are eliminated.
Formally, this process is described as:
\begin{align*}
\Theta_0 & =\{1,\ldots,k\};\\
\Theta_{n+1} &\subset \{ab\mid a,b\in \Theta_n, a\ne b\},\quad n\ge 0;\\
\Theta &= \mathop{\cup}\limits_{n=0}^\infty \Theta_n.
\end{align*}
A {\it length} of a word (i.e. a fly)  $a\in \Theta_n$ (in the initial alphabet $\Theta_0$) is $|a|=2^n$, $n\ge 0$.
Also, denote a {\it generation} 
of a fly $a\in \Theta_n$ as $\gen a=\log_2 |a|=n$.
Denote $\Theta_{0\ldots n}:=\cup_{j=0}^n\Theta_j$.

We call $\Theta$ a {\it specie} of flies.
Observe that a specie $\Theta$ is a set of {\it binary words} (i.e. each word is of length $2^n$, $n\ge 0$)
in an alphabet $\Theta_0$ such that for each $a\in \Theta\setminus \Theta_0$ its halves
(i.e. the halves $b,c$ of the expansion $a=bc$, $|b|=|c|$) belong to $\Theta$.
Different species of flies will be marked by additional symbols $\Theta^\a$, $\Theta^\b$, $\Theta'$,  \ldots, where
$\a,\b, {}'$ denote different selections,
we write also $\Theta^\a=\cup_{n=0}^\infty \Theta_n^\a$.
If the experimenter is sleeping and makes no selection we
get a specie of {\it wild flies} standardly denoted as  $\bar \Theta = \mathop{\cup}\limits_{n=0}^\infty \bar \Theta_n$.
Observe that a specie of wild flies is uniquely determined by the zero generation.
A {\it subspecie $\Theta^\a$} of a specie  $\Theta^\beta$ is its subset such that for any $a=bc\in\Theta^\a$, where $|b|=|c|$,
the halves $b,c$ belong to $\Theta^\a$ (below $b,c$ are referred to as parents of $a$).
Thus, a subspecie is a subset of a specie closed with respect to parenthood.

\subsection{Genealogical relations  $>$, $\succ$, $\sqsupset$, $\vdash$ on flies}
Let $a,b\in \Theta_n$ and $ab\in \Theta_{n+1}$.
We say that $a$ is a {\it father} (or the first parent) and $b$ a {\it mother} of their {\it child} $ab$,
also we say that $a,b$ are {\it parents} of $ab$.
Observe that separate flies have no gender, because $a$ is a father of $ab$, at the same time $a$ is a mother of $ba$.
Thus, a gender of a fly appears for an ancestor with respect to a fixed descendant only.
By $c> d$ we denote that $c$ is an {\it ancestor} of $d$ in some generations (equivalently, $d$ is a {\it descendant} of $c$).
We extend this relation by setting $c\ge c$. 
Observe that $\ge$ is a partial order on $\Theta$. Recall that our words are binary.
Let us describe $>$ in terms of words. Now, $c>d$, where $c,d\in\Theta$, is equivalent to the fact  that
we take one of halves of $d$, again one of its halves, etc. {\ldots} and get $c$.
In this case, we say that $c$ is a proper {\it binary subword} of $d$.

Let $c,d\in \Theta$, denote by $c\succ d$ that $c$ is a  proper {\it paternal ancestor} (i.e. a grand-\dots-grandfather) of $d$.
Equivalently, $c$ is a proper binary prefix of the word $d$. 
Put formally $c\succeq c$. Then, $\succeq$ is a partial order on~$\Theta$.

Similarly, let $c,d\in \Theta$, denote by $c\sqsupset d$ that $c$ is
a  proper {\it maternal ancestor} (a grand-\dots-grandmother) of $d$.
Equivalently, $c$ is a proper binary suffix of the word $d$. 
Set formally $c\sqsupseteq c$. Then, $\sqsupseteq $ is also a partial order on~$\Theta$.

We introduce one more convenient relation on $\Theta$ which is not a partial order.
Let $c,d\in \Theta$, then by $c\vdash d$ we denote that $c$
is one of two parents (i.e. one of two halves) of some $d'\in \Theta$ such that $d'\succeq d$
(i.e. we take a paternal ancestor $d'$ of $d$ and, finally, $c$ is a father or mother of $d'$).
In this case we say that $c$ is a {\it paternal-by-one ancestor} of $d$.
Observe that a fly $d$ has exactly two  such ancestors $c$ in each fixed senior generation $\Theta_m$, where $0\le m<\gen d$.

\subsection{Growth of drosophila species, hybrid species}
Assume that there is an integer $n$ such that $|\Theta_n|=1$ then $\Theta_{n+1}=\emptyset$.
Let also $\Theta_n=\{a,b\mid a\ne b\}$, then $\Theta_{n+1}\subset\{ab,ba\}$
and either all further generations have two flies or the population goes extinct.
So, we assume that the experimenter leaves at least three flies in all generations in order
to have a non-trivial specie of flies. Thus, {\bf we always assume that $|\Theta_n|\ge 3$ for all $n\ge 0$}.

\begin{Lemma}\label{Lasymp_wild}
Let $\bar \Theta$ be a specie of wild flies (i.e. without selection), and $|\bar\Theta_0|=k\ge 3$. Then
\begin{enumerate}
\item
there exists a limit $\theta_k:=\lim\limits_{n\to\infty}\sqrt[2^{n}]{\strut|\bar \Theta_n|}$, and   $2,33< \theta_k<k$.
\item $(\theta_k) ^{2^n}<|\bar\Theta_n|$ for all $n\ge 0$.
\item for any $\epsilon>0$ there exists $n_0$ such that
$|\bar\Theta_n| <(\theta_k+\epsilon) ^{2^n}\!\!,$ $n\ge n_0$.
\item $|\bar \Theta_{0\ldots n}|=(\theta_k+o(1)) ^{2^n}\!\!,$ $n\to\infty $.
\end{enumerate}
\end{Lemma}
\begin{proof}
By construction of the specie of wild flies,
we get a recurrence relation $|\bar \Theta_{n+1}|=|\bar \Theta_n|(|\bar \Theta_n|-1)$, $n\ge 0$, which
implies that the sequence $f(n)=\sqrt[2^{n}]{\strut |\bar \Theta_{n}|}$, $n\ge 0$, is strictly decreasing.
Hence, there exists a limit $\theta_k:=\lim\limits_{n\to\infty} f(n)$ and $\theta_k<f(n)< f(0)=k$ for $n> 0$.
In case $|\bar\Theta_0|=3$ we get $|\bar\Theta_1|=6$, $|\bar\Theta_2|=30$, $|\bar\Theta_3|=30\cdot 29> 2,33045^{2^3}$, \dots.
Thus, we get the first claim.

Claim ii) is valid because $f(n)$ is strictly decreasing.
Two remaining claims follow from i).
\end{proof}

Define a groupoid (i.e. magma) $G(\Theta)$  related to a specie $\Theta$.
Only in the case $a,b\in\Theta_n$ and $ab\in\Theta_{n+1}$
we define a nonzero product $a*b:=ab$.
The remaining products are trivial, so we add formally a zero element $G(\Theta):=\Theta\cup \{0\}$.
Observe that the product is non-associative.
The growth function $\gamma_{G(\Theta)}(m)$ counts all words in $G(\Theta)$ of length at most $m$, $m\ge 0$.
The next result shows that  $\bar \Theta$ is an essential  part of the whole language in $|\bar \Theta_0|=k$ letters.

\begin{Corollary}
Let $G=G(\bar \Theta)$ be the groupoid of wild flies, where $|\bar\Theta_0|=k$. Its growth is exponential:
$$\liminf_{n\to\infty}\sqrt[n]{\gamma_G(n)}=\sqrt{\theta_k},\qquad  \limsup_{n\to\infty}\sqrt[n]{\gamma_G(n)}=\theta_k. $$
\end{Corollary}
\begin{proof}
All elements $\Theta_m$ have length $2^m$ in the generators, $m\ge 0$.
Fix an integer $n$, set $m=[\log_2 n]$. The set of all words in the generators of length at most $n$ is $\bar \Theta_{0\ldots m}\cup\{0\}$.
By Lemma, using $n/2< 2^m\le n$ we get
\begin{align*}
&\gamma_G(n)= 1+|\bar \Theta_{0\ldots m}|=(\theta_k+o(1)) ^{2^{m}},\quad m\to\infty,\\
&(\theta_k+o(1))^{n/2}\le  \gamma_G(n)\le (\theta_k+o(1))^n,\quad n\to\infty.
\end{align*}
On the other hand, by setting $n=2^m$ and $n=2^{m}-1$ we check that the bounds above are exact.
\end{proof}

A specie of flies is {\em triplex} if all generations have 3 flies.
Consider a strictly increasing sequence of integers $M_i$, $i\ge 0$, where $M_0=0$.
Assume that $\mathop{\cup}_{j=M_{n-1}}^{M_n-1}\Theta_j$ is wild for odd $n$ and triplex for even $n$, for $n\ge 1$.
So, each wild or triplex segment starts with $M_j$.
Then $\Theta$ is a {\it hybrid} of wild and triplex species.
\begin{Lemma}\label{Lflieshybrids}
Fix $|\Theta_0|=3$ and any subexponential sequence $\{a_n|n\in \N\}$.
There exists a hybrid specie of flies $\Theta$ with an oscillating growth as follows.
\begin{enumerate}
\item groupoid $G=G(\Theta)$ has periods of slow growth:
$ \liminf_{n\to\infty} {\gamma_{G}(n)}/{\log_2 n}=3.$
\item
there exists an increasing sequence $\{n_j\in\N| j\ge 1\}$, such that $\gamma_G(n_j)\ge a_{n_j}$ for $j\in\N$
(periods of rather fast growth).
\item $G(\Theta)$ has subexponential growth.
\end{enumerate}
\end{Lemma}
\begin{proof}
Assume that generations till $\Theta_{M_{j-1}-1}$ are constructed. Denote $C:=\gamma_{G}(2^{M_{j-1}{-}1})$. Let $j$ be even.
Construct new triplex generations $|\Theta_{M_{j-1}}|=\cdots=|\Theta_m|=3 $. Length of words of $\Theta_m$ is $n:=2^m$.
Then $\gamma_G(n)=C{+}3(m{-}M_{j-1}{+}1)=C{+}3(\log_2 n{-}M_{j-1}{+}1)$.
For each $j$ even, we chose $M_j:=m+1$ sufficiently large to guarantee claim i).
(Similarly one proves the total lower bound using that all generations have at least three flies).

Let $j$ be odd. By construction above, the previous triplex segment yields
$C=\gamma_{G}(2^{M_{j-1}{-}1})=(3+o(1))(M_{j-1}-1)$, (as $M_{j-1}\to\infty$).
Now, we construct the next wild generations $\Theta_{M_{j-1}},\ldots,\Theta_m$, starting with $|\Theta_{M_{j-1}}|=3$,
we assume that their total size is greater than $C$.
Set $n:=2^m$. Lemma~\ref{Lasymp_wild} gives bounds on the new wild generations, for any $\epsilon>0$ we get
\begin{align*}
(\theta_3)^{2^{m-M_{j-1}}} &\le    \gamma_G(n)\le C + (\theta_3+\epsilon)^{2^{m-M_{j-1}}}\le 2(\theta_3+\epsilon)^{2^{m-M_{j-1}}}, \\
(\theta_3)^{2^{-M_{j-1}}} &\le    \sqrt[n]{\gamma_G(n)} \le 2^{-n}(\theta_3+\epsilon)^{2^{-M_{j-1}}}.
\end{align*}
The lower bound above is a fixed number greater than 1, by choosing
$n=2^m$ sufficiently large,  we make that number be greater than $\sqrt[n]{a_{n}}$, next we set $M_j:=m+1$.
The subsequence $n_j':=2^{M_j-1}$ for odd $j$ yields claim~ii).
Since $M_j\to \infty$, the upper bound yields claim iii).
\end{proof}

\subsection{Divided power series ring}
Let $\ch K=p\ge 2$.
Temporarily, let $\Theta$  be an arbitrary non-empty set (below $\Theta$ will be a fixed specie of flies).
Fix also a tuple of integers $\bar S=\{S_a\in \N \mid a\in \Theta \}$, indexed by $\Theta$.
We attach a divided power variable~$t_a$ for each $a\in \Theta$.
We consider a divided power series ring $R=R(\Theta,\bar S)$ which $K$-basis consists of formal symbols
$$\bigg\{ \prod_{a\in \Theta}t_a^{(i_a)} \ \bigg|\ 0\le i_a< p^{S_a}, \ a\in \Theta\bigg\}, $$
where only finitely many formal powers $i_a$ are non-zero.
Define a product of these elements as
$$\bigg(\prod_{a\in \Theta}t_a^{(i_a)}\bigg )\cdot \bigg (\prod_{a\in \Theta}t_a^{(j_a)}\bigg)=
 \prod_{a\in \Theta}\binom{i_a+j_a}{i_a} t_a^{(i_a+j_a)}. $$
Let $n=\sum_{k\ge 0} n_kp^k$, $m=\sum_{k\ge 0} m_kp^k$, $0\le n_k,m_k< p$,
be the $p$-adic expansions of integers $n,m$. One has the Lucas rule~\cite[p.61]{Strade1}:
\begin{equation}\label{congrunce}
\binom mn \equiv \prod_{k\ge 0}\binom {m_k}{n_k}\mod p.
\end{equation}
By this rule, $\binom{i_a+j_a}{i_a}=0$ in case $i_a+j_a\ge p^{S_a}$ for some $a\in \Theta$.
Thus, the product is well defined.
Then $R=R(\Theta,\bar S)$ is an associative commutative ring with unit.
Consider a set of the following (infinite) tuples indexed by $\Theta$:
\begin{equation}\label{tuple}
\Lambda:=\big\{\a=\big(\xi_a\mid \xi_a\in \{0,\dots,p^{S_a}-1\},\ a\in \Theta, \text{ finitely many $\xi_a$ are nonzero}\big)\big\}.
\end{equation}
Put $\mathbf t^\a:=\prod_{a\in\Theta} t_a^{(\xi_a)} \in R $, $\a\in\Lambda$.
Then $\{\mathbf t^\a\mid \a\in \Lambda\}$ is a basis of the algebra  $R(\Theta,\bar S)$.

Often, we shall use a {\it trivial tuple} $\bar S$, namely $S_a=1$ for all $a\in \Theta$.
In this case $R$ is just a ring of truncated polynomials
$R\cong K[T_a|a\in\Theta]/(T_a^p|a\in\Theta)$.
In a general case, $R=R(\Theta,\bar S)$, is also isomorphic to a ring of truncated polynomials~\cite{Strade1}.

\subsection{Lie algebra of special derivations}\label{SS_Liepure}
Fix $a\in \Theta$. Define an action $\partial_{a}$ on the whole of $R$ acting on
the respective divided variables only:
$\partial_{a}(t_a^{(i_a)}):=t_a^{(i_a-1)}$, $i_a\in\{0,\dots,p^{S_a}-1\}$,
where $t_a^{(0)}=1$ and $t_a^{(l)}=0$ for $l<0$.
We obtain derivations $\partial_{a}\in \Der R$, $a\in \Theta$.
Their $p^m$-powers are also derivations:
$\partial_{a}^{p^m}(t_a^{(i_a)})=t_a^{(i_a-p^m)}$, where $i_a\in\{0,\dots,p^{S_a}-1\}$, $m\ge 0.$
Clearly, $\partial_{a}^{p^{m}}=0$ for $m\ge S_a$. Consider a space of all formal sums
\begin{equation}\label{algebra_W}
\WW=\WW(\Theta,\bar S)=\bigg\{\sum_{\a\in \Lambda} {\mathbf t}^{\a}\bigg(\quad
                           \sideset{}{_{}^{\mathrm{fin}}}\sum_{b\in \Theta, l\ge 0} 
                           \lambda_{\a, b, l}\,\partial_{b}^{p^{l}}
                           \bigg)
                    \ \bigg|\ \lambda_{\a, b, l}\in K, 0\le l<S_b
                    \bigg\}.
\end{equation}
The notation $\sum^{\mathrm{fin}}$ denotes (here and below) that
a sum is taken over finitely many elements specified in the sum.
It is essential that the sum at each ${\mathbf t}^{\a}$, $\a\in\Lambda$,  is finite (the {\it finiteness condition}).
Lie product of elements of $\WW$ is well defined and $\WW$ acts on $R$ by derivations, we
obtain a restricted Lie algebra of {\em special derivations},
this term was suggested by Razmyslov and Radford~\cite{Rad86,Razmyslov}, see further applications in~\cite{PeRaSh}.
Terms  ${\mathbf t}^{\a}\partial_{b}^{p^{l}}$  will be called {\it pure Lie monomials}, and $\partial_{b}^{p^l}$ {\it pure derivations}.

\subsection{Ancestral derivations}
Now, assume that $\Theta$ is a specie of flies.
Recall that the ring $R=R(\Theta,\bar S)$ depends on the initial set $\Theta_0$,
the selection process that determines the whole specie $\Theta$, and the tuple $\bar S$.
Since in almost all cases $\Theta$, $\bar S$ are fixed, we often omit them.

Consider a tuple $\a\in \Lambda$~\eqref{tuple},
it has finitely many nonzero entrees, we write
 $\a=(\xi_{a_1},\ldots, \xi_{a_m}|  a_i\in \Theta)$.
Take $b\in\Theta$. We denote $\a> b$ iff $a_i>b$ for all $i=1,\ldots,m$.
Recall that in our terminology this is equivalent to the fact that $a_1,\ldots, a_m$ are ancestors of $b$
(equivalently, $a_1,\ldots, a_m$ are binary subwords of $b$).

Consider derivations~\eqref{algebra_W} where
all pure Lie monomials ${\mathbf t}^{\a}\partial_{b}^{p^{l}}$ satisfy $\a>b$,
i.e. the divided variables in ${\mathbf t}^{\a}$ correspond to ancestors of $b$ (the {\it ancestral condition}).
We call them {\it ancestral derivations}:
\begin{equation}\label{ancestral}
\WW^{\mathrm a}=\WW^{\mathrm a}(\Theta,\bar S):=\Bigg\{\sum_{\a\in \Lambda} {\mathbf t}^{\a}\Bigg(\quad
                         \sideset{}{_{}^{\mathrm{fin}}}\sum_{\substack{b\in \Theta\\ \a>b,\ l\ge 0 }} 
                           \lambda_{\a, b, l}\,\partial_{b}^{p^{l}}
                           \Bigg)
                    \ \Bigg|\ \lambda_{\a, b, l}\in K , 0\le l<S_b
                    \Bigg\}\subset \WW.
\end{equation}
\begin{Lemma}\label{Lancestral}
The set of ancestral derivations
$\WW^{\mathrm a}$ is a restricted Lie subalgebra of the restricted Lie algebra of special derivations $\WW$.
\end{Lemma}
\begin{proof}
Consider a product
\begin{equation*}
\Big[{\mathbf t}^{\a}\partial_{b}^{p^{l}},  {\mathbf t}^{\b}\partial_{c}^{p^{l'}} \Big]
={\mathbf t}^{\a} \partial_{b}^{p^{l}}\!\!\big({\mathbf t}^{\b}\big)\partial_{c}^{p^{l'}}
-{\mathbf t}^{\b}\partial_{c}^{p^{l'}}\!\!\big({\mathbf t}^{\a}\big)\partial_{b}^{p^{l}},
\qquad \a,\b\in\Lambda, \ a,b\in\Theta,\ l,l'\ge 0.
\end{equation*}
By assumption, $\a=(\xi_{a_1},\ldots,\xi_{a_n}| a_i{\in}\Theta)$ where $a_i>b$, $i=1,\ldots,n$,
and $\beta=(\xi_{b_1},\ldots \xi_{b_m}| b_j{\in} \Theta)$ and $b_j>c$, $j=1,\dots,m$.
Let the first term above be non-zero. Then $b=b_{j_{0}}$ for some $j_0$ and we get $a_i>b=b_{j_{0}}> c$ for all $i=1,\ldots,n$
and the first term is as required.
Thus, $\WW^{\mathrm a}$ is closed with respect to the Lie product.

By~\eqref{power_P},
a $p$-power is a sum of $p$-powers of pure Lie monomials and their $p$-fold commutators, the latter are  treated above.
Consider a pure Lie monomial $v=t_{a_1}^{(\xi_1)}\cdots t_{a_n}^{(\xi_n)}\dd_b^{p^l}$, where $a_i>b$, $i=1,\ldots,n$.
If $\xi_i>0$ for some $i$, then $v^p=0$, otherwise $v^p=\dd_b^{p^{l+1}}$.
\end{proof}

\subsection{Ancestral pro-augmentation derivations}
Basically, we follow approach~\cite{ShZe08}, applying ideas of~\cite{Pe17}.
Consider a tuple $\a=(\xi_{a_1},\ldots,\xi_{a_n}| a_i\in\Theta)\in\Lambda$~\eqref{tuple}.
Denote $|\a|=\xi_{a_1}+\cdots+\xi_{a_n}$.
Let $\xi_a\in \{0,\ldots, p^{S_a}-1\}$, $a\in \Theta$, consider the $p$-adic expansion
$\xi_a=\sum _{i\ge 0} \eta_i p^i$, $\eta_i{\in}\{0,\ldots,p{-}1\}$.
Define $p$-{\it adic norms} $|\xi_a|_p:=\sum_{i\ge 0}\eta_i $ and
$|\a|_p:=|\xi_{a_1}|_p+\cdots+|\xi_{a_n}|_p$, where $\a\in\Lambda$ as above.
Define subspaces in terms of the $p$-adic norm:
\begin{align*}
R_m:&=\langle{\mathbf t}^\a \ | \ \a\in\Lambda,\ |\a|_p= m \rangle_K\subset R,\qquad m\ge 0;\\
R^m:&=\langle{\mathbf t}^\a \ | \ \a\in\Lambda,\ |\a|_p\ge m \rangle_K\subset R,\qquad m\ge 0.
\end{align*}
\begin{Lemma} \label{L_p_norm}
Let the action of $\WW=\WW(\Theta,\bar S)$ on $R=R(\Theta,\bar S)$ and notations be as above.
\begin{enumerate}
\item Let $\a,\b\in\Lambda$, we get
${\mathbf t}^\a\cdot {\mathbf t}^\b=\lambda {\mathbf t}^{\a+\b}$, $\lambda\in K$. Let $\lambda\ne 0$,
then $|\a+\b|_p=|\a|_p+|\b|_p$.
\item $R=\mathop{\oplus}\limits_{m=0}^\infty R_m$ is a $\Z$-grading.
\item $R=R^0\supset R^1\supset R^2\supset\cdots\ $ is a decreasing filtration.
\item
Consider $\a,\b\in\Lambda$, $a\in\Theta$, $\a>a$, $0\le m<S_a$.
We have  $({\mathbf t}^\a \partial_a^{p^m})({\mathbf t}^\b)={\mathbf t}^{\b'}$. Assume that the result of the action is non-zero.
Then $|\b'|_p\ge |\a|_p+|\b|_p-1$.
\item $\WW (R^m)\subset R^{m-1}$, $m\ge 1$.
\end{enumerate}
\end{Lemma}
\begin{proof}
Fix $b\in\Theta$ and consider respective powers $t_b^{(\xi)}$, $t_b^{(\xi')}$ inside ${\mathbf t}^\a$, ${\mathbf t}^\b$.
Consider $p$-adic expansions
$\xi=\sum_{j= 0}^{S_b-1} \delta_j p^j$,
$\xi'=\sum_{j= 0}^{S_b-1} \delta_j' p^j$, and
$\xi+\xi'=\sum_{j\ge 0} \gamma_j p^j$, where $\delta_j,\delta'_j,\gamma_j\in\{0,\dots,p-1\}$.
We use congruence~\eqref{congrunce}
$$t_b^{(\xi)}\cdot t_b^{(\xi')}
=\binom {\xi+\xi'}{\xi}t_b^{(\xi+\xi')}
=\bigg(\prod_{j\ge 0} \binom {\gamma_j}{\delta_j}\bigg)t_b^{(\xi+\xi')}.
$$
Assume that $\delta_j+\delta_j'\ge p$ for some $j\ge 0$ and $j$ is minimal with this property.
Then $\gamma_j=\delta_j+\delta_j'-p<\delta_j$, so $\binom{\gamma_j}{\delta_j}=0$
and the product above is equal to zero, a contradiction.
Then $\delta_j+\delta_j'< p$ for all $j\ge 0$ and $\gamma_j= \delta_j+\delta_j'$, $j\ge 0$. We get
$ |\xi+\xi'|_p=\sum_{j\ge 0} \gamma_j=\sum_{j\ge 0} \delta_j+\sum_{j\ge 0}\delta_j'=|\xi|_p+|\xi'|_p .$
Applying equality for all $b\in\Theta$, the first claim follows.
Now ii), iii) are trivial.

Let us prove iv).
Let $t_a^{(\xi)}$ be the respective factor in ${\mathbf t}^\b$.
We have $\partial_{a}^{p^m}t_a^{(\xi)}=t_a^{(\xi')}$, where $\xi'=\xi-p^m$.
Consider the $p$-adic expansion $\xi=\sum_{i=0}^{S_a-1} \delta_i p^i$, $0\le \delta_i<p$.

We have two cases.
a) $\delta_m>0$. We get
$\xi'=\xi-p^m=(\delta_m{-}1)p^m +\!\!\!\sum\limits_{i=0, i\ne m}^{S_a-1} \!\!\!\delta_i p^i$, and
$ |\xi'|_p=\sum\limits_{i=0}^{S_a-1}\delta_i-1=|\xi|_p-1$.

b) $\delta_m=0$.
Assume that $\delta_i=0$ for all $i\ge m$.
Then $\xi=\sum_{i=0}^{m-1}\delta_i p^i\le \sum_{i=0}^{m-1}(p-1)p^i= p^m-1<p^m$
and $t_a^{(\xi')}=t_a^{(\xi-p^m)}=0$,
a contradiction that the action is non-trivial.
Hence, we have more nonzero coefficients in the $p$-adic expansion.
Let $\delta_{j}$, where $j>m$, be the first nonzero coefficient after $\delta_m$.
Coefficients of the $p$-adic expansion of $\xi'=\xi-p^m$ are:
\begin{align*}
&(\delta_0,\dots,\delta_{m-1},(p{-}1),\ldots, (p{-}1),\delta_{j}{-}1,\delta_{j+1},\dots,\delta_{S_a-1}),\\
& |\xi'|_p=\sum_{i=0}^{S_a-1}\delta_i+(p-1)(j-m)-1
  =|\xi|_p+(p-1)(j-m)-1
  \ge |\xi|_p.
\end{align*}
It remains to combine both cases and apply the first claim. Claim v) follows from iv).
\end{proof}
Following~\cite{ShZe08}, define subspaces
\begin{align}\nonumber
R^m \WW^{\mathrm a}:&=\{r_1w_1+\cdots +r_nw_n\ |\ r_i\in R^m, w_i\in \WW^{\mathrm a},n\ge 1\}\\
&=\{{\mathbf t}^{\a_1}w_1+\cdots + {\mathbf t}^{\a_r}w_r\ |\ \a_i\in \Lambda,\  |\a_i|_p\ge m,\ w_i\in \WW^{\mathrm a},\ r\ge 1\}
\subset \WW,\quad m\ge 0.
\label{t_W}
\end{align}
The ancestral condition is not guaranteed, so these sets do not belong to $\WW^{\mathrm a}$.
Since $\mathop{\cap}\limits_{m\ge 0} R^m \WW^{\mathrm a}=\{0\}$, the subspaces $\{R^m \WW^{\mathrm a}\cap \WW^{\mathrm a}\mid m\ge 0\}$
define a topology on $\WW^{\mathrm a}$.
We call it the {\it augmentation topology}.

Consider a subset consisting of {\it finite} linear combinations of ancestral pure Lie monomials:
$$\WW^{\mathrm a}_{\mathrm{fin}}:=
\big\langle {\mathbf t}^{\a} \partial_{b}^{p^{l}} \ \big|\ \a\in \Lambda,\ b\in \Theta,\ \a>b,\ 0\le l< S_b  \big\rangle_K\subset\WW^{\mathrm a}.
$$
Clearly, this is a restricted Lie subalgebra.
\begin{Lemma}\label{L_mathcal_W}
Let $\mathcal W(\Theta,\bar S):=\overline{\WW^{\mathrm a}_{\mathrm{fin}}}$ be
the completion of $\WW^{\mathrm a}_{\mathrm{fin}}$ in the augmentation topology.
\begin{enumerate}
\item We call this set by {\em ancestral pro-augmentation} derivations.  Then
\begin{equation*}
\mathcal W(\Theta,\bar S)=\overline{\WW^{\mathrm a}_{\mathrm{fin}}}
= \bigcap_{m\ge 0} \big(\WW^{\mathrm a}_{\mathrm {fin}}+ (R^m \WW^{\mathrm a}\cap \WW^{\mathrm a})\big)\subset \WW^{\mathrm a}.
\end{equation*}
\item $\mathcal W$ is a restricted Lie subalgebra of $\WW^{\mathrm a}$.
\end{enumerate}
\end{Lemma}
\begin{proof}
i) follows from the definition.
By~\eqref{t_W} and v) of Lemma~\ref{L_p_norm},  $[\WW^a, R^m\WW^a]\subset R^{m-1}\WW^a$, $m\ge 0$.
So,
$$\Big[\WW^{\mathrm a}_{\mathrm {fin}}+ (R^m \WW^{\mathrm a}\cap \WW^{\mathrm a}),
\WW^{\mathrm a}_{\mathrm {fin}}+ (R^n \WW^{\mathrm a}\cap \WW^{\mathrm a})\Big]
\subset \WW^{\mathrm a}_{\mathrm {fin}}+ (R^{\min\{m,n\}-1} \WW^{\mathrm a}\cap \WW^{\mathrm a}).$$
Hence, $\mathcal W$ is closed with respect to the Lie bracket.

Let us check that $\mathcal W$ is closed with respect to the $p$-mapping.
Let $a\in \WW^{\mathrm a}_{\mathrm {fin}}+ (R^m \WW^{\mathrm a}\cap \WW^{\mathrm a})$, $m\ge 1$.
Then $a=\!\!\!\!\sum\limits_{b\in\Theta, 0\le l<S_b} \!\!\!c_{b,l} \partial_b^{p^l}$, $c_{b,l}\in R$,
there exist $u_1,\ldots,u_n\in R^m$ and $N\in\N$, such that
$c_{b,l}\in Ru_1+\cdots +R u_n$ for terms with $\gen b\ge N$.
So, pure derivations of generation at least $N$ contain a non-trivial factor $c_{b,l}$, hence, their $p$th powers are trivial.
Using~\eqref{power_P},
modulo terms with derivations of generation at most $N$,  $a^p$ is a sum of $p$-fold commutators
$$\Big[c_{b_1l_1}\partial_{b_1}^{p^{l_1}},c_{b_2l_2}\partial_{b_2}^{p^{l_2}},\ldots, c_{b_pl_p}\partial_{b_p}^{p^{l_p}}\Big]
=w\partial_{b_{i_q}}^{p^{l_{i_q}}}, \quad 0\ne w\in R, $$
where $b_{i_q}$ is the minimal among $\{b_1,\ldots, b_p\}$ with respect to $<$ by the ancestral property.
Modulo finitely many terms, assume that $\gen b_{i_q}\ge N$.
Then $c_{b_{i_q},l_{i_q}}\in Ru_1+\cdots +R u_n$, hence,
$w$ above belongs to the ideal of $R$ generated by finitely many elements
$\partial_{b_{j_1}}^{p^{l_{j_1}}}\!\!\!\cdots \partial_{b_{j_s}}^{p^{l_{j_s}}} u_i$,
where $1\le i\le n$, $0\le s\le p-1$ and $b_{j_*}$ are ancestors of $b_{i_q}$.
By claim v) of Lemma~\ref{L_p_norm}, these generators belong to $R^{m-p+1}$.
Hence, $a^p\in \WW^{\mathrm a}_{\mathrm {fin}}+ (R^{m-p+1} \WW^{\mathrm a}\cap \WW^{\mathrm a})$.
\end{proof}

Define subspaces determined  by pure Lie monomials corresponding to flies $b\in\Theta$ of generation at least $n$:
\begin{equation*}
\mathcal W_n:=\mathcal W\bigcap
               \Bigg\{\sum_{\substack{\a\in \Lambda,\ b\in \Theta,\ \a>b,\ 0\le l<S_b  \\[2pt] \gen b\ge n} }
                           \lambda_{\a, b,l}\,{\mathbf t}^{\a} \dd_b^{p^l}
                    \ \Bigg|\ \lambda_{\a, b,l}\in K
               \Bigg\},\qquad n\ge 0.
\end{equation*}
\begin{Lemma}\label{LcalW}
$\mathcal W_n$, $n\ge 0$, are restricted ideals of $\mathcal W$.
\end{Lemma}
\begin{proof}
Follows from arguments on the product and $p$-mapping of Lemma~\ref{Lancestral}.
\end{proof}

\begin{Lemma}\label{Lder_subspecies0}
Let $\Theta'$ be a subspecie of a specie of flies $\Theta$ and fix a tuple
$\bar S=(S_a\in \N \mid a\in \Theta )$.
Define a restriction $\bar S|_{\Theta'}=(S_a\mid a\in\Theta')$. Then there exist natural epimorphisms
$$\phi: \WW^{\mathrm a}(\Theta,\bar S) \twoheadrightarrow \WW^{\mathrm a}(\Theta',\bar S|_{\Theta'}), \qquad
\phi: {\mathcal W}(\Theta,\bar S) \twoheadrightarrow {\mathcal W}(\Theta',\bar S|_{\Theta'}).$$
\end{Lemma}
\begin{proof}
We have a natural embedding of divided power series rings:
$R(\Theta',\bar S|_{\Theta'}) \hookrightarrow R(\Theta, \bar S)$.
Let us apply an ancestral derivation $v\in \WW^{\mathrm a}(\Theta)$ to $t_b^{(*)}$ where $b\in \Theta'$.
By definition of ancestral derivations~\eqref{ancestral},
we get a factor ${\mathbf t}^\a$, where non-zero entrees of $\a$  correspond to ancestors of $b\in\Theta'$.
Since a subspecie is closed with respect to ancestors, ${\mathbf t}^\a\in R(\Theta',\bar S|_{\Theta'})$.
We conclude that the restriction of $v$ onto $R(\Theta',\bar S|_{\Theta'})$ is well-defined and belongs to $\WW^{\mathrm a}(\Theta')$.
The action of $\phi$ on elements~\eqref{ancestral} is the following: terms with $b\in\Theta'$ remain while
terms with $b\in\Theta\setminus\Theta'$ are discarded.
The surjectivity is trivial.

Let $a\in \WW^{\mathrm a}_{\mathrm {fin}}(\Theta)+ (R^m(\Theta) \WW^{\mathrm a}(\Theta)\cap \WW^{\mathrm a}(\Theta))$, $m\ge 1$.
Using form of such elements described in proof of Lemma~\ref{L_mathcal_W}, we see that
$\phi(a)\in \WW^{\mathrm a}_{\mathrm {fin}}(\Theta')+ (R^m(\Theta') \WW^{\mathrm a}(\Theta')\cap \WW^{\mathrm a}(\Theta'))$.
By Lemma~\ref{L_mathcal_W}, the assertion on the image of  ${\mathcal W}(\Theta,\bar S)$ follows.
\end{proof}

Our  restricted Lie algebras will be constructed as subalgebras in $\mathcal W\subset \Der R$.

\section{Virtual Pivot elements and Drosophila Lie algebras}

In this section we introduce virtual pivot elements and a general object of our study:
drosophila restricted Lie algebras $\LL=\LL(\Theta,\bar S)=\Lie_p(v_1,\dots,v_k)$.
\subsection{Virtual pivot elements: recurrence definition}
Fix a specie of fruit flies $\Theta$, a tuple $\bar S$, and the respective ring $R=R(\Theta,\bar S)$.
Define recursively its derivations indexed by fruit flies:
\begin{equation}\label{pivot}
v_a=\dd_a+t_a^{(p^{S_a}-1)}\!\!\!\! \sum_{\substack{ b\in \Theta_n\\ ab\in \Theta_{n+1}}}\!\!\!\!  t_b^{(p^{S_b}-1)} v_{ab},\qquad a\in \Theta_n, \ n\ge 0.
\end{equation}
Observe that the specification $b\in \Theta_n$ in the sum above can be omitted.
We refer to $\{v_a| a\in \Theta\}\subset \Der R$ as {\it virtual pivot elements}.

\subsection{Virtual pivot elements: Infinite sums}
Using~\eqref{pivot}, we present them as infinite sums:
\begin{equation}\label{expansion}
v_a=\dd_a+
t_a^{(p^{S_a}-1)}\!\!\!\!\! \sum_{\substack{ b\in \Theta_n\\ ab\in \Theta_{n+1}}} \!\!\!t_b^{(p^{S_b}-1)}\bigg( \dd_{ab}+ t_{ab}^{(p^{S_{ab}}-1)}
\!\!\!\!\!\!\sum_{\substack{cd\in \Theta_{n+1}\\ abcd\in \Theta_{n+2}}}\!\!\!\! t_{cd}^{(p^{S_{cd}}-1)}\bigg(\dd_{abcd}+\cdots \bigg)\bigg),\qquad a\in \Theta_n.
\end{equation}
The specifications $b\in \Theta_n$, $cd\in \Theta_{n+1}$ in the sums above can be omitted.
We make a convention that $v_a=0$ for all $a\notin \Theta$.
We describe expansion~\eqref{expansion} formally as follows.

\begin{Lemma}\label{L_pivot_exp}
Let $a\in \Theta_n$, $n\ge 0$.
Then the respective virtual pivot element is written as an infinite sum:
  \begin{equation}\label{construction}
    v_a= \sum_{\substack{d\preceq a}}\Bigg(\prod_{\substack{c\vdash d\\[2pt] |a|\le |c|<|d|}} t_c^{(p^{S_c}-1)}\Bigg) \dd_d,\qquad a\in\Theta,
  \end{equation}
where $c,d\in\Theta$.
Let us explain the sum above.
\begin{enumerate}
\item
The sum starts with $d=a$ yielding the term $\dd_a$; next it is taken over all descendants $d$ of $a$ such that
$a$ is a proper paternal ancestor of $d$.
\item
Describe the factor at $\dd_d$, where $d\in \Theta_m$, $m> n$.
For each generation $l=n,n+1,\ldots,m-1$,
it contains exactly two divided variables, having their maximal powers,
the variables corresponding to $c\in \Theta_l$ such that $c \vdash d$
(i.e. $c$'s are the both generation-$\Theta_l$-paternal-by-one ancestors of $d$).
(In other words, for all $l=n,n+1,\ldots,m-1$ we take two initial subsequent subwords of $d$ of length $2^l$,
they are the flies yielding pairs of the respective divided variables).
\item The factor at $\dd_d$ has $2(\gen d-\gen a)$ different divided variables, where $d\preceq a$.
\end{enumerate}
\end{Lemma}
\begin{proof} Follows from expansion~\eqref{expansion}.
\end{proof}

\begin{Corollary}
Let $a,d\in \Theta$. Then
\begin{enumerate}
\item
\begin{equation*}
v_a (t_d^{(1)})=
\begin{cases}
0, & a\not\succeq d;\\
1, & a=d;\\
\prod\limits_{\substack{c\vdash d\\[2pt] |a|\le |c|<|d|}} t_c^{(p^{S_c}-1)},\qquad &  a\succ d.
\end{cases}
\end{equation*}
\item In particular, if $t_c^{(*)}$, $c\in\Theta$, appears in the expression above then $c> d$ and $|a|\le |c|$
(i.e. $c$ is a proper ancestor of $d$ of generation at least that of $a$).
In other words, $c$ is a proper binary subword of $d$ of length at least that of $a$.
\end{enumerate}
\end{Corollary}

\subsection{Drosophila Lie algebras}
As a rule, we consider that $\ch K=p\ge 2$, a specie of flies $\Theta$,
where $\Theta_0=\{1,\ldots,k\}$, $k\ge 3$, and a tuple $\bar S=(S_a| a\in\Theta)$ are fixed.
In the present paper we study properties of the restricted Lie algebra $\LL(\Theta,\bar S):=\Lie_p(v_1,\dots,v_k)$.
One can also consider its associative hull $\AA(\Theta,\bar S)=\Alg(v_1,\dots,v_k)$,
both algebras are generated by finitely many virtual pivot elements corresponding to the flies of the zero generation $\Theta_0$.
So, we also write $\LL=\Lie_p(\Theta_0)$. Often we omit the tuple $\bar S$.
We call $\LL(\Theta)$ a {\it drosophila Lie algebra} because it is constructed in terms of a specie of drosophilas $\Theta$.

\begin{Remark}
We warn that in a general situation we cannot guarantee that
all virtual pivot elements $\{v_a| a\in \Theta\}$ belong to $\LL=\Lie_p(\Theta_0)$,
this fact justifies the term {\it virtual pivot elements}.
(Nevertheless, this is the case for clover Lie algebras considered below).
\end{Remark}

Due to this observation, we consider also bigger
infinitely generated algebras $\Lie_p(v_a|a\in \Theta)\subset \Der R$ and $\Alg_p(v_a|a\in \Theta)\subset \End R$, which are generated by
{\it all} virtual pivot elements.

\begin{Lemma}\label{inL_W}
The restricted Lie algebras $\LL=\Lie_p(\Theta_0)$ and $\Lie_p(v_a|a\in \Theta)$
are contained in the Lie algebra  $\mathcal W$ of ancestral pro-augmentation derivations.
\end{Lemma}
\begin{proof}
Using~\eqref{expansion}, observe that $v_a\in \mathcal W$, for all $a\in\Theta$, and apply Lemma~\ref{L_mathcal_W}, claim ii).
\end{proof}

\begin{Lemma}\label{Lder_subspecies}
Let $\Theta'$ be a subspecie of a specie $\Theta$ and $\bar S=(S_a\in \N | a\in \Theta )$  a tuple.
Set $\bar S|_{\Theta'}=(S_a| a\in\Theta')$.
There exist natural epimorphisms of Lie algebras:
$$\phi: \LL(\Theta,\bar S) \twoheadrightarrow \LL(\Theta', \bar S|_{\Theta'}), \qquad
\phi: \Lie_p(v_a|a\in \Theta) \twoheadrightarrow \Lie_p(v_a|a\in \Theta').$$
\end{Lemma}
\begin{proof}
Consider the restriction map $\phi$ of Lemma~\ref{Lder_subspecies0}.
Its action on the pivot elements~\eqref{construction} is the following:
we leave exactly summands with $d\in\Theta'$. Thus, $\phi$ maps the virtual pivot elements with
respect to $\Theta$ onto the respective virtual pivot elements with respect to $\Theta'$.
\end{proof}


\subsection{Powers of virtual pivot elements}
\begin{Lemma}\label{Lrelations}
Consider two flies of the same generation $a,b\in \Theta_n$, $a\ne b$, $n\ge 0$. Then
\begin{enumerate}
\item
\begin{equation}\label{vap}
v_a^{p^m}=
\begin{cases}
\dd_a^{p^m}+ t_a^{(p^{S_a}-p^m)}\sum\limits_{ac\in \Theta_{n+1}}t_c^{(p^{S_c} -1)}v_{ac}, &\quad  0\le m< S_a;\\
\hfill \sum\limits_{ac\in \Theta_{n+1}}t_c^{(p^{S_c} -1)}v_{ac}, &\hfill m=S_a.
\end{cases}
\end{equation}
\item
\begin{equation}\label{v_power}
\left[ v_b^{p^{S_b}-1},v_a^{p^{S_a}}\right]
=v_{ab}+
t_b^{(1)}\!\!\!\!
\sum_{ac,bd\in \Theta_{n+1}}
t_c^{(p^{S_c}-1)}t_d^{(p^{S_d}-1)} [v_{bd},v_{ac}].
\end{equation}
\end{enumerate}
\end{Lemma}
\begin{proof}
We prove i) by induction on $m$. The base of induction $m=0$ is trivial by~\eqref{pivot}.
Assume that the claim is valid for $0\le m< S_a$.
The summation in~\eqref{power_P} is trivial because the second term cannot be used more than once:
\begin{align*}
v_a^{p^{m+1}} &={(v_a^{p^{m}})}^p=\Big(\partial_{a}^{p^m} + v_a^{(p^{S_a}-{p^m})}\!\!\!
\sum\limits_{ac\in \Theta_{n+1}}t_c^{(p^{S_c} -1)}v_{ac}\Big)^p\\
&={(\partial_{a}^{p^{m}})}^p+\big(\ad \partial_{a}^{p^m} \big)^{p-1}
          \Big(v_a^{(p^{S_a}-{p^m})}\!\!\! \sum\limits_{ab\in \Theta_{n+1}}t_c^{(p^{S_c} -1)}v_{ac} \Big)\\
&=\partial_{a}^{p^{m+1}}+
          v_a^{(p^{S_a}-{p^{m+1}})}\!\!\! \sum\limits_{ac\in \Theta_{n+1}}t_c^{(p^{S_c} -1)}v_{ac},\qquad\quad 0\le m< S_a.
\end{align*}
Using that $\partial_{a}^{p^{S_a}}=0$, claim i) is proved.
Consider the last claim.
Using $v_b=\dd_b+t_b^{(p^{S_b}-1)}\sum\limits_{bd\in \Theta_{n+1} }  t_d^{(p^{S_d}-1)} v_{bd} $
and~\eqref{vap} for $v_a^{p^{S_a}}$, we get:
\begin{align*}
[ v_b^{p^{S_b}-1},v_a^{p^{S_a}}]
&=(\ad v_b)^{p^{S_b}-2}[v_b,v_a^{p^{S_a}}]
=(\ad v_b)^{p^{S_b}-2} \Big[v_b,
t_b^{(p^{S_b} -1)}v_{ab} +\!\!\!\sum\limits_{\substack{ac\in \Theta_{n+1}\\ c\ne b}}t_c^{(p^{S_c} -1)}v_{ac}\Big]\\
&=\dd_b^{p^{S_b}-2}
\Big (t_b^{(p^{S_b} -2)}v_{ab} +
t_b^{(p^{S_b}-1)}\!\!\!\!\!
\sum_{ac,bd\in \Theta_{n+1}}
t_c^{(p^{S_c}-1)}t_d^{(p^{S_d}-1)} [v_{bd},v_{ac}]\Big)=\cdots  \\
&=v_{ab} +
t_b^{(1)}\!\!\!\!
\sum_{ac,bd\in \Theta_{n+1}}
t_c^{(p^{S_c}-1)}t_d^{(p^{S_d}-1)} [v_{bd},v_{ac}].\qedhere
\end{align*}
\end{proof}

\begin{Corollary}\label{Cternary}
Assume that $|\Theta_n|=|\Theta_{n+1}|=3$, $n\ge 0$.
Denote $\Theta_n=\{a,b,f\}$. Suppose that $ab\in \Theta_{n+1}$ and $\Theta_{n+1}\ne \{ab,ba, af\}$. Then
\begin{equation*}
\left[ v_b^{p^{S_b}-1},v_a^{p^{S_a}}\right] =v_{ab}.
\end{equation*}
\end{Corollary}
\begin{proof}
By setting, $a\ne b$.
Suppose that we have an additional term in~\eqref{v_power}.
Then $bd,ac\in \Theta_{n+1}$ for some letters $c,d\in\Theta_n$,  thus $b\ne d$, $a\ne c$.
Due to respective variables in the product, $\{b,c,d\}$ are distinct letters.
But $|\Theta_n|=3$.
The only possibility is that $d=a$, $c=f$. Thus,
$\Theta_{n+1}=\{ab, ba, af\}$, a contradiction.
\end{proof}

\subsection{Commutators of two virtual pivot elements of the same generation}
\begin{Lemma}\label{Lcomm_pivot}
Fix $a,b\in\Theta_n$, $n\ge 0$. Then (the first two terms below we call a {\em head}):
\begin{equation*}
[v_b,v_a]=t_a^{(p^{S_a}-1)}t_b^{(p^{S_b}-2)} v_{ab}-t_a^{(p^{S_a}-2)}t_b^{(p^{S_b}-1)} v_{ba}
+t_a^{(p^{S_a}-1)}t_b^{(p^{S_b}-1)}
\!\!\!\!\!\!\!\!\!\!\!\!\sum_{\substack{ 
                         ac,bd\in \Theta_{n+1}\\ a,b,c,d \text { distinct} }}\!\!\!\!\!\!\!\!\!\!
t_c^{(p^{S_c}-1)}t_d^{(p^{S_d}-1)} [v_{bd},v_{ac}]\bigg..
\end{equation*}
\end{Lemma}
\begin{proof}
By~\eqref{pivot}, we have
\begin{align*}
v_a&=\dd_a+t_a^{(p^{S_a}-1)}\bigg(t_b^{(p^{S_b}-1)}v_{ab} +\sum_{\substack{\Theta_n\ni c\ne b\\  ac\in \Theta_{n+1} }}  t_c^{(p^{S_c}-1)} v_{ac}\bigg);\\
v_b&=\dd_b+t_b^{(p^{S_b}-1)}\bigg(t_a^{(p^{S_a}-1)}v_{ba} +\sum_{\substack{\Theta_n\ni d\ne a\\ bd\in \Theta_{n+1} }}  t_d^{(p^{S_d}-1)} v_{bd}\bigg).
\end{align*}
Their product yields the result.
We observe that $a,b,c,d\in \Theta_n$ above are pairwise distinct because otherwise the product is trivial due to divided power variables.
\end{proof}

\begin{Corollary}\label{Ccomm_pivot}
Fix $a,b\in\Theta_n$, $n\ge 0$. Then $[v_b,v_a]\in \mathcal W_{n+1}$.
\end{Corollary}
\begin{proof}
We express the pivot elements $v_{ab}$, $v_{ba}$, $v_{bd}$, $v_{ac}$ above
via pure Lie derivations of generation at least $n+1$ and obtain the claim.
\end{proof}

\begin{Lemma}\label{Lab_ab}
Fix $a,b\in\Theta_n$, $n\ge 0$. Then
\begin{equation}\label{v_power_W}
\left[ v_b^{p^{S_b}-1},v_a^{p^{S_a}}\right]
\equiv v_{ab}\pmod{  t_b^{(1)}{\mathcal W}_{n+2}}.
\end{equation}
\end{Lemma}
\begin{proof}
We apply Corollary~\ref{Ccomm_pivot} to the last term in~\eqref{v_power}.
\end{proof}
\begin{Remark}
In many cases of algebras studied before we found clear monomial bases
or have at least approximate description~\cite{Pe17,Pe16,PeOtto,PeSh18FracPJ,PeSh13fib}.
To find a basis in case of an arbitrary drosophila Lie algebra looks problematic because
an iteration of expansion of Lemma~\ref{Lcomm_pivot} may not stop.
Nevertheless,  in an important particular case of clover Lie algebras we get a clear monomials basis, see below.
\end{Remark}

\section{Actual pivot elements, weights, and gradings}

\subsection{Actual pivot elements}
Recall that in a general case, we cannot guarantee that virtual pivot elements
$\{v_a|a\in\Theta\}$ belong to $\LL(\Theta)=\Lie_p(v_1,\dots,v_k)$.
Next, we define {\it actual pivot elements}  recursively:
\begin{equation*}
\bar v_c:=
\begin{cases}
v_c, & c\in \Theta_0;\\
\big [(\bar v_b)^{p^{S_b}-1},(\bar v_a)^{p^{S_a}}\big],\qquad & c=ab\in\Theta_n,\ a,b\in \Theta_{n-1}, \ n\ge 1.
 \end{cases}
\end{equation*}
\begin{Lemma}\label{Lactual_pivot}
The actual pivot elements $\{\bar v_c\mid c\in\Theta\}$ have the following properties.
\begin{enumerate}
\item $\bar v_c\in\LL(\Theta)=\Lie_p(v_1,\ldots,v_k)$ for all $c\in\Theta$;
\item
$$
\bar v_c\equiv v_c \pmod{ \sum_{d \sqsupset c} t_d^{(1)} {\mathcal W}_{\gen d +2}},\qquad c\in \Theta.
$$
where $d \sqsupset c$ is a sum over proper female ancestors $d$ of $c$.
\item
$\{\bar v_c\mid c\in\Theta\}$ are linearly independent.
\end{enumerate}
\end{Lemma}
\begin{proof}
Claim i) follows by induction from the definition of the actual pivot elements.
Claim~ii) is proved by induction on $\gen c$.
The base of induction $c\in \Theta_0$ is trivial.
Let $c\in \Theta_{n+1}$, $n\ge 0$. Then $c=ab$, $a,b\in \Theta_{n}$. By inductive hypothesis,
$$
\bar v_a\equiv v_a  \pmod{\sum_{d \sqsupset a} t_d^{(1)} {\mathcal W}_{\gen d +2}},\qquad
\bar v_{b}\equiv v_{b}  \pmod{ \sum_{d_1 \sqsupset b} t_{d_1}^{(1)} {\mathcal W}_{\gen d_1 +2}}.
$$
Observe that $v_a, v_{b}$ act trivially on $t_d, t_{d_1}$
because the latter correspond to their female ancestors, which are of senior generations.
Also, by Lemma~\ref{LcalW},  ${\mathcal W}_n$, $n\ge 0$, are ideals.
Using Lemma~\ref{Lab_ab}, we get
\begin{align*}
\bar v_c&=\big[(\bar v_{b})^{p^{S_{b}}-1},(\bar v_a)^{p^{S_a}}\big]\\
&\equiv \big[(v_{b})^{p^{S_{b}}-1},(v_a)^{p^{S_a}}\big]
 \pmod{ \sum_{d \sqsupset a} t_d^{(1)} {\mathcal W}_{\gen d +2}
+\sum_{d \sqsupset b} t_{d}^{(1)} {\mathcal W}_{\gen d +2}}\\
&\equiv v_{ab}  \pmod{ t^{(1)}_{b}{\mathcal W}_{n+2}
+\sum_{d \sqsupset a} t_d^{(1)} {\mathcal W}_{\gen d +2}
+\sum_{d \sqsupset b} t_{d}^{(1)} {\mathcal W}_{\gen d +2}}\\
&\equiv v_{ab}  \pmod{ \sum_{d \sqsupset ab} t_d^{(1)} {\mathcal W}_{\gen d +2}}
\equiv v_c  \pmod{ \sum_{d \sqsupset c} t_d^{(1)} {\mathcal W}_{\gen d +2}}.
\qedhere
\end{align*}
\end{proof}


\subsection{Weights}
By {\it pure monomials} we call products of divided powers and pure derivations, in particular, we have pure Lie monomials, see subsection~\ref{SS_Liepure}.
Set $\a_a=\cwt(\dd_a)=-\cwt(t_a^{(1)})\in\C$ for all $a\in \Theta$.
This values are easily extended to a weight function on pure monomials, additive on their (Lie or associative) products.
Next, consider weight functions such that all terms in recurrence relation~\eqref{pivot} have the same weight,
so, attaching the same value as a weight for the virtual pivot element as well.
Thus, we assume that 
\begin{align}
\alpha_{a}=\cwt(v_a)&=\cwt(\dd_a)=-(p^{S_a}-1)\alpha_{a}-(p^{S_b}-1)\alpha_{b}+\alpha_{ab}, 
\nonumber\\
\label{recorrencia}
\a_{ab}&= p^{S_a}\a_a+(p^{S_b}-1)\a_b, \qquad ab\in \Theta_{n+1},\ n\ge 0.
\end{align}
Denote the zero generation flies as $\Theta_0=\{1,\ldots,k\}$, so, the respective virtual pivot elements are denoted as $v_1,\ldots,v_k$.
Recurrence  relation~\eqref{recorrencia} expresses weights of the virtual pivot elements for flies of  generation $\Theta_{n+1}$
via weights corresponding to their parents in $\Theta_n$, $n\ge 0$.
Hence, any weight function satisfying~\eqref{recorrencia} is determined by its values on the zero generation,
namely, by $\cwt(v_j)$, $j=1,\dots,k$.
Let $\wt_i$ be the weight functions determined by values $\wt_i(v_j)=\delta_{ij}$, $1\le i,j\le k$.
Combine them into a {\it multidegree weight function} $\Gr(v):=(\wt_1(v),\ldots,\wt_k(v))$, where $v$ is a pure monomial.
By definition, $\Gr(v_1)=(1,0,\ldots,0),\ldots, \Gr(v_k)=(0,\ldots,0,1)$.
Thus, the space of weight functions satisfying~\eqref{recorrencia} is $k$-dimensional with a basis $\wt_1(*),\ldots,\wt_k(*)$.
Using~\ref{recorrencia}, we see that $\Gr(v_a)\in\NO^k$, $a\in\Theta$.
Finally, define a total {\it degree weight function} $\wt(v):=\sum_{j=1}^k \wt_j(v)$.

By a {\it generalized monomial} $a\in\End R$ we call any (Lie or associative) product of
pure monomials and virtual pivot elements.
By construction, actual pivot elements and their products are generalized monomials.
Observe that generalized monomials are written as infinite linear combinations of pure monomials.
Our construction implies that these pure monomials have the same weight,
we call this value the weight of a generalized monomial.
Thus, the weight functions are well-defined on generalized monomials as well.
Also, $\Gr(v)\in\NO^k$ for any generalized monomial $v$.

\subsection{$\NO^k$-gradings}
In many examples studied before~\cite{PeSh09,PeSh13fib,Pe16,Pe17,PeOtto,PeSh18FracPJ}
we were able, as a rule, to compute  explicitly basis functions for the space of weight functions and study multigradings in more details.
Using that base weight functions and multigradings we were able to get more information about our algebras.
In a general setting of the present paper it is not possible.
\begin{Theorem}\label{Tgraded}
Let  a specie of flies $\Theta$ and $\bar S$ be fixed, where $\Theta_0=\{1,\ldots,k\}$. Then
\begin{enumerate}
\item
the multidegree weight function $\Gr(v)$ is additive on products of generalized monomials $v,w\in\End R$:
$$ \Gr([v, w])=\Gr(v)+\Gr(w),\qquad \Gr(v\cdot w)=\Gr(v)+\Gr(w). $$
\item
The algebras $\bar \LL=\Lie_p(v_a| a\in\Theta)$, $\bar \AA=\Alg(v_a| a\in \Theta)$ are $\NO^k$-graded.
\item
$\LL=\Lie_p(v_1,\ldots, v_k)$, $\AA=\Alg(v_1,\ldots, v_k)$ are $\NO^k$-graded by multidegree in the generators $\{v_1,\dots,v_k\}$:
$$ \LL=\mathop{\oplus}\limits_{(n_1,\ldots,n_k)\in\NO^k} \LL_{n_1,\ldots,n_k},\qquad
\AA=\mathop{\oplus}\limits_{(n_1,\ldots,n_k)\in\NO^k} \AA_{n_1,\ldots,n_k}. $$
\item
$\wt(*)$ counts the degree of $v\in\LL,\AA$ in $\{v_1,\ldots,v_k\}$ yielding gradins:
$ \LL=\mathop{\oplus}\limits_{n=1}^\infty \LL_n$, $\AA=\mathop{\oplus}\limits_{n=1}^\infty \AA_n. $
\end{enumerate}
\end{Theorem}
\begin{proof}
Claim i) follows from the additivity of the weight function on products of pure monomials. Consider ii).
Recall that $\Gr(v)\in\NO^k$ for any generalized monomial $v$ and  $\Gr(*)$ is additive on their products.
Thus, we get $\NO^k$-gradings on $\bar\LL$, $\bar \AA$.

Let $v$ be a monomial in the generators $\{v_1,\ldots, v_k\}$ containing $n_i$ elements $v_i$ for $i=1,\ldots,k$.
By additivity, $\Gr(v)=n_1\Gr(v_1)+\ldots+n_k\Gr(v_k)= (n_1,\ldots,n_k)$.
Hence, $\LL$, $\AA$ are $\NO^k$-graded by multidegree in the generators.
Now, the last claim is evident.
\end{proof}

\subsection{Uniform tuples}
Now, we define specific tuples $\bar S=(S_a| a\in \Theta)$. Consider two cases.
1). There exist integers $(S_n| n\ge 0)$ such that $S_a=S_n$ for all $a\in \Theta_n$.
A particular case is that $S_a=1$ for all $a\in\Theta$.
2).  Assume that there exist integers $(S_n,R_n| n\ge 0)$ such
that for any $a\in\Theta_{n+1}$, decomposing $a=bc$, $b,c\in\Theta_n$, we have
either $S_a=S_n$, $S_b=R_n$ or $S_a=R_n$, $S_b=S_n$.
In case $S_n\ne R_n$ define $\Theta_n'=\{a\in\Theta_n| S_a=S_n\}$ and $\Theta_n''=\{a\in\Theta_n| S_a=R_n\}$, so
$\Theta_{n+1}$ consists only of some products $bc$ where $b,c$ belong to different sets $\Theta_n'$, $\Theta_n''$.
In both cases we say that the tuple $\bar S$ is {\it uniform}.
The first case is a particular case of the second by setting $R_n:=S_n$ for all $n\ge 0$.
This terminology is justified by the following.

\begin{Lemma}\label{Luniform}
Fix $\ch K{=}p> 0$, a specie of flies $\Theta=\mathop{\cup}_{n=0}^\infty\Theta_n$, and a uniform tuple $\bar S=(S_a| a\in\Theta_n)$. Then
the virtual pivot elements of the same generation $\{v_a| a\in \Theta_n\}$ have the same weight, denoted as $\wt(\Theta_n)$:
$$
\wt(\Theta_n)=\prod_{m=0}^{n-1}(p^{S_m}+p^{R_m}-1), \quad n\ge 0,\qquad \wt(\Theta_0)=1.
$$
In case of the trivial tuple (i.e. $S_a=1$ for all $a\in\Theta$) we get
$$
\wt(\Theta_n)=(2p-1)^n, \quad n\ge 0.
$$
\end{Lemma}
\begin{proof} By induction on $n$. The case $n=0$ is trivial.
Let $\a_n=\wt(a)$ for all $a\in \Theta_n$, and the claimed formula is valid.
Let $c=ab\in\Theta_{n+1}$, by~\eqref{recorrencia},
$\wt c =p^{S_a}\a_n +(p^{S_b}-1)\a_n =(p^{S_n}+p^{R_n}-1)\prod_{m=0}^{n-1}(p^{S_m}+p^{R_m}-1)
=\prod_{m=0}^{n}(p^{S_m}+p^{R_m}-1) $.
\end{proof}
Conversely, if weights of the virtual pivot elements of the same generations are equal,
then it is easy to see that the tuple must be uniform.

\begin{Corollary}\label{C}
Under assumptions of Lemma,
$\wt(\Theta_n)> p^{(S_0+R_0+\cdots +S_{n-1}+R_{n-1})/2}$, \ $n\ge 0$.
\end{Corollary}
\begin{proof}
Fix a real number $N>0$.
Consider a function $f(s)=p^{s}+p^{N-s}-1$ for $s\in [0,N]$.
It has the minimum at $s_0=N/2$.
Hence, $p^{s}+p^{N-s}-1\ge p^{N/2}+p^{N/2}-1> p^{N/2}$.
By applying this estimate to the factors
$p^{S_n}+p^{R_n}-1> p^{(S_n+R_n)/2}$, $n\ge 0$, we get the result.
\end{proof}

\section{Nillity of uniform drosophila Lie algebras $\LL(\Theta,\bar S)$}\label{Snil}

\begin{Lemma}\label{Ltailsnil}
Consider $u_1={\mathbf t}^{\a_1}w_1,\ldots, u_r={\mathbf t}^{\a_r}w_r\in \WW^{\mathrm a}$,
where $w_i\in \WW^{\mathrm a}$, $\a_i\in \Lambda$, $|\a_i|\ge 1$,  for $i=1,\ldots, r$.
We assume that $\wt(u_i)\ge 1$ and $\wt(w_i)\le C$, for $i=1,\ldots,r$, and a positive constant $C$
(namely, we assume that these inequalities  are valid for all  pure Lie monomials in expansions of $u_i$ and $w_i$).
Then the associative algebra $\Alg(u_1,\ldots, u_r)\subset \End R$ generated by these elements is nilpotent.
\end{Lemma}
\begin{proof}
Consider flies in supports of $\a_1,\ldots,\a_r$ and add all their ancestors, we get a finite set $\Omega\subset \Theta$.
Consider a  product $u=u_{i_1}\cdots u_{i_N}\ne 0$.
It has the following property: the number of "heads" $w_i$ is less or equal to the number of letters $t_j$.
We use transformations like
${\mathbf t}^{\a_1}w_1\cdot {\mathbf t}^{\a_2}w_2=
\binom{\a_1+\a_2}{\a_1}{\mathbf t}^{\a_1+\a_2} 
w_1w_2+ {\mathbf t}^{\a_1} [w_1,{\mathbf t}^{\a_2}]w_2,$
in order to rearrange $u$ and collect the heads on the right hand side, keeping their mutual order, (some heads can disappear).
On the left hand side, we get a product of elements of type $ [w_{j_1},\ldots, w_{j_s} , {\mathbf t}^{\a_i}] = {\mathbf t}^{\a_i'}$,
by the ancestral property, the support of $\a_i'$ is contained in $\Omega$.
After these transformations,
we get a linear combination of products $\tilde u={\mathbf t}^\a w_{j_1}\cdots w_{j_{N'}}$.
Observe that the above property is kept,
namely, the number $N'$ of heads $w_{j_*}$ is less or equal to degree of ${\mathbf t}^\a$,
the latter is bounded by $M=\sum_{b\in\Omega} p^{S_b}-1$.
Thus, $N'\le M$. Let $u''$ be a non-zero pure Lie monomial of $\tilde u$.
Hence, $\wt (u'')\le N'C\le MC$.
On the other hand, $\wt(u'')\ge N$. Hence, $N\le MC$.
\end{proof}

Fix a fly $d\in\Theta_n$, $n\ge 0$,  and consider a respective {\it pure Lie monomial} $w$
(such monomials appear in infinite expansions for the algebra $\WW^{\mathbf a}$, see~\eqref{ancestral}):
\begin{equation}\label{ux0}
w =r_{n-1}\partial_{d}^{p^m}= \Big(\prod_{c> d} t_c^{(\xi_c)}\Big) \partial_{d}^{p^m},\qquad
0\le \xi_c< p^{S_c}, \  0\le m< S_d.
\end{equation}
Consider pure Lie monomials appearing in decompositions~\eqref{construction} of  pivot elements $v_1,\ldots,v_k$   of zero generation:
\begin{equation*}
w_0= \Big(\prod_{c \vdash d} t_c^{(p^{S_c}-1)}\Big) \partial_{d}.
\end{equation*}
We compare common factors of $w$ and $w_0$. Namely, let $c\vdash d$.
Consider  $p$-adic expansions of $\xi_c$ and the maximal allowed power $p^{S_c}-1$ of the same variable:
\begin{align*}
\xi_c&=\delta_0 p^0+\delta_1 p^1+\delta_2 p^2+\cdots+ \delta_{S_c-1} p^{S_c-1},\qquad
       0\le \delta_j\le p-1,\ j=0,\dots,S_c-1;\\
p^{S_c}{-}1&=(p{-}1) p^0+(p{-}1) p^1+\cdots+ (p{-}1) p^{S_c{-}1}.
\end{align*}
Define a {\it depletion} of the variable $t_c$ in $r_{n-1}$, using also the $p$-adic norm defined above:
\begin{equation}\label{deplS}
\depl_{c}(r_{n-1}):=\sum_{j=0}^{S_c-1}(p-1-\delta_j)=S_c(p-1)-|\xi_c|_p,
\end{equation}
Formally define a depletion of the variable $t_d$ in a pure derivation, see a justification in Lemma~\ref{Ldep_p}.
\begin{equation}\label{deplpower}
\depl_{d}(\partial_{d}^{p^m}):=(p-1)m, \quad m\in\{0,\ldots,S_d-1\}.
\end{equation}
Define a {\it (paternal-by-one $d$-ancestors)-depletion} of the tail and the whole of a pure Lie monomial $w$~\eqref{ux0}
\begin{align*}
\depl_{\vdash d }(r_{n-1}) &:=\sum_{c\vdash d }\depl_{c}(r_{n-1});\\
\depl(r_{n-1}\partial_{d}^{p^m})
   &:=\depl_{d}(\partial_{d}^{p^m})+\depl_{\vdash d }(r_{n-1})=
   (p-1)m+\sum_{c\vdash d }\depl_{c}(r_{n-1}).
\end{align*}
Let us draw attention that we use depletions not in all variables with indices
$c>d$ (i.e. all ancestors of $d$, and such variables may be present in the tail $r_{n-1}$)
but only such that $c\vdash d$ (i.e. paternal-by-one ancestors of $d$).
Thus, by our definitions~\eqref{deplS}, \eqref{deplpower},
\begin{equation} \label{deplddk}
\depl(\partial_{d}^{p^m})
   =(p-1)\Big(m+\sum_{c\vdash d }S_c,\Big), \qquad 0\le m< S_d,\quad d\in\Theta.
\end{equation}

\begin{Lemma}\label{Ldep_p}
Consider a virtual pivot element or its power $v_a^{p^m}$, $0\le m\le S_a$, $a\in\Theta_n$, $n\ge 0$.
Then all pure Lie monomials $w$ of its expansion have the  same depletion:
$$ \depl(w)= (p-1)\big(m+\sum_{b\vdash a} S_b\big). $$
\end{Lemma}
\begin{proof}
By~\eqref{vap},
\begin{equation*}
v_a^{p^m}=
\begin{cases}
\dd_a^{p^m}+ t_a^{(p^{S_a}-p^m)}\sum\limits_{ac\in \Theta_{n+1}}t_c^{(p^{S_c} -1)}v_{ac}, &\quad  0\le m< S_a;\\
\hfill \sum\limits_{ac\in \Theta_{n+1}}t_c^{(p^{S_c} -1)}v_{ac}, &\hfill m=S_a.
\end{cases}
\end{equation*}
We expand $v_{ac}$ further using~\eqref{expansion} (or equivalently~\eqref{construction}) and get an infinite sum of pure Lie monomials.
Let $w= (\prod_{b> d} t_b^{(\a_b)})\partial_d$ be one of such monomials.
Consider $b\in\Theta$, $b\vdash d$, of generations $0,\ldots,n-1$, then $b\vdash a$ and $w$ has no such variables.
The depletion in these variables yields $\sum_{b\vdash a} (p-1)S_b$.
Consider $b\in\Theta$, $b\vdash d$,  of generations $n+1,\ldots, \gen d-1$.
By~\eqref{expansion} or~\eqref{construction}, $w$
has the maximal powers in these variables and they yield nothing in the total depletion.
We have two variables $b\vdash d$ of generation $n$, namely $a$ and $c$ above.
The depletion in $c$ is trivial. 
We have $p^{S_a}{-}{p^m}=0p^0+\cdots+0p^{m-1}+(p{-}1)p^{m}+\cdots+(p{-}1)p^{S_a{-}1}$.
Hence, $\depl_a(w)=(p-1)m$, coinciding with depletion of $a$ in $\partial_{a}^{p^m}$ in case $m<S_a$, see~\eqref{deplpower}.
\end{proof}

\begin{Lemma}\label{L_depl_prod}
Let $b\in\Theta_n$, $d\in\Theta_m$, where $n\le m$.
Consider ancestral pure Lie monomials 
$u=r_{n-1}\partial_{b}^{p^j}$, $0\le j< S_b$; and
$v=r_{m-1}\partial_{d}^{p^l}$, $0\le l< S_d$.
Assume that $[u,v]\ne 0$. Then
$$ \depl([u,v])\le \depl (v)+1. $$
\end{Lemma}
\begin{proof}
By assumption, $r_{n-1}=\prod_{a>b} t_a^{(\eta_a)}$, $0\le \eta_a<S_a$, and
$r_{m-1}=\prod_{c>d} t_c^{(\xi_c)}$, $0\le \xi_c<S_c$. Then
$$ [u,v]=r_{n-1}r_{m-1}'\partial_{d}^{p^l},\quad\text{where}\quad
r_{m-1}'=[\partial_{b}^{p^j}, r_{m-1}]. $$
Since the product is nontrivial, a nontrivial power of $t_b$ enters $r_{m-1}$, hence $b>d$.

First, we compare (paternal-by-one $d$-ancestors)-depletions of $r_{m-1}$ and $r_{m-1}'$.
We get a change of depletion with respect to $t_b$ only.
In case $b\nvdash d$, variable $t_b$ is irrelevant and depletions of $r_{m-1}$, $r_{m-1}'$ are the same.
Thus, assume that $ b\vdash d$.
Let $t_b^{(\xi)}$, $t_b^{(\xi')}$ be the respective factors inside $r_{m-1}$, $r_{m-1}'$.
By claim iv) of Lemma~\ref{L_p_norm}, we have $|\xi'|_p\ge |\xi|_p-1$.
Using~\eqref{deplS}, $\depl_b r_{m-1}'\le \depl_b r_{m-1}+1$.
Since depletions in other variables do no change, we get $\depl_{c\vdash d }(r_{m-1}')\le \depl_{c\vdash d }(r_{m-1})+1$.

Second, we evaluate a depletion of $r_{n-1}r_{m-1}'$. Fix $a\vdash d$.
Let $t_a^{(\xi)}$, $t_a^{(\xi')}$, $t_a^{(\xi'')}$ be the respective powers in $r_{n-1}$, $r_{m-1}'$, $r_{n-1}r_{m-1}'$.
Proving claim i) of Lemma~\ref{L_p_norm}, we checked that $|\xi''|_p= |\xi|_p+|\xi'|_p$.
Hence,
$$
\depl_{a}(r_{n-1}r_{m-1}')=S_a(p-1)-|\xi''|_p= S_a(p-1)-|\xi|_p-|\xi'|_p\le S_a(p-1)-|\xi'|_p=\depl_a r_{m-1}'.
$$

Therefore, $\depl_{\vdash d} (r_{n-1}r_{m-1}')\le \depl_{\vdash d}(r_{m-1}')\le \depl_{\vdash d}(r_{m-1})+1$,
while the $d$-depletion of $\partial_d^{p^l}$ remains the same.
Hence, $\depl([u,v])\le \depl (v)+1$.
\end{proof}

Let $w\in \WW^{\mathrm a}$. We define the depletion $\depl(w)$ as a maximum of depletions of its pure Lie monomials.

\begin{Lemma}\label{Ldepl_comm}
Let  $w\in \WW^{\mathrm a}$ has a finite depletion. Then $\depl(w^p)\le \depl (w)+p-1$.
\end{Lemma}
\begin{proof}
Write $w$ as an infinite linear combination of pure Lie monomials $u_j$.
By properties of the $p$-mapping, $w^p$ is a linear combination of monomials of two types:
a) $u'=u_j^p$ and b) $u''=[u_{i_1},u_{i_2},\dots, u_{i_p}]$.
By Lemma~\ref{L_depl_prod},
$\depl (u'')\le \max\{\depl (u_{i_1}),\dots, \depl (u_{i_p})\} +p-1\le \depl (w)+p-1$.

Consider the case $u'=u_j^p$.
Let $u_j= r\partial_d^{p^m}$, where $r=\prod_{b> d} t_b^{(\a_b)}$, $d\in\Theta$.
If $r\ne 1$, then 
$r^p=0$ and $u'= r^p\partial_d^{p^{m+1}}=0$.
Thus, $u_j=\partial_d^{p^m}$ and   $u'=\partial_{d}^{p^{m+1}}$, where $m+1<S_d$.
Using~\eqref{deplddk},
$\depl (u') =(p-1)\Big(m+1+\sum_{c\vdash d} S_c\Big) =(p-1)\Big(m+\sum_{c\vdash d} S_c\Big)+p-1 =\depl (u_j)+p-1\le \depl (w)+p-1.$
\end{proof}

\begin{Lemma}\label{Ldepl_finite}
Elements of $\LL(\Theta,\bar S)$ have finite depletions.
\end{Lemma}
\begin{proof}
Consider the generators $v_1,\dots,v_k$ of $\LL$.
By Lemma~\ref{Ldep_p}, $\depl(v_1)=\cdots=\depl(v_k)=0$. The result follows by Lemma~\ref{L_depl_prod} and Lemma~\ref{Ldepl_comm}.
\end{proof}

\begin{Lemma}\label{L_est34}
Fix numbers $p\ge 2$, and $s,r\ge 1$. Then
$ p^s+p^r-1\le \frac 34 p^{s+r}. $
\end{Lemma}
\begin{proof}
$p^s+p^r-1=p^{s+r}(\frac 1{p^r}+\frac 1{p^s}-\frac 1{p^{s+r}})
=p^{s+r}(1-(1-\frac 1{p^s}) (1-\frac 1{p^r}))\le p^{s+r}(1-(1-\frac 12)(1-\frac 12))=\frac 34 p^{s+r}. $
\end{proof}

\begin{Theorem}\label{Tnillity}
Consider $\ch K=p>0$, a specie of flies $\Theta$, a uniform tuple  $\bar S=(S_a| a\in \Theta)$, and
the drosophila restricted Lie algebra $\LL(\Theta,\bar S)=\Lie_p(v_1,\ldots,v_k)$.
Then $\LL$ has a nil $p$-mapping.
\end{Theorem}
\begin{proof}
Fix $0\ne w\in \LL$, and an integer $N$.
Assume that the expansion of $w^{p^N}$ via pure Lie monomials contains a pure derivation $\partial_{a}^{p^m}$, where $a\in\Theta_n$, $0\le m<S_a$.

First, compare depletions. By Lemma~\ref{Ldepl_finite}, we have $C:=\depl w<\infty$.
Applying Lemma~\ref{Ldepl_comm},
\begin{equation*}
\depl\partial_{a}^{p^m}\le \depl(w^{p^N})\le C+N(p-1), \quad N\ge 0.
\end{equation*}
Combining with~\eqref{deplddk}, we get:
\begin{equation}\label{xxn}
\depl(\partial_{a}^{p^m})
   =(p-1)\Big(m + \sum_{c\vdash a}S_c\Big)\le C+N(p-1), \quad 0\le m< S_a.
\end{equation}

Second, we evaluate weights.
Since monomials of $w$ have weights at least 1, a monomial $\partial_{a}^{p^m}$ of $w^{p^N}$ has weight at least $p^N$.
Using Lemma~\ref{Luniform}, the bound of Lemma~\ref{L_est34}, item ii) of~Lemma~\ref{L_pivot_exp}, and~\eqref{xxn}, we get
\begin{align*}
p^N&\le \wt(\partial_{a}^{p^m})=p^m\wt(v_a)=p^m\prod_{i=0}^{n-1}(p^{S_i}+p^{R_i}-1)
\le p^m  \prod_{i=0}^{n-1}\Big(\frac 34 p^{S_i+R_i}\Big)\\
&=\Big(\frac 34 \Big)^n p^{m+\sum_{c\vdash a} S_c}
\le \Big(\frac 34\Big)^n p^{N+C/(p-1) }.
\end{align*}
Above we used that $S_i$, $R_i$ are tuple entrees indexed  by the two flies $c$  of generation $i$, such that $c\vdash a$, $0\le i<n$.
Hence, $(3/4)^n p^{C/(p-1)}\ge 1$. We obtain a bound on a generation of $a\in\Theta_n$:
\begin{equation*}
n\le n_0:=\Big[\frac {C\log_{4/3} p  } {p-1}\Big].
\end{equation*}
Thus, $\wt(\partial_{a}^{p^m})$, $m<S_a$, is bounded as well. Third,
recall that $\partial_{a}^{p^m}$ is a monomial in the expansion of $w^{p^N}$, thus having a weight at least $p^N$.
Choosing $N$ sufficiently large, we get a contradiction.
This contradiction proves, that the expansion of $w^{p^N}$ via pure Lie monomials
cannot contain a pure derivation $\partial_{a}^{p^m}$.

By Lemma~\ref{inL_W}, $w'=w^{p^N}\in\mathcal W$, and using Lemma~\ref{L_mathcal_W},
$w^{p^N}\in \WW^{\mathrm a}_{\mathrm {fin}}+ (R^1 \WW^{\mathrm a}\cap \WW^{\mathrm a})$.
Thus, $w'=w_1+w_2$, where $w_1\in \WW^{\mathrm a}_{\mathrm {fin}}$ and
$w_2\in R^1 \WW^{\mathrm a}\cap \WW^{\mathrm a}$.
By~\eqref{t_W}, there exist an integer $r\ge 1$, tuples $\a_1,\ldots,\a_r\in \Lambda$, where $|\a_i|_p\ge 1$,
and $u_1,\ldots,u_r\in \WW^{\mathrm a}$ such that $w_2={\mathbf t}^{\a_1}u_1+\cdots +{\mathbf t}^{\a_r}u_r$.
Since the expansion of $w'$ via pure Lie monomials
cannot contain a pure derivation $\partial_{a}^{p^m}$, we conclude that
$w_1={\mathbf t}^{\b_1}\partial_{b_1}^{p^{l_1}}+\cdots+ {\mathbf t}^{\b_s}\partial_{b_s}^{p^{l_s}}$, where $b_i\in\Theta$,
$\b_i\in\Lambda$ and $|\b_i|\ge 1$.
Combining these formulas, we get a finite sum $w' ={\mathbf t}^{\g_1}u_1+\cdots +{\mathbf t}^{\g_{r'}}u_{r'}$,
where $u_i\in \WW^{\mathrm a}$, $\a_i\in\Lambda$, and $|\a_i|\ge 1$ for all $i=1,\ldots,r'$.
We conclude that $w'=w^{p^N}$ is nil by Lemma~\ref{Ltailsnil}.
\end{proof}

\section{Clover flies and Clover restricted Lie algebras $\TT(\Xi)$}

\subsection{Duplex specie and Duplex Lie (super)algebras}
We call a specie $\Theta=\cup_{n=0}^\infty \Theta_n$ {\it duplex} provided that $|\Theta_n|=2$ for all $n\ge 0$.
The duplex specie is unique up to relabelling flies.
Let $\Theta_0=\{a_0,b_0\}$, then $\Theta_n=\{a_n,b_n\}$ where $a_{n+1}=a_nb_n$, $b_{n+1}=b_na_n$ for $n\ge 0$.
We warn that in order to keep notations of the original paper, $a_n,b_n$, $n\ge 0$,
denote both the flies and the respective pivot elements below.
The first example of Lie superalgebra $\RR$ in~\cite{Pe16} corresponds to the duplex specie in terms of our construction~\eqref{pivot}.
\begin{Example}[first example $\RR$ in \cite{Pe16}] \label{E1}
Consider the Grassmann superalgebra $\Lambda=\Lambda[x_i,y_i| i\ge 0]$.
Using its partial superderivatives,
define recursively odd elements in the associative superalgebra $\End(\Lambda)$:
\begin{equation*}
\begin{split}
a_i &= \partial_{x_i} + y_ix_i a_{i+1},\\
b_i &= \partial_{y_i} + x_iy_i b_{i+1},\\
\end{split}  \qquad i\ge 0.
\end{equation*}
Define a Lie superalgebra $\RR:=\Lie(a_0,b_0,c_0)\subset \Der \Lambda$.
In case $\ch K=2$, assume that the quadratic mapping on odd elements is just the square of
the respective operator in $\End\Lambda$.
\end{Example}

The following example is also based on the duplex specie.

\begin{Example}[family of restricted Lie algebras $\LL(\Xi)$  in~\cite{Pe17}]\label{E2}
Let $\ch K=p>0$. Consider integers
$\Xi=(S_n,R_n| n\ge 0)$, which determine a divided power series ring
$\Omega(\Xi)=\langle x_0^{(\xi_0)}y_0^{(\eta_0)}\!\!\!\cdots x_i^{(\xi_i)}y_i^{(\eta_i)}|
0\le\xi_i<p^{S_i}, 0\le \eta_i<p^{R_i}, i\ge 0\rangle$. Define pivot elements recursively:
\begin{equation}\label{aibip}
\begin{split}
a_i &= \partial_{x_i} + x_i^{(p^{S_i}-1)} y_i^{(p^{R_i}-1)} a_{i+1};\\
b_i &= \partial_{y_i} + x_i^{(p^{S_i}-1)} y_i^{(p^{R_i}-1)} b_{i+1};
\end{split}
\qquad i\ge 0.
\end{equation}
Define a restricted Lie algebra $\LL(\Xi):=\Lie_p(a_0,b_0)\subset \Der \Omega(\Xi)$.
\end{Example}

\subsection{Triplex species and Triplex Lie (super)algebras}
Let $\Theta_0=\{a,b,c\}$ be the zero generation of a specie of flies.
In case of wild flies (i.e. without selection), they determine three pivot elements (see~\ref{pivot}):
\begin{equation*} 
\begin{split}
v_{a  }&=\dd_{a  }+t_{a  }^{(p^{S_{a  }}-1)}\Big( t_{b  }^{(p^{S_{b  }}-1)} v_{a  b  }+ t_{c  }^{(p^{S_{c  }}-1)} v_{a  c  } \Big);\\
v_{b  }&=\dd_{b  }+t_{b  }^{(p^{S_{b  }}-1)}\Big( t_{a  }^{(p^{S_{a  }}-1)} v_{b  a  }+ t_{c  }^{(p^{S_{c  }}-1)} v_{b  c  } \Big);\\
v_{c  }&=\dd_{c  }+t_{c  }^{(p^{S_{c  }}-1)}\Big( t_{a  }^{(p^{S_{a  }}-1)} v_{c  a  }+ t_{b  }^{(p^{S_{b  }}-1)} v_{c  b  } \Big).
\end{split}
\end{equation*}
These formulas contain pivot elements for all six wild flies of the first generation: $\bar \Theta_1=\{ab,ba, ac,ca, bc,cb\}$.
Now, we require the experimenter to select three flies and rename as $\Theta_1=\{a_1,b_1,c_1\}$, and
continue to leave three flies in all generations and rename as $\Theta_n=\{a_n,b_n,c_n\}$, $n\ge 0$.
We obtain a {\it triplex specie}, that determines a {\it triplex restricted Lie algebra}.
Clearly, there are many triplex species.

The second example of a Lie superalgebra $\QQ$ in~\cite{Pe16} is related to a triplex specie:
$a_{n+1}=a_nb_n$, $b_{n+1}=b_nc_n$, $c_{n+1}=c_na_n$, $n\ge 0$ in terms of the recursion above.
We warn that $a_n,b_n,c_n$, $n\ge 0$,
denote both the flies and the respective pivot elements below keeping  notations of~\cite{Pe16}.
\begin{Example}\label{Example_Q}[second example $\QQ$ in \cite{Pe16}]
Consider the Grassmann superalgebra $\Lambda=\Lambda[x_i,y_i,z_i| i\ge 0]$.
Using its partial superderivatives,
define recursively odd elements in the associative superalgebra $\End(\Lambda)$:
\begin{equation*}
\begin{split}
a_i &= \partial_{x_i} + y_ix_i a_{i+1},\\
b_i &= \partial_{y_i} + z_iy_i b_{i+1},\\
c_i &= \partial_{z_i} + x_iz_i c_{i+1},
\end{split}
\qquad i\ge 0.
\end{equation*}
Define a Lie superalgebra $\QQ:=\Lie(a_0,b_0,c_0)\subset \Der\Lambda$.
In case $\ch K=2$, assume that the quadratic mapping on odd elements is just the square of the respective operator in $\End\Lambda$.
\end{Example}
If we modify this example by  considering positive characteristic and introducing different divided powers, the computation seem to be impossible
because the tuple will not be uniform.
We shall introduce another class of triplex Lie algebras with a uniform tuple that allows reasonable computations.

\subsection{Clover specie and Clover restricted Lie algebras}
Now, we fix the following triplex specie called a {\it clover fly} specie.
Let $\Theta_0=\{a_0,b_0,c_0\}$, put $\Theta_n=\{a_n,b_n,c_n\}$,
where the selection of the new generation is as follows
$$a_{n+1}:=a_nb_n,\quad b_{n+1}:=b_na_n,\quad c_{n+1}:=c_na_n,\qquad n\ge 0.$$
Fix integers $\Xi=(S_n,R_n| n\ge 0)$ (the same as in Example~\ref{E2})
and consider formal divided power series ring:
$$R=R(\Xi):=\Big\langle x_0^{(\a_0)}y_0^{(\b_0)}z_0^{(\gamma_0)}\!\!\cdots x_i^{(\a_i)}y_i^{(\b_i)}z_i^{(\gamma_i)}\ \Big |\
0\le\a_i<p^{S_i},\ 0\le \b_i,\gamma_i<p^{R_i},\ i\ge 0\Big\rangle. $$
Unlike examples above and to avoid ambiguity, we denote flies and the respective pivot elements by different letters.
Using notations
$v_i:=v_{a_i}$, $w_i:=v_{b_i}$, $u_i:=v_{c_i}$ and
$x_i:=t_{a_i}$, $y_i:=t_{b_i}$, $z_i:=t_{c_i}$, for all $i\ge 0$,
our algorithm~\eqref{pivot} yields the pivot elements
\begin{equation}\label{pivot-3}
\begin{split}
v_i &=\dd_{x_i}+x_{i}^{(p^{S_{i}}-1)} y_{i}^{(p^{R_{i}}-1)} v_{i+1} ;\\
w_i &=\dd_{y_i}+y_{i}^{(p^{R_{i}}-1)} x_{i}^{(p^{S_{i}}-1)} w_{i+1};\\
u_i &=\dd_{z_i}+z_{i}^{(p^{R_{i}}-1)} x_{i}^{(p^{S_{i}}-1)} u_{i+1};
\end{split}\qquad\qquad i\ge 0.
\end{equation}
We define the {\it clover restricted Lie algebra} $\TT(\Xi):=\Lie_p(v_0,w_0,u_0)$.
\begin{Remark}
Let us draw attention that there is a symmetry between $v_i$ and $w_i$,
while the remaining $u_i$ stay separate because there is {\bf no $\Z_3$-cyclic symmetry} unlike Example~\ref{Example_Q}.
\end{Remark}
\begin{Lemma}\label{L_clover_relations}
Let $\ch K=p>0$ and a tuple $\Xi$ be fixed. Consider the clover restricted Lie algebra $\TT(\Xi)=\Lie_p(v_0,w_0,u_0)$. Then
\begin{enumerate}
\item
$v_i^{p^{S_{i}}} = y_{i}^{(p^{R_{i}}-1)} v_{i+1}$,\quad\
$w_i^{p^{R_{i}}} = x_{i}^{(p^{S_{i}}-1)} w_{i+1}$,\quad\
$u_i^{p^{R_{i}}} = x_{i}^{(p^{S_{i}}-1)} u_{i+1}$, for all $i\ge 0$.
\item
 $ [ w_i^{p^{R_i}-1},v_i^{p^{S_i}}]=v_{i+1}$,\quad
 $ [ v_i^{p^{S_i}-1},w_i^{p^{R_i}}]=w_{i+1}$,\quad
 $ [ v_i^{p^{S_i}-1},u_i^{p^{R_i}}]=u_{i+1}$, for all $i\ge 0$.
\item $v_i,w_i,u_i\in \TT(\Xi)$, $i\ge 0$.
\end{enumerate}
\end{Lemma}
\begin{proof}
Follows from Lemma~\ref{Lrelations} and its Corollary~\ref{Cternary} or
by direct computations.
\end{proof}

\begin{Lemma}
The subalgebra of $\TT(\Xi)$ generated by $v_0,w_0$ is isomorphic to $\LL(\Xi)$ defined by~\eqref{aibip}.
\end{Lemma}
\begin{proof}
We observe that $v_0,w_0$ have the same presentation as $a_0,b_0\in \LL(\Xi)$.
\end{proof}

\subsection{Clover restricted Lie algebras of Quasi-linear Growth}

As a specific case, we construct algebras of quasi-linear growth.
The proof is rather technical and
close to that for two-generated duplex restricted Lie algebras~\cite{Pe17} (Example~\ref{E2} above).
But species of flies having two flies in some generation either have two flies
in all subsequent generations or go extinct.
The goal in introducing the clover specie was to have three flies in each generation,
so that at some moment three flies can produce a wild specie and the constructed Lie algebra
can return to a rather fast intermediate growth.
To this end we extend the duplex specie in a specific way and obtain the clover specie.
On the other hand, our approach allows to use computations of~\cite{Pe17}.
This idea enables us to construct hybrid restricted Lie algebras with an oscillating growth and achieve the lower bound
of the main theorem.

Proofs of the next two theorems are rather technical,
similar to~\cite{Pe17}, and placed in a separate paper~\cite{Pe20clover}.

\begin{Theorem}[\cite{Pe20clover}]\label{Tparam}
Let $K$ be a field, $\ch K=p> 0$, fix $\kappa\in(0,1)$.
There exists a tuple of integers $\Xi_\kappa$ such that
the 3-generated clover restricted Lie algebra $\TT=\TT(\Xi_\kappa)=\Lie_p(v_0,w_0,u_0)$ has the following properties.
\begin{enumerate}
\item
$\gamma_{\TT}(m)= m\exp \big((C+o(1))(\ln m)^\kappa\big)$ as $m\to\infty$, where $C:=2(\ln p)^{1-\kappa}/\kappa^\kappa$;
\item $\GKdim \TT=\LGKdim \TT= 1$;
\item $\Ldim^0 \TT=\LLdim^0 \TT=\kappa$;
\item the growth function $\gamma_\TT(m)$ is not linear;
\item algebras $\TT(\Xi_\kappa)$ for different $\kappa\in(0,1)$ are not isomorphic.
\end{enumerate}
\end{Theorem}
To prove this theorem we consider the tuple: $\Xi_\kappa:=(S_i:=[(i+1)^{1/\kappa-1}], R_i:=1\mid i\ge 0)$.

Denote $\ln^{(q)}(x):=\underbrace{\ln(\cdots\ln}_{q\text{ times}}(x)\cdots)$ and
$\exp^{(q)}(x):=\underbrace{\exp(\cdots\exp}_{q\text{ times}}(x)\cdots)$ for all $q\in\N$.
In comparison with~\cite{Pe17},  algebras with even slower quasi-linear growth are constructed in the next theorem.

\begin{Theorem}[\cite{Pe20clover}]\label{Tparam2}
Let $\ch K=p> 0$, fix $q\in\N$ and $\kappa\in\R^+$.
There exists a tuple of integers $\Xi_{q,\kappa}$ such that
the 3-generated clover restricted Lie algebra $\TT=\TT(\Xi_{q,\kappa})=\Lie_p(v_0,w_0,u_0)$ has the following properties.
\begin{enumerate}
\item
$\gamma_{\TT}(m)= m \big(\ln^{(q)} \!m\big )^{\kappa+o(1)}$ while $m\to\infty$;
\item
$\GKdim \TT=\LGKdim \TT= 1$;
\item $\Ldim^q \TT=\LLdim^q \TT=\kappa$;
\item
the growth function $\gamma_\TT(m)$ is not linear;
\item algebras $\TT(\Xi_{q,\kappa)}$ for different pairs $(q,\kappa)$ are not isomorphic.
\end{enumerate}
\end{Theorem}
Let us describe the tuple used in~\cite{Pe20clover}.
Fix $\lambda:=(\ln p^2)/\kappa$.
Put $\Xi_{q,\kappa}=(S_i,R_i\mid i\ge 0)$, where
$R_i:=1$ for all $i\ge 0$. Also $S_0:=1$ and
$S_n:=[\exp^{(q)}(\lambda (n+2) )]+1-S_0-\cdots- S_{n-1}$ for  $n\ge 1$.

\section{Lie algebra of Wild Drosophilas: Lower bound on Growth}


Fix a prime $p> 0$, and a specie of wild flies $\bar\Theta$, $\bar\Theta_0=\{1,2,\dots,k\}$, $k\ge 3$.
By construction, $k$ uniquely determines $\bar\Theta$.
Now we do not need divided powers, consider the {\it trivial tuple} $\bar S=(S_a=1| a\in \bar \Theta )$,
and we get the ordinary truncated polynomial ring $R=R(\bar\Theta)$.
The pivot elements of the zero generation are denoted as $v_1,\ldots,v_k$.

\begin{Lemma} \label{Lweight_wild_pivot}
Fix $\ch K=p> 0$, a wild specie of flies $\bar \Theta$, where $|\bar\Theta_0|=k\ge 3$.
Assume that the tuple $\bar S$ is trivial.
Denote $\wt \bar\Theta_n:=\wt v_a$, for all $a\in\bar\Theta_n$, and $\mu:=\log_{2p-1}2$. Then
\begin{enumerate}
\item
$ \wt \bar\Theta_n=(2p-1)^n$,\  $n\ge 0$. 
\item $\displaystyle |\bar\Theta_n|> (\theta_k)^{(\wt\bar \Theta_n)^\mu}$,\ $n\ge 0$.
\item For any $\theta_*>\theta_k$  ($\theta_k$ is defined in Lemma~\ref{Lasymp_wild}) there exists $M$ such that
$\displaystyle |\bar\Theta_n|< (\theta_*)^{(\wt\bar \Theta_n)^\mu}$,\ $ n\ge M.$
\end{enumerate}
\end{Lemma}
\begin{proof}
Claim~i)  is a partial case of Lemma~\ref{Luniform}.
Denote $m=\wt\bar\Theta_n$.
By i), $n=\log_{2p-1} m$.
By the lower bound of Lemma~\ref{Lasymp_wild}
$ |\bar\Theta_n|>\theta_k^{2^n}=\theta_k^{2^{\log_{p-1}(m)}}
=\theta_k^{m^{\log_{2p-1}(2)}}=(\theta_k)^{(\wt\bar \Theta_n)^\mu}. $ 
Similarly, we prove claim~iii) using the upper bound iii) of Lemma~\ref{Lasymp_wild}.
\end{proof}

We show that the upper bound above extends to an arbitrary specie with a uniform tuple.
\begin{Corollary}\label{Cweight_wild_pivot}
Fix $\ch K=p> 0$, a specie of flies $\Theta$, where $|\Theta_0|=k\ge 3$, and a uniform tuple $\bar S$.
So, we denote $\wt \Theta_n:=\wt v_a$, for all $a\in\Theta_n$, and $\mu:=\log_{2p-1}2$.
Then for any $\theta_*>\theta_k$ there exists $M$ such that
$\displaystyle |\Theta_n|< (\theta_*)^{(\wt\Theta_n)^\mu}$,\ $ n\ge M.$
\end{Corollary}
\begin{proof}
Remark that $|\Theta_n|\le |\bar \Theta_n|$ for all $n\ge 0$, assuming that the initial generations of both species have $k$ flies.
Using Lemma~\ref{Luniform}, by induction
$\wt(\bar \Theta_{n+1})=(2p-1)\wt(\bar \Theta_n)\le (p^{S_n}+p^{R_n}-1)\wt(\Theta_n)=\wt(\Theta_{n+1})$, $n\ge 0$.
It remains to apply the upper bound of Lemma.
\begin{equation*}
|\Theta_n|\le |\bar \Theta_n|< (\theta_*)^{(\wt\bar \Theta_n)^\mu}\le (\theta_*)^{(\wt\Theta_n)^\mu},\quad n\ge M.
\qedhere
\end{equation*}
\end{proof}

We easily have the following lower bound, but it is not the best possible one.
\begin{Corollary}
Consider the restricted Lie algebra $\LL=\LL(\bar \Theta)=\Lie_p(v_1,\ldots,v_k)$,
where $\bar \Theta$ is a wild specie of flies, the tuple $\bar S$ being trivial.
Set $\mu=\log_{2p-1}2$, then  $0<\mu<1$.
There exists a constant $C>0$ such that we have a lower bound on the growth:
\begin{equation*}
\gamma_{\LL}(x)> \exp\big(C x^{\mu} \big),  \qquad x> 0.
\end{equation*}
\end{Corollary}
\begin{proof}
Fix $x> 0$ and set $n=[\log_{2p-1 }x]$.
By Lemma~\ref{Lweight_wild_pivot}, $\wt(v_a)\le x$ for all $a\in\bar\Theta_n$.
By Lemma~\ref{Lactual_pivot}, the respective pivot elements $\{v_a| a\in\bar\Theta_n\}$ are linearly independent.
Using Lemma~\ref{Lasymp_wild} item~i),
\begin{equation*}
\gamma_{\LL}(x)> |\bar \Theta_n|>(2,33)^{2^n}> (2,33)^{2^{\log_{2p-1}(x)-1}}
=\exp\big(\ln(\sqrt{ 2,33})\, x^{\mu} \big), \qquad x> 0.
\qedhere
\end{equation*}
\end{proof}

Define recursively Lie monomials (we consider they belonging to a free Lie algebra, see~\cite{Ba,BMPZ}):
\begin{align*}
F_1(X_1,X_2)&=[X_1,X_2];\\
F_2(X_1,X_2,X_3,X_4)&=[[X_1,X_2],[X_3,X_4]];\\
F_{n}(X_1,\dots,X_{2^{n}})&=[F_{n-1}(X_1,\dots,X_{2^{n-1}}),F_{n-1}(X_{2^{n-1}+1},\dots,X_{2^{n}})],
\qquad n\ge 3.
\end{align*}
A Lie algebra is {\it solvable} of length $n$ iff it satisfies the identity $F_{n}(X_1,\dots,X_{2^{n}})\equiv 0$~\cite{Ba}.
The symmetric group $\Sym(2^n)$ naturally acts on these monomials. Permutations that permute halves of some of
commutators inside $F_n(X_1,\ldots,X_{2^n})$ add only a sign to this polynomial and constitute a Sylov 2-subgroup of  $\Sym(2^n)$.
Let $\Gamma$  be the set of permutations $\pi\in\Sym(2^n)$ such that:
\begin{equation}\label{sylov}
\begin{tabular}{c}
$\pi(1)<\pi(2), \ \pi(3)<\pi(4)$,\ \dotfill,\ $\pi(2^{n}{-}1)<\pi(2^n)$,\\
$\pi(1)<\pi(3), \ \pi(5)<\pi(7),\ \dotfill,\ \pi(2^{n}{-}3)< \pi(2^{n}{-}1)$,\\
\dotfill\\[2pt]
$\pi(1)<\pi(1{+}2^l), \ \pi(1{+}2{\cdot} 2^l)<\pi(1{+}3{\cdot} 2^l),\ \ldots,\ \pi(1{+}(2^{n-l}\!{-}2){\cdot} 2^l)<\pi(1{+}(2^{n-l}\!{-}1){\cdot} 2^l)$,\\
\dotfill\\[2pt]
$\pi(1)<\pi(1{+}2^{n-1}).$
\end{tabular}
\end{equation}

It is well-known that $\Gamma$ is the coset set of $\Sym(2^n)$ modulo its Sylov 2-subgroup of the following cardinality.
\begin{Lemma}\label{Lfactor_sylov}
Let $\Gamma$ be as above, then
$|\Gamma|=(2^n)!/2^{n(n-1)/2}$.
\end{Lemma}

\begin{Lemma}\label{LSylov_1}
Let $\bar \Theta$ be a specie of wild flies, $F_n$, $\Gamma\subset \Sym(2^n)$ as above, the tuple $\bar S$ being trivial.
Assume that there exist integers $m,n$  and $2^n$ flies $\{a_1,\dots,a_{2^n}\}\subseteq \bar \Theta_m$.
Using respective virtual pivot elements $v_{a_i}$, we get the following linearly independent set
$$
\{F_n(v_{a_{\pi(1)}},v_{a_{\pi(2)}},\ldots, v_{a_{\pi(2^n)}})\mid \pi\in \Gamma\}\subset \bar\LL:=\Lie_p(v_a\mid a\in\bar\Theta).
$$
\end{Lemma}
\begin{proof}
Observe that
$\Theta'':=\mathop{\cup}\limits_{l\ge m}\bar \Theta_l$ is a specie of wild flies generated by generation $\bar \Theta_m$.
For brevity, denote its elements as $\bar\Theta_m=\{1,\ldots,2^n,\ldots\}$.
Consider a subspecie $\Theta'\subset \Theta''$ generated by $\Theta_0':=\{1,\ldots,2^n\}$,
imposing a selection according to~\eqref{sylov}. Let $\lessdot$ denote the natural order on $\Theta_0'$.
Assume that the generation $\Theta_l'$, $l\ge 0$, is constructed.
Define a {\it temporary} non-genealogical partial order $\lessdot$ on $\Theta_l'$.
Namely, assuming that $a,b\in\Theta_l'$ together have any $\Theta_0'$-letter  at most once,
we compare $a,b$ using their first $\Theta_0'$-letters with the natural order.
Let $a \lessdot b$, then we set $ab\in\Theta_{l+1}'$, which is again "multilinear" in $\Theta_0'$,
thus constructing $\Theta_{l+1}'$.
By construction,  $\Theta'_n$ consists of the second rows of permutations $\Gamma$ introduced above~\eqref{sylov}.
Remark that $\Theta'_{n+1}=\emptyset$.

By Lemma~\ref{Lder_subspecies}, we have an epimorphism:
$\phi: \Lie_p(v_a|a\in\Theta'')\twoheadrightarrow \Lie_p(v_a|a\in \Theta')$.
Recall its action on  a virtual pivot element for $\Lie_p(\Theta'')$, which is also a virtual pivot element for $\bar\LL$.
Consider the infinite expansion~\eqref{expansion}, as a result of action, terms
with derivations for flies $a\in \Theta''\setminus \Theta'$ should disappear.

Denote the virtual pivot elements for $\LL(\Theta')$ as $w_a$, $a\in\Theta'$.
Fix $a,b\in \Theta'_0=\{1,\ldots,2^n\}$, $a \lessdot b$.  Then
\begin{align*} 
w_a=\dd_a+t_a^{(p-1)} \sum_{c\in \Theta_0'}  t_c^{(p-1)} w_{ac},\qquad
w_b=\dd_b+t_b^{(p-1)} \sum_{d\in \Theta_0'}  t_d^{(p-1)} w_{bd}.
\end{align*}
They are $\phi$-images of virtual pivot elements of $\LL(\Theta'')$ and $\LL(\Theta)$,
the latter being of generation $m$.
Denote $T_l=\{t_a^{(\a)}| a\in \Theta_l', 0{\le} \a{<}p \}$, $l\ge 0$.
We shall compute modulo ideals
$\mathcal W_n=\mathcal W_n(\bar \Theta)\triangleleft\mathcal W(\bar \Theta) $, $n\ge 0$,
(Lemma~\ref{LcalW}), omitting $\pmod{*}$  for brevity.
By construction of $\Theta'$, we have $ba\notin \Theta_1'$.
Thus, using Lemma~\ref{Lcomm_pivot}, we get a head with one term:
\begin{align*}
[w_a,w_b]&=-t_a^{(p-1)}t_b^{(p-2)} w_{ab}
+t_a^{(p-1)}t_b^{(p-1)} \sum_{c,d\in \Theta_0'} t_c^{(p-1)}t_d^{(p-1)} [w_{ac},w_{bd}]\\
&\equiv-t_a^{(p-1)}t_b^{(p-2)} \Big (w_{ab} +T_0^{2p-1}{\mathcal W}_{m+2}\Big),
\end{align*}
because $[w_{ac},w_{bd}]\in {\mathcal W}_{m+2}$ (Corollary~\ref{Ccomm_pivot}).
Let $c,d\in\Theta_0'\setminus \{a,b\}$ and $c\lessdot d$.  By computations above,
$$ [w_c,w_d]\equiv-t_c^{(p-1)}t_d^{(p-2)} \Big(w_{cd}+ T_0^{2p-1}{\mathcal W}_{m+2}\Big). $$
We continue the process, assuming that $a \lessdot c$:
\begin{align*}
&[[w_a,w_b],[w_c,w_d]]
\equiv
-\underbrace{t_a^{(p-1)}t_b^{(p-2)}t_c^{(p-1)}t_d^{(p-2)}}_{T_0}\Big(
\underbrace{t^{(p-1)}_{ab}t_{cd}^{(p-2)}}_{T_1} \Big( w_{abcd}
+  T_1^{2p-1}{\mathcal W}_{m+3}\Big) + T_0^{2p-1}{\mathcal W}_{m+2}\Big).
\end{align*}
We continue and get by induction the following.
Let $(a_1,\ldots,a_{2^n})$ be the second row of $\pi\in\Gamma$.
Then
\begin{multline} 
F_n(w_{a_1},\ldots w_{a_{2^n}})
\equiv -\underbrace{t_{a_1}^{(*)}t_{a_2}^{(*)}\!\cdots t_{a_{2^n}}^{(*)}}_{T_0}\Big(
\underbrace{t_{a_1a_2}^{(*)}\!\!\!\cdots t_{a_{2^n-1}a_{2^n}}^{(*)}}_{T_1} \Big(
\underbrace{t_{a_1a_2a_3a_4}^{(*)}\!\!\!\cdots t_{a_{2^n-3}a_{2^n-2}a_{2^n-1}a_{2^n}}^{(*)}}_{T_2}\Big( \qquad\\
\qquad \cdots\!\Big( \underbrace{t_{a_1\ldots a_{2^{n-1}}}^{(*)} t_{a_{2^{n-1}\!+1}\ldots a_{2^n} }^{(*)}}_{T_{n-1}}
\Big( w_{a_1a_2\ldots a_{2^n}}+ T_{n-1}^{2p-1}{\mathcal W}_{m+n+1}\Big)+T_{n-2}^{2p-1}{\mathcal W}_{m+n}\Big)\\[-9pt]
\cdots + T_{2}^{2p-1}{\mathcal W}_{m+4}\Big)
+ T_{1}^{2p-1}{\mathcal W}_{m+3}\Big)
+ T_{0}^{2p-1}{\mathcal W}_{m+2}\Big). \label{fnw}
\end{multline}
In these computations, we used the following observation:
in~\eqref{fnw}, $w_{a_1a_2\ldots a_{2^n}}$ contains only derivations $\dd_a$, $a\in\bar \Theta_l$ such that $l\ge m+n$,
while $T_0,\ldots,T_{n-1}$ are of previous generations $m,\ldots, m+n-1$ (in terms of $\bar\Theta$).
Thus, in further computations the factors $T_{j}^{2p-1}$, $j=0,\ldots,n-1$, remain untouched.
Either group of products in $T_0,\ldots,T_{n-1}$  above
has divided powers with alternating powers $(*)$ either $p-2$ or $p-1$.

Consider either of groups of variables belonging to $T_0,T_1,\ldots,T_{n-1}$ above,
observe that they correspond to all parents of $a_1a_2\ldots a_{2^n}$ in the specie $\Theta'$.
In particular, all variables in the product above are different.
Thus, \eqref{fnw} yields a unique pure Lie monomial with the same factors at $\dd_{a_1a_2\ldots a_{2^n}}$.
The remaining terms of~\eqref{fnw} can yield
pure Lie monomials with derivations of the same generation as the first term,
but such terms originate from $T_i^{2p-1}{\mathcal W}_{i+m+2}$, $i\in\{0,\ldots,n-1\}$, and have
a bigger total power in $T_i$  than the respective power of the first term.
Hence, the elements~\eqref{fnw} corresponding to all permutations $\Gamma$ are linearly independent.

Let $v_{a_i}$, $i=1,\ldots,2^n$ be the virtual pivot elements of $\bar\LL$. By above, $\phi(v_{a_i})=w_{a_i}$, $i=1,\ldots,2^n$. Thus,
$$
\phi(F_n(v_{a_{\pi(1)}},\ldots, v_{a_{\pi(2^n)}}))=F_n(w_{a_{{\pi(1)}}},\ldots v_{a_{\pi(2^n)}}),\quad \pi\in \Gamma.
$$
Since the set~\eqref{fnw}~is linearly independent, we conclude that the required set is linearly independent.

We just proved a linear independence of a set determined in terms of virtual pivot elements that in general do not belong to $\LL$.
Thus, we proved a linear independence of a set belonging to
a bigger infinitely generated Lie algebra  $\bar\LL=\Lie_p(v_a| a\in \bar \Theta)$.
\end{proof}

\begin{Lemma}\label{LSylov_2}
Let $\bar \Theta$ be a specie of wild flies, $\LL=\LL(\bar\Theta)$, the tuple $\bar S$ being trivial.
Suppose that there exist integers $m,n$ and $2^n$ flies $\{a_1,\dots,a_{2^n}\}\subseteq \bar \Theta_m$.
Denote $q:=\wt \bar \Theta_m$. Then
$$
\dim \LL_{q 2^n}\ge \frac {(2^n)!}{2^{n(n-1)/2}}.
$$
\end{Lemma}
\begin{proof}
We use notations $F_n$, and $\Gamma\subset \Sym(2^n)$ of Lemma~\ref{LSylov_1}.
Consider the respective {\it actual} pivot elements $\bar v_{a_i}\in \LL_q$, $i=1,\ldots,2^n$.
Since $v_{a_i}$ have derivations of generation at least $m$, using Lemma~\ref{Lactual_pivot}, we get
\begin{align}\nonumber
\bar v_{a_i}&\equiv v_{a_i} \pmod{ \sum_{d \sqsupset a_i} t_d^{(1)} {\mathcal W}_{\gen d +2}},\qquad i=1,\ldots,2^n;\\[-3pt]
F_n(\bar v_{a_{\pi(1)}},\ldots, \bar v_{a_{\pi(2^n)}})&\equiv F_n(v_{a_{\pi(1)}},\ldots, v_{a_{\pi(2^n)}})
\pmod{\!\!\!\! \sum_{d\in\bar \Theta_{0\ldots m-1}}\!\!\! t_d^{(1)} {\mathcal W}_{\gen d +2}},\qquad \pi\in\Gamma.
\label{formulaFn}
\end{align}
Since $v_{a_i}$ have variables of generation at least $m+1$,
the same applies to $F_n(*)$  of Lemma~\ref{LSylov_1}. Thus, \eqref{formulaFn}
are linearly independent elements of weight $q2^n$ and cardinality given by Lemma~\ref{Lfactor_sylov}.
\end{proof}

\begin{Theorem}\label{Twild_low}
Fix $\ch K=p> 0$, a specie of wild flies $\bar \Theta$, where $|\bar\Theta_0|=k\ge 3$, the  trivial  tuple $\bar S$, and
$\LL=\LL(\bar\Theta)=\Lie_p(v_1,\ldots,v_k)$ the respective restricted Lie algebra.
Denote $\lambda:=\log_2(p-\frac 12)$. There exist positive constants $C_3=C_3(p,k)$, $n_3$ such that 
$$
\gamma_{\LL}(n)\ge \exp\Big( C_3 \frac{n} {(\ln n)^{\lambda}} \Big),\qquad n\ge n_3.
$$
\end{Theorem}
\begin{proof}
Denote $\mu=\log_{2p-1}2$. Then $\frac 1\mu=1+\lambda$. Let $N$ be a large natural number. Set
\begin{equation}\label{m000}
q_0=(\log_{\theta_k} N)^{1/\mu}.
\end{equation}
Choose the smallest integer of the form $q=(2p-1)^m$, where $m\in\N$, and such that $q_0\le q$. Then
\begin{equation}\label{log_bound}
\log_{2p-1} q_0 \le m< \log_{2p-1} q_0 +1.
\end{equation}
Fix integers $m,q$.
By Lemma~\ref{Lweight_wild_pivot},  $q=\wt\bar \Theta_m$ is the weight of all pivot elements of generation $m$.
Using~\eqref{m000} and~\eqref{log_bound}, we get
\begin{equation}\label{lower_wt}
(\log_{\theta_k} N)^{1/\mu}\le  q=(2p-1)^m=\wt\bar \Theta_m < (2p-1)(\log_{\theta_k} N)^{1/\mu}.
\end{equation}
Using claim~ii) of Lemma~\ref{Lweight_wild_pivot} and the lower bound~\eqref{lower_wt},
we get a lower bound on the number of flies of generation $m$:
\begin{equation*}
|\bar \Theta_{m}|>{\theta_k}^{(\wt \bar \Theta_{m})^\mu}\ge N.
\end{equation*}
Lemma~\ref{LSylov_2} requires $2^n$ flies in $\bar\Theta_m$, we choose a more appropriate number $M=2^n< N$ as follows. Set
\begin{align}
n:=\Big[\log_2 \Big(\frac{N}{(2p-1)(\log_{\theta_k} N)^{1/\mu}}\Big)\Big];
\nonumber\\
\log_2 \Big(\frac{N}{(2p-1)(\log_{\theta_k} N)^{1/\mu}}\Big)-1 <  n \le \log_2 \Big(\frac{N}{(2p-1)(\log_{\theta_k} N)^{1/\mu}}\Big);
\label{log2} \\
\frac{N}{2(2p-1)(\log_{\theta_k} N)^{1/\mu}} <  M:=2^n\le \frac{N}{(2p-1)(\log_{\theta_k} N)^{1/\mu}}.
\label{lower_n}
\end{align}
Choose $M=2^n$ flies in $\bar\Theta_m$.
Weight of "multilinear" $M$-fold products of the respective pivot elements
is equal to $qM=q2^n$ and does not exceed $N$ by the upper bounds in~\eqref{lower_wt} and~\eqref{lower_n}.
Applying Lemma~\ref{LSylov_2},
\begin{equation}\label{gamm_l}
\gamma_\LL (N)> \dim \LL_{q 2^n}\ge \frac{M!}{2^{n(n-1)/2}}.
\end{equation}
Using~\eqref{log2}, we see that the denominator  in~\eqref{gamm_l} is not essential
$$
2^{n(n-1)/2}=\exp\big(n^2(C'+o(1))\big)=\exp\big((\ln N)^2(C''+o(1))\big)
=\exp\Big(\frac{N} {(\ln N)^{\lambda}}o(1)\Big),\qquad N\to \infty.
$$
Using the Stirling formula and the lower bound~\eqref{lower_n}, we have
\begin{align}
\nonumber
M!&>\bigg(\frac M e\bigg)^M=\exp(M(\ln M-1))
>\exp\bigg(\frac{N \big(\ln N- \frac 1\mu \ln\log_{\theta_k}  N- \ln (4p-2)\big)}{2(2p-1)(\log_{\theta_k} N)^{1/\mu}}  \bigg)\\
&=\exp\bigg( \frac{N} {(\ln N)^{\lambda}} \Big(\frac{ (\ln {\theta_k})^{1+\lambda}}{4p-2}+o(1)\Big)\bigg),\qquad N\to \infty.
\label{mbiggl}
\end{align}
Fix any positive constant $C_3<(\ln {\theta_k})^{1+\lambda}/(4p-2)$. 
Now, the result follows from~\eqref{gamm_l} and~\eqref{mbiggl}.
\end{proof}

\section{Growth of Lie algebra of Wild Drosophilas}

The goal of this section is to finish the proof of the following result specifying
an intermediate growth of the restricted Lie algebra corresponding to a specie of wild flies.
The difficulty is that we have no clear bases.

\begin{Theorem}\label{Twild_growth}
Fix $\ch K=p> 0$, a specie of wild flies $\bar \Theta$, where $|\bar\Theta_0|=k\ge 3$, the  trivial  tuple $\bar S$, and
$\LL=\LL(\bar\Theta)=\Lie_p(v_1,\ldots,v_k)$ the respective restricted Lie algebra.
Denote $\lambda:=\log_2(p-\frac 12)$. There exist positive constants $C_3=C_3(p,k)$, $C_4=C_4(p,k)$, $n_0$ such that
$$
\exp\Big(C_3\frac{n} {(\ln n)^{\lambda}} \Big)
\le \gamma_{\LL}(n)\le
\exp\Big(C_4\frac{n} {(\ln n)^{\lambda}} \Big),\qquad n\ge n_0.
$$
\end{Theorem}
\begin{proof}
The lower and upper bounds are proved in Theorem~\ref{Twild_low} and~Theorem~\ref{Twild_growth2} below.
\end{proof}

For simplicity,
a reader  can think that in the next result we are dealing with a specie of wild flies $\bar\Theta$ and the trivial tuple $\bar S$.
A more general setting  is needed to get the total upper bound of claim~iii) in main Theorem~\ref{T_main}.
Namely, we assume that a specie $\Theta$ is a hybrid of wild and clover species,
namely, $\Theta$ has interchanging wild segments  with the trivial tuple, and clover segments with their tuples,
see section~\ref{Shybrid}.

\begin{Theorem}\label{Twild_growth2}
Fix $\ch K=p> 0$, a hybrid specie $\Theta$, where $|\Theta_0|=k\ge 3$, and a uniform  tuple $\bar S$.
Let $\LL=\LL(\Theta,\bar S)=\Lie_p(v_1,\ldots,v_k)$ be the respective restricted Lie algebra.
Denote $\lambda:=\log_2(p-\frac 12)$. There exist positive constants $C_4=C_4(p,k)$, $n_0$ such that
$$ \gamma_{\LL}(n)\le \exp\Big(C_4\frac{n} {(\ln n)^{\lambda}} \Big),\qquad n\ge n_0. $$
\end{Theorem}
\begin{proof}
Set $\mu:=\log_{2p-1}2$. Then $1/\mu=1+\lambda$.
Fix a number $\theta_*>\theta_k$.
By Corollary~\ref{Cweight_wild_pivot}, there exists $M_0$ such that
\begin{equation}\label{theta_upper}
|\Theta_m|\le \theta_*^{(\wt\Theta_m)^\mu}, \qquad m\ge M_0.
\end{equation}
Fix a large natural number $N$.
Consider a function $f(x):=x\theta_*^{x^\mu}$, $x\ge 0$.
Let $y\in\R^+$ be the unique solution of the equation
\begin{equation}\label{function}
f(y)=y\cdot \theta_*^{y^\mu}=N.
\end{equation}
Rewrite this equation as $y=\phi(y)$ where $\phi(x):=(\log_{\theta_*} (N/x))^{1/\mu}$.
Consider two iterations:
\begin{align}
\label{iteration1}
y_0:&=\phi(1)=(\log_{\theta_*} N)^{1/\mu};\\
\label{iteration2}
y_1:&=\phi(y_0)=\Big(\log_{\theta_*} \Big(\frac{N}{(\log_{\theta_*} N)^{1/\mu}}\Big)\Big)^{1/\mu}
=\Big(\log_{\theta_*}{N} -\frac 1{\mu}\log_{\theta_*}\log_{\theta_*} N\Big)^{1/\mu}.
\end{align}
We check that
\begin{align*}
f(y_0)&=(\log_{\theta_*} N)^{1/\mu}\cdot N>N;\\
f(y_1)&=\Big(\log_{\theta_*}{N}
-\frac 1{\mu}\log_{\theta_*}\log_{\theta_*} N\Big)^{1/\mu}\!\!\!\cdot \frac{N}{(\log_{\theta_*} N)^{1/\mu}}<N.
\end{align*}
Since function~\eqref{function} is increasing, we get $y_1<y<y_0$.
Next, we choose integers $m,q$ such that $q=\wt \Theta_m\le y_1$ and $m$ is maximal.
By Lemma~\ref{Luniform}, we have $y_1\ge q=\wt\Theta_m\ge (2p-1)^m$. Hence,
\begin{equation}\label{mset}
m\le \log_{2p-1} y_1.
\end{equation}

A) Consider the case that $\Theta_m$ belongs to a wild segment, in particular $S_m=1$.
(This case includes the most important situation of a wild specie with the trivial tuple.
Case B) below is needed only to prove that the growth function stays inside the wide angle, i.e. claim iii) of Theorem~\ref{T_main}).
By Lemma~\ref{Luniform}, $\wt \Theta_{m+1}=(2p-1)\wt \Theta_m>y_1$, the latter follows by maximality of $m$.
This observation explains the lower bound in relations between the numbers fixed above
\begin{align}\label{qqq}
\frac {y_1}{2p-1}< q=\wt\Theta_m\le y_1<y<y_0.
\end{align}
We use~\eqref{theta_upper}, \eqref{qqq}, and~\eqref{function},
\begin{equation}\label{theta_m_up}
|\Theta_m|\le \theta_*^{(\wt \Theta_m)^\mu}=\theta_*^{q^\mu}<\theta_*^{y^\mu}
=\frac Ny < N.
\end{equation}

Now we describe an approach to evaluate the growth of $\LL$.
Consider the generators $v_1,\ldots,v_k$ of $\LL$ and present them as follows.
Recall that $\Theta_0=\{1,\ldots,k\}$.
By cutting presentation~\eqref{construction} 
at step $m$, we get
\begin{equation}\label{generation0}
v_i= \sum_{i\succeq   d\in \Theta_{0\ldots {m-1}}}\!\!\!
\Big(\prod_{c\vdash d} t_c^{(p^{S_c}-1)}\Big) \dd_d
+\sum_{i\succ d\in  \Theta_m}\!\!
\Big(\prod_{c\vdash d} t_c^{(p^{S_c}-1)}\Big)v_d,\qquad i=1,\dots,k.
\end{equation}

The first terms belong to a bigger set of pure Lie monomials
$D:=\{\delta=\prod_{c> d} t_c^{(*)} \dd_d\mid d\in \Theta_{0\ldots {m-1}}\}$
while terms of the second sum belong to
$\Omega:=\{\omega=\prod_{c> d} t_c^{(*)} v_d\mid d\in \Theta_{m}\}$
(where (*) denote all possibilities).
Using these observations, we write
right-normed $N$-fold products of the generators~\eqref{generation0} as follows:
\begin{equation}\label{n-fold}
w=[v_{i_1},\ldots,v_{i_N}]=\!\!\sum_{\delta_{*}\in D} [\delta_{j_1},\ldots,\delta_{j_N}]+\!\!
\sum_{1\le P\le N} 
\Big[ [\delta_{*}, \ldots, \delta_{*},\omega_{i_1}],[\delta_{*}, \ldots, \delta_{*},\omega_{i_2}],\ldots,
 [\delta_{*}, \ldots, \delta_{*},\omega_{i_P}] \Big],
\end{equation}
where $\omega_*\in \Omega$, $\delta_*\in D$ are terms in~\eqref{generation0}.
This formula is easily proved by induction on $N$.

By the nature of pure Lie monomials,
$\delta= [\delta_{j_1},\ldots,\delta_{j_N}]=\lambda \prod_{c> d} t_c^{(*)} \partial_d$, where $\lambda\in K$
and $d$ is the maximal among flies of $\delta_{j_i}$.
Up to a scalar, $\delta$ belongs to $D$ and
\begin{equation}\label{wnbound}
\wt (\delta)\le \wt(\Theta_{m-1})<\wt (\Theta_{m})=q<y_0=(\log_{\theta_*} N)^{1/\mu},
\end{equation}
by~\eqref{qqq} and~\eqref{iteration1}.
On the other hand, by~\eqref{n-fold}, $\wt(\delta)=\wt(w)=N \nless(\log_{\theta_*} N)^{1/\mu}$.
This contradiction proves that the first summand in~\eqref{n-fold} is trivial.

Consider actions of pure Lie monomials $\delta_*$ on $\omega_{i_*}\in\Omega$ in~\eqref{n-fold}.
A derivation decreases degree of a variable corresponding to an ancestor of the fly of the pivot element,
optionally, adding variables corresponding to senior ancestors, yielding  again an element of $\Omega$.

Thus, $N$-fold products $w$ of the generators (see~\eqref{n-fold}) are expressed via the following elements.
Consider $\omega_i=\prod_{c> d_i} t_c^{(*)} v_{d_i}\in\Omega$, where $d_i\in\Theta_m$. We obtain elements
\begin{align}
\nonumber
\tilde w&=[\omega_{1},\omega_{2},\ldots,\omega_{P}]=
\Big[\prod\limits_{c> d_1} t_c^{(*)} v_{d_1},\prod\limits_{c> d_2} t_c^{(*)} v_{d_2},\ldots,
\prod\limits_{c> d_P} t_c^{(*)} v_{d_P}\Big]\\
&=\Big(\nu\!\!\! \prod_{c\in\Theta_{0\ldots m-1}}\!\!\! \!\!t_c^{(*)}\Big)[v_{d_1},v_{d_2},\ldots,v_{d_P}],
\qquad \nu\in K, \qquad 1\le P\le N,\quad d_{i_j}\in  \Theta_m.
\label{razb}
\end{align}

It remains to evaluate the number of monomials~\eqref{razb}.
We use~\eqref{mset}, \eqref{qqq}, and~\eqref{function}
\begin{align}\label{etheta2}
\theta_*^{2^{m-1}}
\le \theta_*^{2^{\log_{2p-1} (y_1)-1}}=\theta_*^{(y_1)^{\log_{2p-1}(2)}\!/2}
=\theta_*^{y_1^{\mu}/2}< \sqrt{\theta_*^{y^\mu}}= \sqrt{N/y} <\sqrt{N}.
\end{align}
By Lemma~\ref{Lasymp_wild},
$|\Theta_n|\le |\bar \Theta_n|\le \theta_*^{2^n}$, for sufficiently large $n\ge n_0$.
We use~\eqref{etheta2}
\begin{align}\label{sum_theta_m}
\sum_{j=0}^{m-1}|\Theta_j|\le C'+\sum_{j=n_0}^{m-1} \theta_*^{2^j}
\le C'+\theta_*^{2^{m-1}}\sum_{i=0}^\infty (2,33)^{-i} < 1,8\cdot\theta_*^{2^{m-1}} <2\sqrt{N}.
\end{align}
For any variable $t_c^{(\a_c)}$ its power has $p^{S_c}$ possibilities.
Recall that a uniform tuple $\bar S$ yields disjoint decompositions $\Theta_n=\Theta_n'\cup \Theta_n''$,
the parts having tuple values $S_c=S_n$ and $S_c=R_n$, for $n\ge 0$.
We evaluate the number of possibilities for the product of divided variables in~\eqref{razb},
using Corollary~\ref{C}, \eqref{wnbound}, and~\eqref{sum_theta_m}:
\begin{align}
\prod_{c\in\Theta_{0\ldots m-1}} p^{S_c}=\prod_{j=0}^{m-1}p^{S_j|\Theta_j'|+R_j|\Theta_j''|}
< p^{(S_0+R_0+\cdots+ S_{m-1}+R_{m-1})  (|\Theta_0|+\cdots +|\Theta_{m-1}|) }\nonumber\\
<(\wt\Theta_m)^{2(|\Theta_0|+\cdots +|\Theta_{m-1}|)}
\le \exp\Big(\frac 4\mu \ln\log_{\theta_*} N \cdot \sqrt N \Big),
\label{num_poly}
\end{align}
this number is not changing the expected upper bound.

Let us evaluate $P$ in~\eqref{razb}.
We estimate weight of the first factor in~\eqref{razb} using Lemma~\ref{Luniform}, and~\eqref{sum_theta_m}:
\begin{align}
\bigg |\wt\Big(\!\!\sum_{c\in\Theta_{0\ldots m-1}}\!\!\!\! \!\!t_c^{(*)}\Big)\bigg|
\le \sum_{c\in\Theta_{0\ldots m-1}} (p^{S_c}-1)\wt v_c
< \sum_{j=0}^{m-1} (p^{S_j}+p^{R_j}-1) \wt \Theta_j  |\Theta_j| \nonumber \\
= \sum_{j=0}^{m-1} \wt (\Theta_{j+1}) |\Theta_j|
< \wt \Theta_{m } \sum_{j=0}^{m-1}|\Theta_j|
<  2 q \sqrt N.
\label{weight_poly}
\end{align}
We evaluate weight of~\eqref{n-fold} and~\eqref{razb} using~\eqref{weight_poly}:
\begin{align*}
N=\wt w=\wt\tilde w = P\wt(\Theta_m)+\wt\Big(\!\!\sum_{c\in\Theta_{0\ldots m-1}}\!\!\!\! \!\!t_c^{(*)}\Big)
> Pq- 2q\sqrt{N}.
\end{align*}
Using this relation, lower bound in~\eqref{qqq}, and~\eqref{iteration2} we get an upper bound on $P$ in~\eqref{razb}:
\begin{align}\nonumber
P&< \frac N q+2\sqrt{N}\le \frac{(2p-1)N}{y_1}+2\sqrt{N}
 =\frac{(2p-1)N}{\big(\log_{\theta_*}{N} -\frac 1{\mu}\log_{\theta_*}\log_{\theta_*} N\big)^{1/\mu}} +2\sqrt{N}\\
&=N \frac{2p-1+o(1)}{(\log_{\theta_*}{N})^{1/\mu}},
\qquad\quad N\to \infty.
\label{upperP}
\end{align}
For any fixed $P$ we evaluate the number of commutators in~\eqref{razb}  using~\eqref{upperP} and~\eqref{theta_m_up}:
\begin{align}\nonumber
|\Theta_m|^{P}&=\exp(P\ln |\Theta_m|)
\le \exp\bigg( N \frac{2p-1+o(1) }
{(\log_{\theta_*}{N})^{1/\mu}}\ln N\bigg)\\
&=\exp\Big(\frac {N} {(\ln{N})^{\lambda} } \big((2p-1)(\ln\theta_*)^{1+\lambda}+o(1)\big) \Big),\qquad N\to \infty.
\label{wild_upper}
\end{align}
Now it remains to multiply by the number of divided power factors in~\eqref{razb} evaluated as~\eqref{num_poly}
and sum over possible $P\in \{1,\ldots,N\}$, we just multiply by $N$.
Actually,~\eqref{wild_upper} is an upper bound on the growth
of an algebra obtained by using only Lie brackets for the generators. Denote~\eqref{wild_upper} by $g(N)$.
The growth of $\LL$, which is a $p$-hull of that algebra, is evaluated as $\tilde g(n)=\sum_{j=0}^\infty g(p^{-j}n)$,
where the summation is over integer arguments of $g(n)$ only, instead we can roughly multiply by $\log_pN$.
The asymptotic~\eqref{wild_upper} remains the same.
By choosing any constant $C_4>(2p-1)(\ln\theta_*)^{1+\lambda}$, the desired upper bound follows from the asymptotic~\eqref{wild_upper}.

B) Now consider the case that $\Theta_m$ belongs to a clover segment.
We evaluate the number of products~\eqref{razb} repeating a simplified version of the approach above.
Let $V_m:=\{v_m,w_m,u_m\}$ denote the pivot elements for generation $\Theta_m$.
Then $\Theta_{m+1}$ has 3 pivot elements denoted as $V_{m+1}:=\{v_{m+1},w_{m+1},u_{m+1}\}$.
Recall that $R_m=1$, using notations~\eqref{pivot-3}, we rewrite elements of $V_m$ (i.e. pivot elements of generation $\Theta_m$) as
\begin{equation}\label{pivot-4}
\begin{split}
v_m &=\dd_{x_m}+x_{m}^{(p^{S_{m}}-1)} y_{m}^{(p-1)} v_{m+1} ;\\
w_m &=\dd_{y_m}+y_{m}^{(p-1)} x_{m}^{(p^{S_{m}}-1)} w_{m+1};\\
u_m &=\dd_{z_m}+z_{m}^{(p-1)} x_{m}^{(p^{S_{m}}-1)} u_{m+1}.
\end{split}
\end{equation}
Temporarily denote elements of $V_{m+1}$ also as $\{b_1,b_2,b_3\}$.
The commutators in~\eqref{razb} are commutators of the elements ~\eqref{pivot-4}.
After substitution, $N$-fold commutators in the generators of $\LL$   
are expressed via
\begin{equation}\label{reduc}
\tilde{\tilde w}
= \Big(\prod_{c\in\Theta_{0\ldots m-1}}\!\!\! \!\!t_c^{(*)}\Big)
\Big(\sum_{\substack{0\le \a <p^{S_m},\ 0\le \b,\gamma<p \\ 1\le P'\le P }} x_m^{(\a)} y_m^{(\beta)} z_m^{(\gamma)}\Big)
[b_{c_1},\ldots, b_{c_{P'}}],\qquad b_{i_j}\in V_{m+1}.
\end{equation}
Nonessential estimate~\eqref{num_poly} on the number of the first factors in~\eqref{reduc} remains the same.
By Lemma~\ref{Luniform} ,
\begin{equation}\label{thetam1}
\wt\Theta_{m+1}=(p^{S_m}+p-1)\wt \Theta_m.
\end{equation}
Using~\eqref{thetam1}, we evaluate weight of the second factor in~\eqref{reduc}
\begin{equation}\label{wxm}
\big|\wt(x_m^{(\a)} y_m^{(\beta)} z_m^{(\gamma)})\big|\le
\wt\Theta_m (\a+2(p-1)) < q(\a+2p).
\end{equation}
Using~\eqref{weight_poly}, \eqref{thetam1}, \eqref{wxm}, and that $P'\ge 1$, we evaluate weight of~\eqref{reduc}
\begin{align*}
N&= \wt (\tilde{\tilde w})> P'\wt(\Theta_{m+1})- q(\a +2p)-2q\sqrt N\\
&\ge q(p^{S_m}+p-1)-q(\a+2p)-2q\sqrt N =q((p^{S_m}-1-\a) -p-2\sqrt N);\\
\a &> p^{S_m}-1-p-2\sqrt N-\frac{N}q  \ge p^{S_m}-1-(p+3N).
\end{align*}
Since $0\le\a<p^{S_m}$,  number of second factors in~\eqref{reduc} is bounded by  a small number $(p+3N)p^2$.

It remains to evaluate the number of different commutators in~\eqref{reduc}. Now we use a rough estimate:
\begin{align*}
&\big|\wt(x_m^{(\a)} y_m^{(\beta)} z_m^{(\gamma)})\big|\le \wt\Theta_m (p^{S_m}-1+2(p-1))
< 2\wt \Theta_m (p^{S_m}+p-1)= 2\wt\Theta_{m+1}.
\end{align*}

Estimate~\eqref{weight_poly} on weight of the first factor in~\eqref{reduc} remains the same,
where we use $q=\wt\Theta_m<\wt\Theta_{m+1}$.
Thus, we evaluate weight of~\eqref{reduc} as:
\begin{align*}
N&=  \wt \tilde{\tilde w}\ge P'\wt \Theta_{m+1}- 2\wt \Theta_{m+1}\sqrt N -2\wt\Theta_{m+1};\\
P'&\le  \frac {N}  {\wt \Theta_{m+1}} + 2\sqrt N+2.
\end{align*}
By the choice of $m$, we have $\wt(\Theta_{m+1})\ge y_1$, the latter is given by~\eqref{iteration2}.
Continuing our estimates,
$$
P'\le \frac{N}{(\log_{\theta_*}{N} -\frac 1{\mu}\log_{\theta_*}\log_{\theta_*} N)^{1/\mu}}+ 2\sqrt N+2
=\big((\ln \theta_*)^{1/\mu}{+}o(1)\big) \frac{N}{(\ln {N})^{1/\mu}},\quad N\to\infty.
$$

The commutators in~\eqref{reduc} have length $P'$ in $|\Theta_{m+1}|=3$ letters, we evaluate their number as
$$
3^{P'}\le \exp\Big( \big(\ln 3(\ln \theta_*)^{1/\mu}{+}o(1)\big) \frac{N}{(\ln N)^{1/\mu}}  \Big),\qquad N\to\infty.
$$
Since $1/\mu= \lambda+1$, in this case we get even a smaller estimate than needed.
\end{proof}

\section{Phoenix Lie algebras: construction and Oscillating Growth}\label{SPhoenix}
\label{Shybrid}

\subsection{Cutting off senior generations of flies}
First, we study a general change of the growth of a drosophila Lie algebra after cutting off finitely many senior generations of flies.

Fix $\ch K=p> 0$.
Let $\Theta=\cup_{j=0}^\infty \Theta_j$ be a specie of flies, where
$\Theta_0=\{1,\ldots,k\}$, $k\ge 3$,
$\bar S=(S_a\in\N | a\in\Theta)$ a tuple, and
$L=\LL(\Theta,\bar S)=\Lie_p(\Theta_0)$ the respective restricted Lie algebra.
We need a technical assumption that the tuple $\bar S$ is uniform.
Recall that we assume that all pivot elements of a fixed generation $\{v_a|a\in\Theta_n\}$
are of the same weight, denoted as $\wt(\Theta_n)$ (below, we shall need this assumption only for $n\ge M$).
We have a grading $L=\mathop{\oplus}\limits_{n=1}^\infty L_n$ by degree in the generators $V:=\{v_1,\ldots,v_k\}$
corresponding to the zero generation $\Theta_0$ (Theorem~\ref{Tgraded}).
Recall that $\gamma_L(V,n)=\sum_{j=1}^n \dim L_j$, for $n\ge 1$.

Fix $M\in\N$ and define $\Theta':=\cup_{j\ge M}\Theta_j$, this is a specie generated by $\Theta'_0=\Theta_M$
(remark that $\Theta'\subset\Theta$ is not a subspecie  because ancestors of $\Theta_M$  do not belong to $\Theta'$).
Consider the restriction tuple $\bar S':=(S_a| a\in\Theta')$, and
the respective restricted Lie algebra $H=\LL(\Theta',\bar S')=\Lie_p(\Theta_M)$.
Remark that, in general, $H$ is not a Lie subalgebra of $L$ because it is generated by
the virtual pivot elements of its zero generation  $V':=\{v_a| a\in\Theta_M\}$, which do no necessarily belong to $L$.

\begin{Lemma} \label{LgrowthLH}
There exist constants $q,q_1,A_0,A,N_0> 1$, determined in terms of the data for the generations $j=0,\ldots,M$ above,
such that the growth functions of $L$, $H$,
with respect to the pivot elements for $\Theta_0$ and $\Theta_0'=\Theta_M$, respectively, are related as:
\begin{enumerate}
\item
$ \gamma_{H}(V',[N/q])< \gamma_L(V,N) < A_0+A\cdot \gamma_{H}(V',[(N+q_1)/q])$ for $N\ge 1. $
\item $\gamma_L(V,N) < A_0+A\cdot \gamma_{H}(V',N)$ for $N\ge N_0. $
\end{enumerate}
\end{Lemma}
\begin{proof}
Below, the growth functions of $L$ and $H$ are considered with respect to $V$ and $V'$, respectively.
Denote $q:=\wt(\Theta_M)$.
We have a grading of the algebra $H=\mathop{\oplus}\limits_{n=1}^\infty H_n$ by degree in its generating set $V'$.
Consider a basis of $H$ formed by restricted Lie monomials in~$V'$.
Namely, choose bases for homogeneous components $H_n$, consisting of elements of type $[v_{a_1},\ldots,v_{a_m}]^{p^l}$,
where $mp^l=n$, $l\ge 0$, $v_{a_i}\in V'$, and $n\ge 1$.

Consider the actual pivot elements  $\{\bar v_a| a\in\Theta_M\}\subset L_{q}\subset L$.
By Lemma~\ref{Lactual_pivot} and Lemma~\eqref{LcalW}, we get
$$
[\bar v_{a_1},\ldots,\bar v_{a_m}]^{p^l}\equiv [v_{a_1},\ldots,v_{a_m}]^{p^l}
  \pmod {\sum_{d\in \Theta_{0\ldots M-1}} t_d^{(1)}\mathcal W_{\gen d+2}}.
$$
Since the sum above contains factors with senior variables, which are not acted by $v_a$, $a\in\Theta_M$,
the monomials above remain linearly independent.
These products have weight $n\wt(\Theta_M)=nq$. Hence, $\dim L_{qn}\ge \dim H_n$ for all $n\ge 1$.
Thus, we obtain the required lower bound:
$$\gamma_L(N)> \gamma_L(q[N/q])> \sum_{j=1}^{[N/q]}\dim L_{qj}\ge \sum_{j=1}^{[N/q]}\dim H_j=\gamma_H([N/q]),\qquad N\ge 1.$$

Consider the generators $v_1,\ldots,v_k$ of $L$ and  cut their presentation~\eqref{construction} as we did above~\eqref{generation0}:
\begin{equation}\label{generation02}
v_i= \sum_{i\succeq d\in \Theta_{0\ldots {M-1}} }
\Big(\prod_{c\vdash d} t_c^{(p^{S_c}-1)}\Big) \dd_d
+\sum_{i\succ d\in \Theta_M}
\Big(\prod_{c\vdash d} t_c^{(p^{S_c}-1)}\Big)v_d,\qquad i=1,\dots,k.
\end{equation}
Terms of the first sum belong to a bigger set of pure Lie monomials
$D:=\{\delta{=}\prod_{c> d} t_c^{(*)} \dd_d\mid d\in \Theta_{0\ldots {M-1}}\}$
while terms of the second sum belong to
$\Omega:=\{\omega=\prod_{c> d} t_c^{(*)} v_d\mid d\in \Theta_{M}\}$
(where (*) denote all possibilities).
Using this presentation, we write right-normed $N$-fold products of the generators~\eqref{generation02} as follows:
\begin{equation}\label{n-fold2}
w=[v_{i_1},\ldots,v_{i_N}]=\!\!\sum_{\delta_{*}\in D} [\delta_{j_1},\ldots,\delta_{j_N}]+\!\!
\sum_{1\le P\le N}
\Big[ [\delta_{*}, \ldots, \delta_{*},\omega_{i_1}],[\delta_{*}, \ldots, \delta_{*},\omega_{i_2}],\ldots,
 [\delta_{*}, \ldots, \delta_{*},\omega_{i_P}] \Big],
\end{equation}
where $\omega_*\in \Omega$, $\delta_*\in D$ are terms of~\eqref{generation02}. This formula is proved by induction on $N$.
Products of the first sum (up to scalar) belong to $D$, hence, their weight is bounded by $N_1:=\wt(\Theta_{M-1})$.
Let $N>N_1$ and the first sum disappears.
Write multiplicands in~\eqref{n-fold2} as
$[\delta_{*}, \ldots, \delta_{*},\omega_{i_j}]=\tilde \omega_{i_j}\in\Omega$ (omitting scalar for brevity).

Consider $\omega_i=\prod_{c> d_i} t_c^{(*)} v_{d_i}$, where $d_i\in\Theta_M$. We write their products as
\begin{align}\nonumber
\tilde w&=[\omega_{1},\omega_{2},\ldots,\omega_{P}]=
\Big[\prod\limits_{c> d_1} t_c^{(*)} v_{d_1},\prod\limits_{c> d_2} t_c^{(*)} v_{d_2},\ldots,
\prod\limits_{c> d_P} t_c^{(*)} v_{d_P}\Big]\\
&=\Big(\!\prod_{c\in\Theta_{0\ldots M-1}}\!\!\! \!\!t_c^{(\a_c)}\Big)[v_{d_1},v_{d_2},\ldots,v_{d_P}],
\qquad   0\le \a_c< p^{S_c},\quad 1\le P\le N,\quad d_{i}\in \Theta_M.
\label{razb2}
\end{align}
Thus, using~\eqref{n-fold2}, $N$-fold products  $w$  of the generators~\eqref{n-fold2} are expressed via monomials~\eqref{razb2}.
The absolute value of weight of divided power factors in~\eqref{razb2}
is bounded by $q_1:=\sum_{c \in \Theta_{0\ldots M-1} }(p^{S_c}{-}1)\wt(v_c)$.
Thus, we evaluate weight of a monomial~\eqref{razb2} as
$N =\wt w=\wt\tilde w\ge P\wt (\Theta_M)-q_1= Pq-q_1$. Hence,
\begin{align} 
P \le [(N+q_1)/q].
\label{nw2}
\end{align}
The next number bounds the number of different divided power factors in~\eqref{razb2}
\begin{equation}\label{contantC}
A:=\prod_{c\in \Theta_{0\ldots M-1}} p^{S_c}.
\end{equation}

We need to add $p^l$-powers of commutators of the generators.
In~\eqref{razb2}, the commutator is replaced by $[v_{d_1},v_{d_2},\ldots,v_{d_{P'}}]^{p^l}\in H_{P}$, where $P=P'p^l$,
the product of variables is replaced by its power,
which is  non-zero only in the case this product is absent.
Estimate~\eqref{nw2} remains valid.
Using~\eqref{razb2}, \eqref{nw2}, \eqref{contantC},  we get: 
\begin{align*}
\gamma_L(N)&\le \gamma_L(N_1)+ A\gamma_H([(N+q_1)/q])\le A_0 + A\gamma_H([(N+q_1)/q]),\qquad N\ge 1, 
\end{align*}
where we can set $A_0:=k^{N_1}$. By choosing $N_0$ such that $(N_0+q_1)/q\le N_0$ we obtain claim ii).
\end{proof}

\subsection{Almost wild and almost clover species of flies and their Lie algebras}
A specie of flies $\Theta$ is {\it almost wild} if after a component
$\Theta_M$, $|\Theta_M|\ge 3$, there is no selection,
in other words, $\Theta'=\cup_{j\ge M}\Theta_j$ is a wild specie generated by $\Theta_M$.
We show that initial generations are not essential and
the respective Lie algebra has a growth similar to that for wild flies
(see Theorem~\ref{Twild_growth} and constants $C_3$, $C_4$ in it),
with somewhat weaker bounds.

\begin{Lemma}\label{L_almost_wild}
Fix $\ch K=p> 0$, and $M\in\N$.
Let senior generations of flies $\Theta_{0\ldots M}$, where $|\Theta_M|=k'\ge 3$,
and tuple integers $\{S_a| a\in \Theta_{0\ldots M-1}\}$ are fixed,
and flies of generation $M$ are of the same weight denoted as $q=\wt(\Theta_M)$.
Assume that starting with generation $\Theta_M$ the specie is wild and $S_a=1$ for $a\in\Theta_j$, $j\ge M$.
We get an almost wild specie $\Theta$, a tuple $\bar S$, and the respective restricted Lie algebra  $L=\LL(\Theta,\bar S)$.
Denote $\lambda:=\log_2(p-\frac 12)$.
Fix any constant $\delta>C_4(p,k')/q$.
Then for any $\epsilon>0$ there exists $N(\epsilon)$ such that
$$
\exp\Big( \frac{n} {(\ln n)^{\lambda+\epsilon}} \Big)
\le \gamma_{L}(n)\le
\exp\Big(\delta \frac{n} {(\ln n)^{\lambda}} \Big),\qquad n\ge N(\epsilon).
$$
\end{Lemma}
\begin{proof}
Consider $H=\Lie_p(\Theta_M)$, the restricted Lie algebra for the wild specie $\Theta'=\cup_{j\ge M}\Theta_j$.
By Theorem~\ref{Twild_growth}, there exist positive constants $C_3$, $n_0$ such that
$$ \gamma_{H}(n)\ge \exp\Big(C_3 \frac{n} {(\ln n)^{\lambda}} \Big),\qquad n\ge n_0.$$
We apply Lemma~\ref{LgrowthLH}:
\begin{align*}
\gamma_{L}(n) &\ge \gamma_H([n/q])
\ge \exp\Big( C_3\frac{n/q-1} {(\ln n-\ln q)^{\lambda}} \Big)
= \exp\Big( \frac{n} {(\ln n)^{\lambda+\epsilon}}\cdot   \frac {C_3(\ln n)^\epsilon}q \frac{n-q}{n}
                        \Big(1-\frac {\ln q}{\ln n}\Big)^{-\lambda} \Big)\\
&\ge \exp\Big( \frac{n} {(\ln n)^{\lambda+\epsilon}}\Big),\qquad\quad n\ge N(\epsilon),
\end{align*}
by choosing $N(\epsilon)$ such that the second factor in brackets above is greater than 1.

By Theorem~\ref{Twild_growth}, there exist $C_4=C_4(p,k')$, $n_0$ such that
\begin{equation*}
\gamma_{H}(n)\le \exp\Big(C_4 \frac{n} {(\ln n)^{\lambda}} \Big),\qquad n\ge n_0.
\end{equation*}
We use Lemma~\ref{LgrowthLH}, the following asymptotic implies the desired upper bound
\begin{align*}
\gamma_L(n) &\le A_0+A\gamma_H\big([(n+q_1)/q]\big)
\le A_0+A  \exp\Big(C_4 \frac{(n+q_1)/q} {(\ln (n+q_1)-\ln q)^{\lambda}} \Big)\\
&=\exp\bigg(\Big  (\frac{C_4(p,k')}{q}+o(1)\Big) \frac{n} { (\ln n)^{\lambda}} \bigg)
\le \exp\Big(\delta \frac{n} {(\ln n)^{\lambda}} \Big),\qquad n\ge N(\epsilon).
\qedhere
\end{align*}
\end{proof}

Similarly, a specie of flies $\Theta$ is {\it almost clover} if there exists an integer $M$ such that
$\Theta'=\cup_{j\ge M}\Theta_j$ is a clover specie.
We show that initial generations are not essential and
the respective Lie algebra has a growth actually the same as that of the clover specie given by
Theorem~\ref{Tparam2} (one can also prove a similar statement in terms of the specie of Theorem~\ref{Tparam}).
We hope that using $q$ as the number of iterations of logarithm and parameter of Lemma~\ref{LgrowthLH} is not very misleading below.

\begin{Lemma}\label{L_almost_clover}
Fix $\ch K=p> 0$, $q\in\N$, $\kappa\in\R^+$, and $M\in\N$.
Let initial generations of flies $\Theta_{0\ldots M-1}$ and tuple integers $\{S_a| a\in \Theta_{0\ldots M-1}\}$ be fixed.
We assume that the flies of generation $M-1$ are of the same weight.
Then there exist an almost clover specie $\Theta$
(the clover generations and its tuple values start with the generation $M$), parameters $\bar S$, extending the initial values,
and an almost clover restricted Lie algebra $L=\LL(\Theta,\bar S)$ as follows.
\begin{enumerate}
\item 
For any $\epsilon>0$ there exists $N_0$ such that
$$ n \big(\ln^{(q)} \!n\big )^{\kappa-\epsilon} \le\gamma_L(n) \le n \big(\ln^{(q)} \!n\big )^{\kappa+\epsilon},\quad n\ge N_0;$$
\item $\GKdim L= \LGKdim L=1$;
\item $\Ldim^q L=\LLdim^q L=\kappa$.
\end{enumerate}
\end{Lemma}
\begin{proof}
By Theorem~\ref{Tparam2}, there exist a clover specie $\Theta'$, $|\Theta_0'|=3$, a tuple $\Xi_{q,\kappa}$ and
a 3-generated restricted Lie algebra $H=\LL(\Theta',\Xi_{q,\kappa})$ with the following growth function.
For any $\epsilon>0$ there exists $N$ such that
$$ n \big(\ln^{(q)} \!n\big )^{\kappa-\epsilon} \le\gamma_H(n) \le n \big(\ln^{(q)} \!n\big )^{\kappa+\epsilon}, \quad n\ge N.$$

We take the initial generations $\Theta_{0\ldots  M-1}$,
in the next generation select any 3 flies, thus $|\Theta_M|=3$,
and take the clover specie described above starting with $\Theta_M$.
Extending the initial tuple values by $\Xi_{q,\kappa}$ we get a tuple $\bar S$.
We obtain an almost clover restricted Lie algebra $L=\LL(\Theta,\bar S)$ and
apply bounds of Lemma~\ref{LgrowthLH}:
\begin{align*}
\gamma_{L}(n) &\le A_0+A\gamma_H(n)
\le A_0+A n \big(\ln^{(q)} \!n\big )^{\kappa+\epsilon}
\le n \big(\ln^{(q)} \!n\big )^{\kappa+2\epsilon},\quad n\ge n_1(\epsilon);\\
\gamma_{L}(n) &\ge \gamma_H([n/q])
\ge \frac nq  \big(\ln^{(q)}(n/q)\big )^{\kappa-\epsilon}
\ge n \big(\ln^{(q)} \!n\big )^{\kappa-2\epsilon}, \qquad\qquad n\ge n_2(\epsilon). \qedhere
\end{align*}
\end{proof}


\subsection{Phoenix Lie algebras}
Now we are ready to finish proof of the main result of the paper.

\begin{proof}[Proof of Theorem~\ref{T_main}] The nillity follows by Theorem~\ref{Tnillity}.

Let $\{\epsilon_i | i\ge 1\}$ be a decreasing sequence of positive numbers approaching to zero.
We shall construct a hybrid specie of flies $\Theta$, where $|\Theta_0|=3$, a uniform tuple $\bar S$, and
a strictly increasing sequence of integers $\{M_i|i\in\N\}$, such that
the segment $\Theta_{M_{m-1}},\ldots ,\Theta_{M_m-1}$ along with $S_{M_{m-1}},\ldots ,S_{M_m-1}$
are wild for odd $m$ and clover for even $m$, for all $m\ge 1$.
Above, we set $M_0=0$, also we start  all segments with 3 flies, namely, $|\Theta_{M_m}|=3$ for $m\ge 0$.
We obtain a hybrid restricted Lie algebra $\LL=\LL(\Theta,\bar S)$.
Simultaneously, we shall construct a sequence $\{n_i| i\in\N\}$, where $M_{j-1}\le n_j< M_j$, and such that
$\gamma_\LL(n_j)$ satisfies inequalities of item  i) of Theorem  with $\epsilon=\epsilon_j$, $\delta=\epsilon_j$ for odd $j$,
and inequalities of item ii) with $\epsilon=\epsilon_j$ for even $j$, for all $j\in\N$.
We say that at points $n_j$ we get {\it gate estimates}.
This construction yields claims i),~ii).
Intermediate constructed algebras will be denoted as $\LL$ for brevity.

Let $M_{j-1}$ be constructed, we shall increase $M_{j-1}$ below  if needed.  Consider that $j$ is odd.
The experimenter selects a wild segment $\Theta_{M_{j-1}},\ldots ,\Theta_{M_j-1}$,
starting with $M:=M_{j-1}$ and $|\Theta_{M_{j-1}}|=3$, and trivial entrees of the tuple (i.e. $S_i=1$ for all respective generations).
There exists $n_j$ such that the growth function satisfies bounds of Lemma~\ref{L_almost_wild} with $\epsilon:=\epsilon_j$  for all $n\ge n_j$.
In that Lemma we needed a constant $\delta>C_4(p,3)/\wt(\Theta_{M_{j-1}})$,
on the previous step we could select the denominator sufficiently large so that $\delta<\epsilon_j$.
Thus, $\gamma_\LL(n_j)$ satisfies the bounds of item i) with parameters $\delta=\epsilon:=\epsilon_j$.
Lemma~\ref{L_almost_wild} is based on estimates of Lemma~\ref{LgrowthLH},
where we need $m:=[(n_j+q_1)/q]$ wild generations, for some constants $q,q_1$.
Thus, we set $M_j:=\max\{n_j,M_{j-1}{+}m\}{+}1$.
The lower bound on the growth is based on estimate of Lemma~\ref{LSylov_1}, that needs a specie being wild till $n_j$.
Therefore, we can change generations and the tuple arbitrarily, starting with $M_j$,  without changing
the obtained value $\gamma_\LL(n_j)$ and the established gate estimates for it.

Let $M_{j-1}$ be constructed and $j$ is even.
Now the experimenter selects a clover segment $\Theta_{M_{j-1}},\ldots ,\Theta_{M_j-1}$,
and respective tuple entrees corresponding to the parameters $q$ and $\kappa$ (see details of the construction in~\cite{Pe20clover}).
There exists $n_j$, such that the growth function satisfies bounds of Lemma~\ref{L_almost_clover} for all $n\ge n_j$,
where we set $\epsilon:=\epsilon_j$.
Lemma~\ref{L_almost_clover} is based on Lemma~\ref{LgrowthLH},
where we need a clover segment of length $m:=[(n_j+q_1)/q]$, for some constants $q,q_1$.
Our computation of the growth function for a clover specie
is based on the explicit basis, see~\cite[Theorem 4.7]{Pe20clover}.   
Hence, to have the growth function the same as that of the clover specie till $m$ it is sufficient
to have $m$ new clover generations.
We set $M_j:=\max\{n_j,M_{j-1}{+}m\}{+}1$, and we can modify the specie and the tuple starting with $M_j$
while keeping the obtained gate estimates for $\gamma_\LL(n_j)$.

It remains to prove claim iii), which states that the growth function stays in the specified wide angles.
Let us prove the total lower bound in claim iii).
The problem is that at the beginning of clover generations $M_{j-1},\ldots, M_{j}-1$
the clover algebra $H$ used in the construction is rather "small".
First, we are going to check the total lower bound during the whole of a clover segment.
Since each generation of the specie has at least 3 flies,
we specify the construction so that we have a total clover subspecie
through all generations $\Theta'=\{a_n,b_n,c_n|n\ge 0\}\subset \Theta$. Consider its Lie algebra $H:=\LL(\Theta',\bar S')$,
where $\bar S'$ is the respective restriction tuple.
By Lemma~\ref{Lder_subspecies}, we have an epimorphism $\phi: \LL(\Theta,\bar S) \twoheadrightarrow H$
and it is sufficient to prove a total lower bound for $H$.

Let $j$ be even (i.e. the last segment is clover) and $M_{j-2}$ is constructed, we refine construction of integers $M_{j-1},M_j$.
Denote $m_2:=\wt(\Theta_{M_{j-2}})$, which is now fixed, and $m_1:=\wt(\Theta_{M_{j-1}})$.
Using~\eqref{recorrencia}, we have
\begin{equation}\label{m1m2}
m_1=m_2(2p-1)^{M_{j-1}-M_{j-2}}.
\end{equation}
A basis of $H$ contains standard monomials of second type (see~\cite{Pe20clover}):
\begin{equation}\label{prod3}
w=\prod_{\substack{M_{j-2}\le i< M_{j-1}\\ 0\le \a_i,\b_i,\gamma_i<p}}\!\! \underbrace{x_i^{(\a_i)} y_i^{(\b_i)}z_i^{(\gamma_i)}}_{\text{wild variables}}\cdot
\bigg(\!\!\!\!\!\prod_{\substack{M_{j-1}\le i\le n-2 \\ 0\le \a_i<p^{S_i}\!,\  0\le \b_i,\gamma_i<p }}\!\!
\underbrace{x_i^{(\a_i)} y_i^{(\b_i)}z_i^{(\gamma_i)}}_{\text{clover variables}}\!\bigg)
g_n^{\zeta_{n-1},\zeta_{n-1}},\quad n< M_j.
\end{equation}
The idea is that the first factor corresponds to a basis of the clover Lie algebra with the trivial tuple,
which Gelfand-Kirillov dimension is strictly greater than 1 (see~\cite{Pe20clover}) and we take this segment
sufficiently large to guarantee the lower bound.
We shall evaluate the number of some monomials~\eqref{prod3} with $\wt w\le m$, where $m\ge m_1$.

The second factors in~\eqref{prod3} are basis monomials of the clover specie constructed in Theorem~\ref{Tparam2}
generated by its zero generation
$\{a_{M_{j-1}},b_{M_{j-1}},c_{M_{j-1}}\}$, which weight is $\wt(M_{j-1})=m_1$.
Consider such monomials of weight bounded by $m$, but we need to normalize by weight of the zero generation.
Multiplying the function in Theorem~\ref{Tparam2} by a sufficiently small positive constant $C_0$,
we get a total lower bound, i.e. for all $n\ge 1$.
Substituting $m/m_1$ in that function,
we get a quasi-linear lower bound on the number of the second factors in~\eqref{prod3}
\begin{equation*}
C_0 \frac m{m_1}  \Big(\ln^{(q)} \frac m{m_1}\Big )^{\kappa-\epsilon},\qquad m\ge m_1.
\end{equation*}
(recall that if an iterated logarithm is not defined, negative, or less than 1, we redefine it to be 1).
Put $\delta:=\log_{2p-1}{p^3}$, then $\delta>1$. 
Using~\eqref{m1m2}, the number of first factors in~\eqref{prod3} has a "slow polynomial" growth:
\begin{equation*}
p^{3(M_{j-1}-M_{j-2})}=p^{3 \log_{2p-1} (m_1/m_2)}=\Big(\frac {m_1}{m_2}\Big)^{\log_{2p-1}{p^3}}
=\Big(\frac {m_1}{m_2}\Big)^{\delta}.
\end{equation*}
The first factors in~\eqref{prod3} are of negative weight and $w$ remains of weight bounded by $m$.
Thus, we get a lower bound on the number of monomials~\eqref{prod3}
\begin{align*}
\gamma_\LL(m)&\ge \gamma_H(m)\ge
\Big(\frac {m_1}{m_2}\Big)^{\delta}  C_0\frac{m}{m_1} \Big(\ln^{(q)} \frac m{m_1}\Big )^{\kappa-\epsilon}\\
&= \frac{C_0m_1^{(\delta-1)/2}}{m_2^{\delta }}\cdot m  \exp \Big(\frac{\delta{-}1}2\ln m_1+  (\kappa-\epsilon) \ln^{(q+1)} \frac m{m_1}\Big)\\
&\ge m\exp  (\ln^{(q)}m)^{\kappa-\epsilon},\qquad m\ge m_1.
\end{align*}
Indeed, since $m_2$ is fixed, by choosing $m_1$ sufficiently large, the first fraction is greater than 1.
Also, consider the formula inside the exponent as a function of $m_1$, while $m$ is fixed.
Using its derivation, we see that it is increasing for $m_1\in [1, m/\tau]$, where $\tau>1$.
Similarly, one checks the total lower bound during the whole of a wild segment, now the clover and wild segment change
their places in monomials~\eqref{prod3}, changing intervals of quasi-linear growth and slow polynomial growth.

Let us prove the total upper bound in claim iii).
Fix a large integer $j$, put $H:=\LL(\Theta_{M_j})$, this is a 3-generated restricted Lie algebra.
By Theorem~\ref{Twild_growth2},  there exist $C_4=C_4(p,3)$, $n_0$ such that
$$ \gamma_{H}(n)\le \exp\Big(C_4\frac{n} {(\ln n)^{\lambda}} \Big),\qquad n\ge n_0. $$
By Lemma~\ref{LgrowthLH}  there exist constants $A_0$, $A$, $q_1$, and $q:=\wt(\Theta_{M_j})$  such that
$$\gamma_\LL(N) < A_0+A\cdot \gamma_{H}([(N+q_1)/q])=\exp\Big(\Big(\frac{C_4}q+o(1)\Big)\frac{N} {(\ln N)^{\lambda}} \Big) ,\qquad N\to\infty.$$
By choosing $j$ sufficiently large we get $C_4/q<\delta$, thus yielding the upper bound of claim~iii).
\end{proof}


\end{document}